\documentclass{amsart}

\usepackage{amssymb, amscd, amsthm, mathrsfs, bm}

\usepackage{color}
\usepackage[colorlinks=true]{hyperref}

\newtheorem{Thm}{Theorem}[section]
\newtheorem{Prop}[Thm]{Proposition}
\newtheorem{Cor}[Thm]{Corollary}

\newtheorem{Def}[Thm]{Definition}

\theoremstyle{remark}

\newtheorem*{Ack}{Acknowledgments}

\numberwithin{equation}{section}

\newcommand{\Order}{\mathcal{O}}
\newcommand{\into}{\hookrightarrow}

\newcommand{\onto}{\twoheadrightarrow}
\newcommand{\isomto}{\overset{\sim}{\to}}

\newcommand{\compose}{\mathbin{\circ}}
\newcommand{\tensor}{\mathbin{\otimes}}
\newcommand{\closure}[1]{\overline{#1}}
\newcommand{\N}{\mathbb{N}}
\newcommand{\Z}{\mathbb{Z}}
\newcommand{\Q}{\mathbb{Q}}

\newcommand{\F}{\mathbb{F}}
\newcommand{\et}{\mathrm{et}}
\newcommand{\Et}{\mathrm{Et}}
\newcommand{\fppf}{\mathrm{fppf}}

\newcommand{\Gm}{\mathbf{G}_{m}}
\newcommand{\Ga}{\mathbf{G}_{a}}
\newcommand{\tor}{\mathrm{tor}}
\newcommand{\divis}{\mathrm{div}}
\newcommand{\dirlim}{\varinjlim}
\newcommand{\invlim}{\varprojlim}

\mathchardef\mhyphen="2D

\newcommand{\rat}{\mathrm{rat}}
\newcommand{\ind}{\mathrm{ind}}
\newcommand{\pro}{\mathrm{pro}}
\newcommand{\perf}{\mathrm{perf}}
\newcommand{\perar}{\mathrm{perar}}

\newcommand{\Alg}{\mathrm{Alg}}
\newcommand{\Ind}{\mathrm{I}}
\newcommand{\Pro}{\mathrm{P}}

\newcommand{\alg}[1]{\mathbf{#1}}
\newcommand{\var}{\;\cdot\;}
\newcommand{\Ch}{\mathrm{Ch}}
\newcommand{\ideal}[1]{\mathfrak{#1}}
\newcommand{\ctensor}{\mathbin{\Hat{\otimes}}}

\newcommand{\algebrize}{\Acute{\mathsf{a}}}

\newcommand{\varalg}[1]{\bm{#1}}

\newcommand{\genby}[1]{\langle #1 \rangle}

\newcommand{\RP}{\mathrm{RP}}
\newcommand{\RPS}{\mathrm{RPS}}
\newcommand{\cs}{\mathrm{cs}}
\newcommand{\Fin}{\mathrm{Fin}}

\newcommand{\sAb}{\mathrm{sAb}}
\DeclareMathOperator{\Gal}{Gal}
\DeclareMathOperator{\Hom}{Hom}

\DeclareMathOperator{\Ker}{Ker}

\DeclareMathOperator{\Spec}{Spec}
\DeclareMathOperator{\Ab}{Ab}

\DeclareMathOperator{\Pic}{Pic}

\let\Im\relax

\DeclareMathOperator{\Im}{Im}
\DeclareMathOperator{\sheafhom}{\alg{Hom}}
\DeclareMathOperator{\sheafext}{\alg{Ext}}

\DeclareMathOperator{\gr}{gr}

\DeclareMathOperator{\Br}{Br}

\DeclareMathOperator{\Weil}{\mathfrak{R}}

\title[Class field theory, Hasse principles, Pic-Br]
	{Class field theory, Hasse principles and Picard-Brauer duality for two-dimensional local rings}
\author{Takashi Suzuki}
\address{
	Department of Mathematics, Chuo University,
	1-13-27 Kasuga, Bunkyo-ku, Tokyo 112-8551, Japan
}
\email{tsuzuki@gug.math.chuo-u.ac.jp}
\date{August 1, 2023}

\begin{document}

\begin{abstract}
	We draw concrete consequences from our arithmetic duality
	for two-dimensional local rings with perfect residue field.
	These consequences include class field theory, Hasse principles for coverings and $K_{2}$
	and a duality between divisor class groups and Brauer groups.
	To obtain these,
	we analyze the ind-pro-algebraic group structures on arithmetic cohomology obtained earlier
	and prove some finiteness properties about them.
\end{abstract}

\maketitle

\tableofcontents


\section{Introduction}


\subsection{Aim of the paper}

Let $A$ be a two-dimensional normal complete noetherian local ring
of mixed characteristic $(0, p)$ with perfect residue field $F$.
In \cite{Sai86, Sai87},
Saito gives arithmetic duality and class field theory for $A$ when $F$ is finite.
We refined this theory in our recent work \cite{Suz24} when $F$ is a general perfect field,
by giving ind-pro-algebraic group structures for arithmetic cohomology attached $A$
and taking these structures into account.
However, our duality was on the derived category level
and we could not deduce concrete consequences
due to lack of precise understanding of the structure of each cohomology object.

In this paper, we will prove some finiteness (or admissibility) properties of
our ind-pro-algebraic group structures.
This allows us to deduce concrete duality statements
for each cohomology object.
The statements we deduce include class field theory,
Hasse principles for coverings and $K_{2}$
and a duality between divisor class groups and Brauer groups.
The key inputs are the relatively perfect unipotent group structures
on $p$-adic nearby cycle sheaves \cite{BS20},
the finite generation of the pro-$p$ fundamental group of the punctured spectrum \cite{Suz23}
and Lipman's group scheme structure for the divisor class group \cite{Lip69, Lip76}.

Along the way, we also prove a duality for relatively perfect unipotent groups (especially, wound ones)
over fields with finite $p$-bases,
which may be of independent interest.
When $F$ is finite, we explain how to recover Saito's duality
and its local counterpart, namely Kato's two-dimensional local class field theory \cite{Kat79}.
Our duality actually gives locally compact (and, in particular, Hausdorff!)\ topologies for
the groups appearing in these theories.


\subsection{Main results}

We formulate our results.
Let $F$ be a perfect field of characteristic $p > 0$.
Let $\Alg / F$ be the category of perfections (inverse limits along Frobenius)
of commutative algebraic groups over $F$
with group scheme morphisms (\cite{Ser60}).
Let $\Alg_{u} / F \subset \Alg / F$ be the full subcategory
of perfections of unipotent groups.
Let $\Pro \Alg_{u} / F$ and $\Ind \Alg_{u} / F$ be its pro-category and ind-category, respectively.
Let $\Ind \Pro \Alg_{u} / F$ be the ind-category of $\Pro \Alg_{u} / F$.
For $G \in \Alg / F$, denote its identity component by $G^{0}$
and set $\pi_{0} G = G / G^{0}$.
The functors $G \mapsto G^{0}, \pi_{0} G$ naturally extend to $\Ind \Pro \Alg_{u} / F$.
An object $G \in \Ind \Pro \Alg_{u} / F$ is said to be connected
if $\pi_{0} G = 0$.
As in \cite[Definition 3.1.1]{Suz24},
we define the following class of objects of $\Ind \Pro \Alg_{u} / F$
with nice finiteness (or admissibility) properties:

\begin{Def}
	Define $\mathcal{W}_{F} \subset \Ind \Pro \Alg_{u} / F$ to be
	the full subcategory of objects $G$
	admitting a filtration $G \supset G^{0} \supset G' \supset 0$ such that:
	\begin{itemize}
		\item
			$G / G^{0} = \pi_{0} G$ is finite \'etale ($p$-primary),
		\item
			$G^{0} / G'$ can be written as a direct limit $\dirlim_{n \ge 1} G''_{n}$,
			where each $G''_{n} \in \Alg_{u} / F$ is connected
			and each transition morphism $G''_{n} \to G''_{n + 1}$ is injective,
		\item
			$G'$ can be written as an inverse limit $\invlim_{n \ge 1} G'_{n}$,
			where each $G'_{n} \in \Alg_{u} / F$ is connected
			and each transition morphism $G'_{n + 1} \to G'_{n}$ is surjective
			with connected kernel.
	\end{itemize}
\end{Def}

What is nice about $\mathcal{W}_{F}$ is that
the full subcategory of $\mathcal{W}_{F}$ of connected groups
has a canonical contravariant autoequivalence (\cite[Proposition 3.1.7]{Suz24})
that maps the perfection of the additive group to itself.
We call it Serre duality for connected groups in $\mathcal{W}_{F}$.

Let $A$ be a two-dimensional normal complete noetherian local ring of mixed characteristic
with maximal ideal $\ideal{m}$ whose residue field is the above $F$.
Let $K$ be its fraction field.
Let $P$ be the set of height one prime ideals of $A$.
Set $X = \Spec A \setminus \{\ideal{m}\}$.
Let $U$ be a dense open subscheme of $X$.
Let $\Z / p^{n} \Z(r) \in D(U_{\et})$ be the Bloch cycle complex mod $p^{n}$
in the \'etale topology if $r \ge 0$
and let $\Z / p^{n} \Z(r) \in D(U_{\et})$ be the extension-by-zero of the usual Tate twist $\Z / p^{n} \Z(r)$
on $U \cap \Spec A[1 / p]$ if $r < 0$.
Denote the \'etale cohomology functor by $H^{q}(U, \var)$
and let $H^{q}_{c}(U, \var) = H^{q}(X, j_{!} \var)$,
where $j \colon U \into X$ is the inclusion.
By \cite[Theorem 1.3.1]{Suz24},
we have canonical objects
$\alg{H}^{q}(\alg{U}, \Z / p^{n} \Z(r))$
and $\alg{H}^{q}_{c}(\alg{U}, \Z / p^{n} \Z(r))$ of $\Ind \Pro \Alg_{u} / F$
for each integers $n \ge 1$, $q$ and $r$,
which are algebraic structures on $H^{q}(U, \Z / p^{n} \Z(r))$ and $H^{q}_{c}(U, \Z / p^{n} \Z(r))$.
Note that when $U = X$, there is no difference between
$\alg{H}^{q}(\alg{U}, \Z / p^{n} \Z(r))$
and $\alg{H}^{q}_{c}(\alg{U}, \Z / p^{n} \Z(r))$.
Here is the structure result for these groups:

\begin{Thm} \label{0077}
	We have $\alg{H}^{q}(\alg{U}, \Z / p^{n} \Z(r)),
	\alg{H}^{q}_{c}(\alg{U}, \Z / p^{n} \Z(r)) \in \mathcal{W}_{F}$.
	They are zero unless $0 \le q \le 3$.
	Their structures are described as follows.
	\begin{enumerate}
		\item
			Assume $U = X$.
			Then:
				\[
					\begin{array}{c|ccc}
							\alg{H}^{q}(\alg{X}, \Z / p^{n} \Z(r))
						&
							r = 0
						&
							r = 1
						&
							r = 2
						\\ \hline
							q = 0
						&
							\Z / p^{n} \Z
						&
							\text{finite}
						&
							\text{finite}
						\\
							q = 1
						&
							\text{finite}
						&
							\text{pro-alg}
						&
							\text{pro-alg}
						\\
							q = 2
						&
							\text{ind-alg}
						&
							\text{alg}
						&
							\text{pro-alg}
						\\
							q = 3
						&
							\text{ind-alg}
						&
							\text{ind-alg}
						&
							\Z / p^{n} \Z
					\end{array}
				\]
			Here finite means finite \'etale,
			alg means in $\Alg_{u} / F$,
			pro-alg means in $\Pro \Alg_{u} / F$
			and ind-alg means in $\Ind \Alg_{u} / F$.
		\item
			Assume $U \subset \Spec A[1 / p]$.
			Then:
				\[
					\begin{array}{c|c}
							\alg{H}^{q}(\alg{U}, \Z / p^{n} \Z(r))
						&
							\forall r
						\\ \hline
							q = 0
						&
							\text{finite}
						\\
							q = 1
						&
							\text{pro-alg}
						\\
							q = 2
						&
							\text{pro-alg}
						\\
							q = 3
						&
							0
					\end{array}
					\qquad
					\begin{array}{c|c}
							\alg{H}^{q}_{c}(\alg{U}, \Z / p^{n} \Z(r))
						&
							\forall r
						\\ \hline
							q = 0
						&
							0
						\\
							q = 1
						&
							\text{finite}
						\\
							q = 2
						&
							\text{ind-alg}
						\\
							q = 3
						&
							\text{ind-alg}
					\end{array}
				\]
		\item
			In general:
				\[
					\begin{array}{c|ccc}
							\alg{H}^{q}(\alg{U}, \Z / p^{n} \Z(r))
						&
							r \le 0
						&
							r = 1
						&
							r \ge 2
						\\ \hline
							q = 0
						&
							\text{finite}
						&
							\text{finite}
						&
							\text{finite}
						\\
							q = 1
						&
							\text{pro-alg}
						&
							\text{pro-alg}
						&
							\text{pro-alg}
						\\
							q = 2
						&
							\text{general}
						&
							\text{pro-alg}
						&
							\text{pro-alg}
						\\
							q = 3
						&
							\text{ind-alg}
						&
							\text{ind-alg}
						&
							\text{finite}
					\end{array}
				\]
				\[
					\begin{array}{c|ccc}
							\alg{H}^{q}_{c}(\alg{U}, \Z / p^{n} \Z(r))
						&
							r \le 0
						&
							r = 1
						&
							r \ge 2
						\\ \hline
							q = 0
						&
							\text{finite}
						&
							\text{finite}
						&
							\text{finite}
						\\
							q = 1
						&
							\text{finite}
						&
							\text{pro-alg}
						&
							\text{pro-alg}
						\\
							q = 2
						&
							\text{ind-alg}
						&
							\text{ind-alg}
						&
							\text{general}
						\\
							q = 3
						&
							\text{ind-alg}
						&
							\text{ind-alg}
						&
							\text{ind-alg}
					\end{array}
				\]
			Here general means a general object of $\mathcal{W}_{F}$.
	\end{enumerate}
\end{Thm}

Note that the apparently missing group
$\alg{H}^{q}(\alg{X}, \Z / p^{n} \Z(r))$ for $r < 0$ (resp.\ $r > 2$) is isomorphic to
$\alg{H}^{q}_{c}(\alg{U}, \Z / p^{n} \Z(r))$
(resp.\ $\alg{H}^{q}(\alg{U}, \Z / p^{n} \Z(r))$)
with $U = \Spec A[1 / p]$,
so it is also covered in the tables.

In particular, the group $\alg{H}^{1}(\alg{X}, \Z / p^{n} \Z)$ is finite \'etale.
We will first prove that this group is in $\Alg_{u} / F$.
That it is finite \'etale then follows from the result of \cite{Suz23}.
This implies the above structure result $\alg{H}^{3}(\alg{X}, \Z / p^{n} \Z(2)) \cong \Z / p^{n} \Z$
by duality (Theorem \ref{0088} below).
The group $H^{3}(X, \Z / p^{n} \Z(2))$ is the group ``$SK_{1}(X)$''
(\cite{Blo81}, \cite{Sai86}) mod $p^{n}$ when $F$ is algebraically closed.
Hence we obtain:

\begin{Cor}
	Assume $F$ is algebraically closed.
	Then the sequence
		\[
				K_{2}(K) / p^{n} K_{2}(K)
			\to
				\bigoplus_{\ideal{p} \in P}
					\kappa(\ideal{p})^{\times} / \kappa(\ideal{p})^{\times p^{n}}
			\to
				\Z / p^{n} \Z
			\to
				0
		\]
	is exact,
	where the first map is the tame symbols to the residue fields at $\ideal{p}$
	and the second map is the sum of the normalized valuation maps.
\end{Cor}

By Levine \cite[Theorem 2.1]{Lev88},
the group $SK_{1}(X)$ is isomorphic to
the Grothendieck group of the category of $A$-modules of finite length and finite projective dimension.
The above corollary says that the kernel of the ``length'' map
$SK_{1}(X) \to \Z$ studied by Srinivas \cite{Sri87}, \cite[Chapter 9]{Sri96} is $p$-divisible.
It is relatively easy to prove that $\Ker(SK_{1}(X) \to \Z)$ is $l$-divisible for any prime $l \ne p$.%
\footnote{
	Use the fact $H^{3}(X, \Z / l \Z(2)) \cong H^{4}_{\ideal{m}}(A, \Z / l \Z(2)) \cong \Z / l \Z$
	(\cite[Expos\'e XVII, Corollaire 3.4.1.6]{ILO14})
	in place of our result $H^{3}(X, \Z / p \Z(2)) \cong \Z / p \Z$ above.
}
Therefore this kernel is divisible.
For an equal characteristic $A$ (normal and two-dimensional),
this divisibility is proved
in Srinivas's unpublished notes \cite{Sri85} by a global method.

The part $\alg{H}^{q}(\alg{U}, \Z / p^{n} \Z(r)),
\alg{H}^{q}_{c}(\alg{U}, \Z / p^{n} \Z(r)) \in \mathcal{W}_{F}$
of Theorem \ref{0077} allows us to deduce concrete duality statements
from the abstract duality statement of \cite{Suz24}.
The result is:

\begin{Thm} \label{0088}
	We have a Serre duality
		\[
				\alg{H}^{q}(\alg{U}, \Z / p^{n} \Z(r))^{0}
			\leftrightarrow
				\alg{H}^{4 - q}_{c}(\alg{U}, \Z / p^{n} \Z(2 - r))^{0}
		\]
	of connected groups in $\mathcal{W}_{F}$
	and a Pontryagin duality
		\[
				\pi_{0}
				\alg{H}^{q}(\alg{U}, \Z / p^{n} \Z(r))
			\leftrightarrow
				\pi_{0}
				\alg{H}^{3 - q}_{c}(\alg{U}, \Z / p^{n} \Z(2 - r))
		\]
	of finite \'etale groups over $F$.
\end{Thm}

Theorems \ref{0077} and \ref{0088} will be proved together in Section \ref{0115}.

As a particular example, we have a Pontryagin duality
	\[
			\alg{H}^{1}(\alg{X}, \Z / p^{n} \Z)
		\leftrightarrow
			\pi_{0} \alg{H}^{2}(\alg{X}, \Z / p^{n} \Z(2)),
	\]
which is unramified class field theory in this context.
The Serre duality
	\[
			\alg{H}^{1}(\alg{U}, \Z / p^{n} \Z)^{0}
		\leftrightarrow
			\alg{H}^{3}_{c}(\alg{U}, \Z / p^{n} \Z(2))^{0}
	\]
should look more similar to Saito's class field theory \cite{Sai87},
since $H^{3}_{c}(U, \Z / p^{n} \Z(2))$ is the $K_{2}$-id\`ele class group of $U$ mod $p^{n}$
when $F$ is algebraically closed.

We now state Hasse principles in this context.
Assume for the moment that $F$ is algebraically closed.
Let $H^{1}_{\cs}(X, \Z / p^{n} \Z) \subset H^{1}(X, \Z / p^{n} \Z)$
be the subgroup consisting of coverings of $X$ completely split at all closed points of $X$.
Let
	\[
			\pi_{0} \alg{H}^{2}(\alg{K}, \Z / p^{n} \Z(2))
		:=
			\dirlim_{U}
				\pi_{0} \alg{H}^{2}(\alg{U}, \Z / p^{n} \Z(2)),
	\]
where the direct limit is over all dense open subschemes $U \subset X$.
For each $\ideal{p} \in P$,
let $\Hat{K}_{\ideal{p}}$ be the fraction field
of the completion of the localization of $A$ at $\ideal{p}$.
We have a natural map
	\[
			\pi_{0} \alg{H}^{2}(\alg{K}, \Z / p^{n} \Z(2))
		\to
			\bigoplus_{\ideal{p} \in P}
				\pi_{0} \alg{H}^{2}(\Hat{\alg{K}}_{\ideal{p}}, \Z / p^{n} \Z(2))
		\cong
			\bigoplus_{\ideal{p} \in P}
				\Z / p^{n} \Z,
	\]
where $\alg{H}^{2}(\Hat{\alg{K}}_{\ideal{p}}, \Z / p^{n} \Z(2))$ is the algebraic structure on
the Galois cohomology $H^{2}(\Hat{K}_{\ideal{p}}, \Z / p^{n} \Z(2))$
(\cite[Theorem 6.4.2 (1)]{Suz24}).

\begin{Thm} \label{0109}
	Assume $F$ is algebraically closed.
	\begin{enumerate}
		\item
			The sequence
				\[
						\pi_{0} \alg{H}^{2}(\alg{K}, \Z / p^{n} \Z(2))
					\to
						\bigoplus_{\ideal{p} \in P}
							\Z / p^{n} \Z
					\stackrel{\mathrm{sum}}{\to}
						\Z / p^{n} \Z
					\to
						0
				\]
			is exact and the kernel of the first map is Pontryagin dual to
			$H^{1}_{\cs}(X, \Z / p^{n} \Z)$.
		\item
			Let $\mathfrak{X} \to \Spec A$ be a resolution of singularities
			such that $\mathfrak{X} \times_{A} A / p A \subset \mathfrak{X}$ is
			supported on a strict normal crossing divisor
			(\cite[Tag 0BIC]{Sta22}).
			Let $Y$ be the reduced part of $\mathfrak{X} \times_{A} F$.
			Then $H^{1}_{\cs}(X, \Z / p^{n} \Z) \cong H^{1}(Y, \Z / p^{n} \Z)$.
		\item
			Let $Y_{1}, \dots, Y_{m}$ be the irreducible components of $Y$.
			Let $\Gamma$ be the dual graph of $Y$.
			Then we have an exact sequence
				\[
						0
					\to
						H^{1}(\Gamma, \Z / p^{n} \Z)
					\to
						H^{1}(Y, \Z / p^{n} \Z)
					\to
						\bigoplus_{i}
							H^{1}(Y_{i}, \Z / p^{n} \Z)
					\to
						0.
				\]
	\end{enumerate}
\end{Thm}

Saito gives this type of Hasse principle for finite $F$ in \cite{Sai87}.
We will actually prove the above theorem for not necessarily algebraically closed $F$,
where the Galois actions will be taken into account.
The finite \'etale property of $\alg{H}^{1}(\alg{X}, \Z / p^{n} \Z)$
is crucial in deducing Theorem \ref{0109} from Theorem \ref{0088}.

We return to not necessarily algebraically closed $F$
to state a duality between the divisor class group and the Brauer group,
which is a refinement of Saito's duality \cite{Sai86}.
Let $\alg{Pic}_{X} \in \Alg / F$ be the perfection of Lipman's group scheme structure on $\Pic(X)$
with identity component $\alg{Pic}^{0}_{X}$.
Let $\alg{Pic}_{X, \sAb} \subset \alg{Pic}_{X}$ be the maximal semi-abelian part
and set $\alg{Pic}^{0}_{X} / \sAb = \alg{Pic}_{X}^{0} / \alg{Pic}_{X, \sAb}$.
Define
	\[
			\alg{Br}_{X}[p^{\infty}]
		:=
			\dirlim_{n \ge 1}
				\alg{H}^{2}(\alg{X}, \Z / p^{n} \Z(1)).
	\]
Let $\closure{A}$ be the completed unramified extension of $A$ corresponding to $\closure{F}$
and set $\closure{X} = X \times_{A} \closure{A}$.

\begin{Thm} \label{0083} \mbox{}
	\begin{enumerate}
		\item
			The group of $\closure{F}$-points of $\alg{Br}_{X}[p^{\infty}]$ is $\Br(\closure{X})[p^{\infty}]$.
		\item
			We have $\alg{Br}_{X}[p^{\infty}]^{0} \in \Alg_{u} / F$
			and the object $\pi_{0}(\alg{Br}_{X}[p^{\infty}])$ is of cofinite type.
		\item
			We have a Serre duality
				\[
						\alg{Pic}^{0}_{X} / \sAb
					\leftrightarrow
						\alg{Br}_{X}[p^{\infty}]^{0}
				\]
			of connected groups in $\Alg_{u} / F$ and a Pontryagin duality
				\[
						T_{p} \alg{Pic}^{0}_{X, \sAb}
					\leftrightarrow
						\pi_{0}(\alg{Br}_{X}[p^{\infty}]),
				\]
			where $T_{p}$ denotes the $p$-adic Tate module.
	\end{enumerate}
\end{Thm}

In particular, the dimension of $\alg{Br}_{X}[p^{\infty}]$ is equal to
the dimension of the unipotent part of $\alg{Pic}_{X}$.

We give consequences for finite $F$.
In this case,
by \cite[Proposition (10.3) (b)]{Suz20},
the group of $F$-valued points of any object of (the bounded derived category of) $\Ind \Pro \Alg_{u} / F$
can be equipped with a canonical structure as an ind-profinite group.
Applying this process to (the derived categorical versions of)
$\alg{H}^{q}(\alg{U}, \Z / p^{n} \Z(r))$
and $\alg{H}^{q}_{c}(\alg{U}, \Z / p^{n} \Z(r))$,
we obtain ind-profinite group structures on
$H^{q}(U, \Z / p^{n} \Z(r))$ and $H^{q}_{c}(U, \Z / p^{n} \Z(r))$.

\begin{Thm} \label{0112}
	Assume $F$ is finite.
	\begin{enumerate}
		\item
			The groups $H^{q}(U, \Z / p^{n} \Z(r))$ and $H^{q}_{c}(U, \Z / p^{n} \Z(r))$
			are locally compact (and, in particular, Hausdorff).
		\item
			We have a Pontryagin duality
				\[
						H^{q}(U, \Z / p^{n} \Z(r))
					\leftrightarrow
						H^{4 - q}_{c}(U, \Z / p^{n} \Z(2 - r)).
				\]
	\end{enumerate}
\end{Thm}

We also prove a variant of this for local fields $\Hat{K}_{\ideal{p}}$
at height one primes $\ideal{p} \in P$.
This gives a locally compact group structure on
$H^{q}(\Hat{K}_{\ideal{p}}, \Z / p^{n} \Z(r))$ and shows its duality.
The Hausdorff property of
$H^{2}(\Hat{K}_{\ideal{p}}, \Z / p^{n} \Z(2)) \cong K_{2}(\Hat{K}_{\ideal{p}}) / p^{n} K_{2}(\Hat{K}_{\ideal{p}})$
is claimed in \cite[Section 4.3]{Fes01}
(see also \cite[Theorem 4.7]{Fes01} and \cite[Section 6.6, Theorem 3]{Fes00}).
But these references omit the relevant calculations of Vostokov's pairing,
and this pairing is even undefined for the case $p = 2$.
We do not use the results of \cite{Fes00, Fes01}.


\subsection{Outline of proof}

To prove Theorems \ref{0077} and \ref{0088},
we first need to prove corresponding statements for
the local field $\Hat{K}_{\ideal{p}}$ (and its ring of integers) at each $\ideal{p} \in P$.
We have nothing to add to \cite[Section 7]{Suz24} if $\ideal{p}$ does not divide $p$.
For $\ideal{p}$ dividing $p$,
what we need is already given in
\cite[Proposition 6.4.1 and the proof of Proposition 6.5.2]{Suz24}
if $n = 1$.
The same proof does not work for $n \ge 2$
since the graded pieces of $p$-adic nearby cycles $R^{q} \Psi \Z / p^{n} \Z(r)$ are not completely determined in this case.
A simple d\'evissage on $n$ does not work either
since $\mathcal{W}_{F}$ is not an abelian category.
What we know is that $R^{q} \Psi \Z / p^{n} \Z(r)$ has a structure
as a relatively perfect unipotent group over $\kappa(\ideal{p})$ (\cite{BS20}).
For a relatively perfect unipotent group $G$, in Section \ref{0007},
we show that $H^{s}(\kappa(\ideal{p}), G)$ for each $s \in \Z$ has a structure
as an object of $\mathcal{W}_{F}$.
To prove this statement for $s = 0$,
we use N\'eron models of wound unipotent groups and their Greenberg transforms.
The statement for $s = 1$ is reduced to the statement for $s = 0$
by the duality for relatively perfect unipotent groups given in Sections \ref{0113} and \ref{0007}.
Based on this result, we prove the desired result for $\Hat{K}_{\ideal{p}}$
in Section \ref{0070}.

We then prove that $\alg{H}^{q}(\alg{U}, \Z / p^{n} \Z(r)) \in \Pro \Alg_{u} / F$
and $\alg{H}^{q}_{c}(\alg{U}, \Z / p^{n} \Z(r)) \in \Ind \Alg_{u} / F$
for $U \subset \Spec A[1 / p]$
in Section \ref{0086}.
These statements do not immediately imply that these groups are in $\mathcal{W}_{F}$,
but they are the best we can have at this point.
To prove these statements,
we show the corresponding statements for
curves, tubular neighborhoods and two-dimensional regular local rings
in Sections \ref{0059} and \ref{0086}.
Applying them to a resolution of singularities of $A$,
we get the desired statements for $U$.

Now the above results for $\Hat{K}_{\ideal{p}}$ and $U \subset \Spec A[1 / p]$
are sufficient, in Section \ref{0115}, to deduce Theorems \ref{0077} and \ref{0088} for general $U$
by somewhat heavy homological algebra in $\Ind \Pro \Alg_{u} / F$.
The lemmas needed for this homological algebra are provided in Section \ref{0114}.

This being done, it is relatively straightforward
to deduce Theorem \ref{0109} in Section \ref{0116}.
The proof of Theorem \ref{0083}, on the other hand, is a bit more complicated
since we need to deal with the non-torsion coefficient sheaf $\Gm$
and direct/inverse limits in $n$.
Sections \ref{0117} and \ref{0118} are devoted for this.

Finally, in Section \ref{0119},
we deduce consequences for the case of finite $F$.
In this case, for any connected group $G$ in $\mathcal{W}_{F}$,
the group of rational points $G(F)$ has a canonical structure
as a locally compact group.
We use the pro-\'etale site of a point $\ast_{\pro\et}$ for showing this,
just as we did in \cite[Section 10]{Suz20} and \cite[Section 4.2]{SuzCurve}.
This proves Theorem \ref{0112}.

\begin{Ack}
	The author would like to thank Kazuya Kato for his suggestion of this project,
	Ivan Fesenko for answering questions about Vostokov's pairing
	and Vasudevan Srinivas for sharing his notes \cite{Sri85}
	(now available from his webpage) with the author.
\end{Ack}


\section{Notation and homological algebra of ind-pro-algebraic groups}
\label{0114}

Throughout this paper, let $F$ be a perfect field of characteristic $p > 0$.
We generally follow the notation of \cite{Suz24},
especially \cite[Section 2]{Suz24}.
We first briefly recall some of the notation there.
We then give several facts on homological algebra of ind-pro-algebraic groups.
They will be used later on to show that
the ind-pro-algebraic groups at hand belong to the category $\mathcal{W}_{F}$.

All sites in this paper are defined by given pretopologies.
For a site $S$, let $\Ab(S)$ be the category of sheaves of abelian groups on $S$,
$\Ch(S)$ the category of complexes in $\Ab(S)$ (in the cohomological grading)
and $D(S)$ its derived category.
Let
	\[
			\Lambda_{n}
		=
			\Z / p^{n} \Z,
		\quad
			\Lambda_{\infty}
		=
			\Q_{p} / \Z_{p},
		\quad
			\Lambda
		=
			\Lambda_{1}
		=
			\Z / p \Z.
	\]
Denote $\tensor = \tensor_{\Z}$.
Let $\sheafhom_{S}$ be the sheaf-Hom functor for $\Ab(S)$
with right derived functors $\sheafext_{S}^{q}$ and $R \sheafhom_{S}$.
For objects $G, H, K \in D(S)$,
we say that a morphism $G \tensor^{L} H \to K$ is a perfect pairing
if $G \to R \sheafhom_{S}(H, K)$ and $H \to R \sheafhom_{S}(G, K)$ are
both isomorphisms.
A premorphism of sites $f \colon S' \to S$ is a functor $f^{-1}$
from the underlying category of $S$ to the underlying category of $S'$
that sends covering families to covering families
such that $f^{-1}(Y \times_{X} Z) \isomto f^{-1} Y \times_{f^{-1} X} f^{-1} Z$
whenever $Y \to X$ appears in a covering family.
Let $f^{\ast} \colon \Ab(S) \to \Ab(S')$ be the pullback functor for sheaves of abelian groups
and $L f^{\ast} \colon D(S) \to D(S')$ its left derived functor.
We say that an object $G \in D(S)$ is $f$-acyclic
if $G \isomto R f_{\ast} L f^{\ast} G$.

We recall the ind-rational pro-\'etale site $\Spec F^{\ind\rat}_{\pro\et}$ from \cite{Suz20}.
An $F$-algebra is said to be rational if it is a finite direct product of
perfections (direct limits along Frobenius) of finitely generated fields over $F$.
An ind-rational $F$-algebra is a direct limit of a filtered direct system consisting of
rational $F$-algebras.
Let $F^{\ind\rat}$ be the category of ind-rational $F$-algebras
with $F$-algebra homomorphisms.
Then $\Spec F^{\ind\rat}_{\pro\et}$ is (the opposite category of) the category $F^{\ind\rat}$
equipped with the pro-\'etale topology.
We denote $\Ab(F^{\ind\rat}_{\pro\et}) = \Ab(\Spec F^{\ind\rat}_{\pro\et})$
and use similar such notation as $\sheafhom_{F^{\ind\rat}_{\pro\et}}$.

We recall the perfect artinian \'etale site $\Spec F^{\perar}_{\et}$
from \cite{Suz21Imp}.
Let $F^{\perar}$ be the category of perfect artinian $F$-algebras
(or, equivalently, finite direct products of perfect field extensions of $F$)
with $F$-algebra homomorphisms.
Then $\Spec F^{\perar}_{\et}$ is the category $F^{\perar}$ endowed with the \'etale topology.

We recall the perfect pro-fppf site $\Spec F^{\perf}_{\pro\fppf}$ \cite{Suz22}.
A perfect $F$-algebra homomorphism $R \to S$ is said to be
flat of finite presentation (in the perfect algebra sense)
if $S$ is the perfection of a flat $R$-algebra of finite presentation.
It is said to be flat of ind-finite presentation
if $S$ is the direct limit of a filtered direct system of perfect $R$-algebras
flat of finite presentation.
Define $\Spec F^{\perf}_{\pro\fppf}$ to be the category of perfect $F$-algebras
where a covering of an object $R$ is a finite family $\{R_{i}\}$ of $R$-algebras
flat of ind-finite presentation such that $\prod_{i} R_{i}$ is faithfully flat over $R$.

In any of these sites, if $F_{0}$ is a finite extension of $F$,
let $\Weil_{F_{0} / F}$ be the Weil restriction functor for $F_{0} / F$
(which is the pushforward functor by the morphism of sites
defined by the base change functor for $\Spec F_{0} \to \Spec F$).

We recall the functor
$\algebrize \colon D(F^{\perar}_{\et}) \to D(F^{\ind\rat}_{\pro\et})$
\cite{Suz21Imp}.
Let
	\begin{equation} \label{0080}
			\Spec F^{\perf}_{\pro\fppf}
		\stackrel{f}{\to}
			\Spec F^{\ind\rat}_{\pro\et}
		\stackrel{g}{\to}
			\Spec F^{\perar}_{\et}
	\end{equation}
be the premorphisms defined by the inclusion functors on the underlying categories.
Let $h = g \compose f$.
Then $\algebrize = R f_{\ast} L h^{\ast}$.

Recall that the Yoneda functors
$\mathcal{W}_{F} \to \Ab(F^{\perar}_{\et})$
and $\Alg / F \to \Ab(F^{\perar}_{\et})$ are
fully faithful (\cite[Proposition 7.1]{Suz21Imp}).
As in \cite[Definition 3.1.3]{Suz24},
define $\genby{\mathcal{W}_{F}}_{F^{\perar}_{\et}}$
to be the smallest full triangulated subcategory of $D(F^{\perar}_{\et})$
closed under direct summands containing objects of $\mathcal{W}_{F}$ placed in degree zero.
We say that an object $G \in \Ind \Pro \Alg_{u} / F$ is pro-algebraic or ind-algebraic
if it is in the subcategory $\Pro \Alg_{u} / F$ or $\Ind \Alg_{u} / F$, respectively.

\begin{Prop} \label{0015}
	Let $q \in \Z$ and $G \in \genby{\mathcal{W}_{F}}_{F^{\perar}_{\et}}$.
	Set $H = \algebrize G$.
	Assume that $H^{q} G \in \mathcal{W}_{F}$ and $H^{q} H \in \mathcal{W}_{F}$.
	Then $H^{q} G \cong H^{q} H$.
\end{Prop}

\begin{proof}
	By \cite[Proposition 3.1.4]{Suz24},
	we have $R \Gamma(F'_{\et}, G) \cong R \Gamma(F'_{\pro\et}, H)$
	for all $F' \in F^{\perar}$.
	In particular, we have $(H^{q} G)(\closure{F'}) \cong (H^{q} H)(\closure{F'})$
	for any field $F' \in F^{\perar}$ and its algebraic closure $\closure{F'}$
	compatible with the actions of $\Gal(\closure{F'} / F')$.
	The assumption $H^{q} G \in \mathcal{W}_{F}$ implies that
	the $\Gal(\closure{F'} / F')$-invariant part of $(H^{q} G)(\closure{F'})$ is
	$(H^{q} G)(F')$.
	The assumption $G \in \genby{\mathcal{W}_{F}}_{F^{\perar}_{\et}}$ implies that
	$H \in D^{b}(\Ind \Pro \Alg_{u} / F)$ by \cite[Proposition 3.1.5]{Suz24}
	and hence $H^{q} H \in \Ind \Pro \Alg_{u} / F$.
	Therefore the $\Gal(\closure{F'} / F')$-invariant part of $(H^{q} H)(\closure{F'})$ is
	$(H^{q} H)(F')$.
	Thus $(H^{q} G)(F') \cong (H^{q} H)(F')$ for all $F' \in F^{\perar}$.
	Since $\mathcal{W}_{F}$ is a full subcategory of $\Ab(F^{\perar}_{\et})$,
	we obtain an isomorphism $H^{q} G \cong H^{q} H$.
\end{proof}

\begin{Prop} \label{0024}
	Let $G \in \Ind \Alg_{u} / F$.
	Then any subobject and quotient object of $G$ in $\Ind \Pro \Alg_{u} / F$
	is in $\Ind \Alg_{u} / F$.
\end{Prop}

\begin{proof}
	Let $H \into G$ be a subobject in $\Ind \Pro \Alg_{u} / F$.
	Write $H = \dirlim_{\lambda} H_{\lambda}$ with $H_{\lambda} \in \Pro \Alg_{u} / F$.
	Then for any $\lambda$,
	the composite morphism $H_{\lambda} \to H \into G$ is a morphism
	from an object of $\Pro \Alg_{u} / F$ to an object of $\Ind \Alg_{u} / F$.
	Hence it factors through an object of $\Alg_{u} / F$.
	Hence $H_{\lambda} \to H$ factors through an object of $\Alg_{u} / F$.
	This implies $H \in \Ind \Alg_{u} / F$,
	thus $G / H \in \Ind \Alg_{u} / F$.
\end{proof}

\begin{Prop} \label{0047}
	Let $G \onto H$ be a surjection in $\Ind \Pro \Alg_{u} / F$.
	Assume that $G \in \mathcal{W}_{F}$ and $H \in \Pro \Alg_{u} / F$.
	Then $H \in \mathcal{W}_{F}$.
\end{Prop}

\begin{proof}
	First assume that $G$ is pro-algebraic.
	We may assume $G$ connected.
	Let $K$ be the kernel of $G \onto H$.
	Write $G$ as the inverse limit of a surjective system of connected groups $G_{n} \in \Alg_{u} / F$
	with connected $\Ker(G_{n + 1} \onto G_{n})$.
	Let $H_{n}$ be the cokernel of the composite $K \into G \onto G_{n}$.
	Then $H$ is the inverse limit of the surjective system of the connected unipotent groups $H_{n}$
	with connected $\Ker(H_{n + 1} \onto H_{n})$.
	
	For the general case, write $G$ is a filtered union of proalgebraic subgroups $G_{m} \in \mathcal{W}_{F}$.
	Since $H \in \Pro \Alg_{u} / F$,
	we know that the composite $G_{m} \into G \onto H$ is surjective for some $m$,
	which reduces the statement to the first case.
\end{proof}

\begin{Prop} \label{0089}
	Let $\varphi \colon G \to H$ be a morphism in $\Ind \Pro \Alg_{u} / F$.
	Assume $G, H \in \mathcal{W}_{F}$.
	Then $\Im(\varphi) \in \mathcal{W}_{F}$.
\end{Prop}

\begin{proof}
	We may assume that $G$ and $H$ are connected.
	Write $G$ as a direct limit $\dirlim_{n} G_{n}$
	with $G_{n} \in \mathcal{W}_{F}$ connected pro-algebraic
	and $G_{n + 1} / G_{n} \in \Alg_{u} / F$.
	The restriction of $\varphi$ to $G_{n}$ factors through some pro-algebraic subgroup of $H$.
	Hence $\varphi(G_{n}) \in \Pro \Alg_{u} / F$,
	so $\varphi(G_{n}) \in \mathcal{W}_{F}$ by Proposition \ref{0047}.
	It is connected.
	The group $\varphi(G_{n + 1}) / \varphi(G_{n})$ is a quotient of $G_{n + 1} / G_{n} \in \Alg_{u} / F$,
	hence $\varphi(G_{n + 1}) / \varphi(G_{n}) \in \Alg_{u} / F$.
	Thus $\Im(\varphi) = \dirlim_{n} \varphi(G_{n})$ is in $\mathcal{W}_{F}$.
\end{proof}

\begin{Prop} \label{0048}
	Let $H \into G$ be an injection in $\Ind \Alg_{u} / F$.
	If $G \in \mathcal{W}_{F}$, then $H^{0} \in \mathcal{W}_{F}$.
	Hence if $G \in \mathcal{W}_{F}$ and $\pi_{0}(H)$ is finite,
	then $H \in \mathcal{W}_{F}$.
\end{Prop}

\begin{proof}
	We may assume that $G$ and $H$ are connected.
	First assume that $G \in \Alg_{u} / F$.
	Write $H$ as a filtered direct limit of connected $H_{\lambda} \in \Alg_{u} / F$.
	Then $\Im(H_{\lambda} \to H) \cong \Im(H_{\lambda} \to G)$ is connected in $\Alg_{u} / F$.
	As $G \in \Alg_{u} / F$, this implies that
	the system of subgroups $\Im(H_{\lambda} \to H)$ of $G$ is stationary,
	meaning that $\Im(H_{\lambda} \to H) = H$ for some $\lambda$.
	Hence $H \in \Alg_{u} / F$.
	
	For the general case, write $G$ as the direct limit of an injective system of
	connected groups $G_{n} \in \Alg_{u} / F$.
	Then $(H \cap G_{n})^{0} \in \Alg_{u} / F$ by the first case.
	Hence $H$ is the direct limit of the injective system of
	the connected groups $(H \cap G_{n})^{0} \in \Alg_{u} / F$.
\end{proof}

\begin{Prop} \label{0049}
	Let $0 \to K \to G \to H \to 0$ be an exact sequence in $\Ind \Pro \Alg_{u} / F$.
	Assume that $G \in \mathcal{W}_{F}$ and $H$ is ind-algebraic in $\mathcal{W}_{F}$.
	Then $K^{0} \in \mathcal{W}_{F}$ and $\pi_{0}(K)$ is \'etale
	(that is, in $\Ind \Alg_{u} / F$).
\end{Prop}

\begin{proof}
	Let $G' \subset G^{0}$ be connected pro-algebraic in $\mathcal{W}_{F}$
	such that $G^{0} / G' \in \mathcal{W}_{F}$ is ind-algebraic.
	Then the image of $G'$ in $H$ is in $\Alg_{u} / F$.
	Therefore $G' \cap K$ is pro-algebraic in $\mathcal{W}_{F}$.
	In the exact sequence
		\[
				0
			\to
				K / (G' \cap K)
			\to
				G / (G' \cap K)
			\to
				H
			\to
				0,
		\]
	the second and third terms are ind-algebraic in $\mathcal{W}_{F}$.
	Hence the first term $K / (G' \cap K)$ is ind-algebraic
	and its identity component is in $\mathcal{W}_{F}$
	by Proposition \ref{0048}.
\end{proof}

\begin{Prop} \label{0041}
	Let
		\begin{equation} \label{0023}
				\cdots
			\to
				C^{n - 1}
			\to
				A^{n}
			\to
				B^{n}
			\to
				C^{n}
			\to
				A^{n + 1}
			\to
				\cdots
		\end{equation}
	be a long exact sequence in $\Ind \Pro \Alg_{u} / F$.
	Assume that $A^{n} \in \Pro \Alg_{u} / F$,
	$B^{n} \in \mathcal{W}_{F}$ and
	$C^{n} \in \Ind \Alg_{u} / F$ for all $n$.
	Then any term of \eqref{0023} and the image of any morphism in \eqref{0023} are in $\mathcal{W}_{F}$.
	The images of $A^{n} \to B^{n}$, $B^{n} \to C^{n}$ and $C^{n} \to A^{n + 1}$
	are pro-algebraic, ind-algebraic and in $\Alg_{u} / F$, respectively, for all $n$.
\end{Prop}

\begin{proof}
	Let $D^{n}$, $E^{n}$ and $F^{n}$ be the images of
	$A^{n} \to B^{n}$, $B^{n} \to C^{n}$ and $C^{n} \to A^{n + 1}$, respectively.
	Choose a pro-algebraic subgroup  $B'^{n} \in \mathcal{W}_{F}$ of $B^{n}$
	such that $B^{n} / B'^{n} \in \mathcal{W}_{F}$ is ind-algebraic
	and the morphism $A^{n} \to B^{n}$ factors through $B'^{n}$.
	Then $D^{n}$ is the image of $A^{n} \to B'^{n}$, so it is in $\Pro \Alg_{u} / F$.
	We have $E^{n}, F^{n} \in \Ind \Alg_{u} / F$
	by Proposition \ref{0024}.
	The exact sequence $0 \to F^{n} \to A^{n + 1} \to D^{n + 1} \to 0$
	shows that $F^{n}$ is also in $\Pro \Alg_{u} / F$, hence in $\Alg_{u} / F$.
	We have an exact sequence
		\[
				0
			\to
				B'^{n} / D^{n}
			\to
				E^{n}
			\to
				B^{n} / B'^{n}
			\to
				0.
		\]
	Hence $B'^{n} / D^{n} \in \Alg_{u} / F$ by Proposition \ref{0024}.
	As $B^{n} / B'^{n} \in \mathcal{W}_{F}$, we have $E^{n} \in \mathcal{W}_{F}$.
	As $B'^{n} \in \mathcal{W}_{F}$, we have $D^{n} \in \mathcal{W}_{F}$.
	The rest follows since $\mathcal{W}_{F}$ is closed under extensions.
\end{proof}

A perfect pairing on the derived category level gives a duality for each cohomology objects
as soon as the cohomology objects are shown to belong to $\mathcal{W}_{F}$:

\begin{Prop} \label{0078}
	Let $G \tensor^{L} H \to \Lambda_{\infty}$ be a perfect pairing in
	$D(F^{\ind\rat}_{\pro\et})$ or $D(F^{\perar}_{\et})$.
	Assume that $H^{q} G$ are in $\mathcal{W}_{F}$ for all $q$
	and zero for all but finitely many $q$.
	Then $H^{q} H$ are in $\mathcal{W}_{F}$ for all $q$
	and zero for all but finitely many $q$,
	and the pairing induces a Serre duality
		\[
				(H^{q} G)^{0}
			\leftrightarrow
				(H^{1 - q} H)^{0}
		\]
	of connected groups in $\mathcal{W}_{F}$ and a Pontryagin duality
		\[
				\pi_{0}(H^{q} G)
			\leftrightarrow
				\pi_{0}(H^{- q} H)
		\]
	of finite \'etale groups over $F$ for all $q$.
\end{Prop}

\begin{proof}
	This follows from \cite[Propositions 3.1.7 and 3.1.9]{Suz24}.
\end{proof}


\section{Unipotent groups in positive characteristic}
\label{0113}

In this section, we show that
the derived category level duality statement by Kato \cite[Theorem 4.3 (ii)]{Kat86}
has a counterpart for each cohomology objects
when the base scheme is a field with a finite $p$-basis.
The cohomology objects turn out to be relatively perfect unipotent groups
studied in \cite[Section 8]{BS20}.
Hence Kato's duality in this case can be interpreted
as a duality for relatively perfect unipotent groups.
Under this duality, split groups correspond to split groups
and wound groups correspond to wound groups.

Let $k$ be a field of characteristic $p > 0$
such that $[k : k^{p}] = p^{r}$ for some finite $r \ge 0$.
Let $\Spec k_{\RP}$ be the relatively perfect site of $k$
(\cite[Section 2]{Kat86}, \cite[Section 2]{KS19}, \cite[Definition 8.2]{BS20}).
It is the category of relatively perfect $k$-schemes
(where relatively perfect means that
the relative Frobenius $Y \to Y^{(p)}$ over $k$ is an isomorphism)
with $k$-scheme morphisms endowed with the \'etale topology.
The inclusion functor from the category of relatively perfect $k$-schemes
to the category of all $k$-schemes admits a right adjoint (\cite[Definition 1.8]{Kat86}),
called the relative perfection functor and denoted by $Y \mapsto Y^{\RP}$.
Let $\Alg^{\RP} / k \subset \Ab(k_{\RP})$ be the full subcategory
consisting of relative perfections of quasi-compact smooth group schemes over $k$.
It is an abelian subcategory closed under extensions
by \cite[Propositions 8.7 and 8.12]{BS20}.

Recall from \cite[Chapter IV, Section 2, Definition 2.1]{DG70b} that
a (commutative) group scheme $G$ over $k$ (not necessarily of finite type)
is said to be \emph{unipotent}
if it is affine and for any closed subgroup scheme $0 \ne H \subset G$,
there is a non-zero morphism $H \to \Ga$ of group schemes over $k$.

\begin{Prop}
	An object of $\Alg^{\RP} / k$ is unipotent
	if and only if it is the relative perfection of a quasi-compact smooth unipotent group scheme over $k$.
	Denote by $\Alg^{\RP}_{u} / k$ the full subcategory of $\Alg^{\RP} / k$
	consisting of objects satisfying these equivalent conditions.
	It is an abelian subcategory closed under extensions.
\end{Prop}

\begin{proof}
	Let $G \in \Alg^{\RP} / k$ be unipotent.
	Write it as the relative perfection of a quasi-compact smooth group scheme $G_{0}$ over $k$.
	The natural morphism $G \to G_{0}$ is faithfully flat
	by \cite[Proposition 8.13]{BS20}.
	Hence $G_{0}$ is unipotent by \cite[Chapter IV, Section 2, Proposition 2.3 (a)]{DG70b}.
	Conversely, let $G_{0}$ be a quasi-compact smooth unipotent group scheme over $k$.
	Then $G = G_{0}^{\RP}$ can be written as the inverse limit of a system
	$\cdots \to G_{2} \to G_{1} \to G_{0}$ of groups schemes over $k$,
	where each $G_{n}$ is the Weil restriction
	along the $p^{n}$-th power Frobenius $\Spec k \to \Spec k$ of $G_{0}$
	(\cite[Section 2]{BS20}).
	Since Weil restrictions of unipotent groups are unipotent,
	we know that $G_{n}$ is unipotent.
	Hence $G = \invlim_{n} G_{n}$ is unipotent.
	The second statement follows from the fact that
	unipotent group schemes are closed under kernels, cokernels and extensions.
\end{proof}

Let $D_{0}(k_{\RP}) \subset D(k_{\RP})$ be the smallest triangulated full subcategory
containing $\Ga^{\RP}$ (\cite[Definition 4.2.3]{Kat86}).

\begin{Prop} \label{0000}
	$D_{0}(k_{\RP}) \subset D(k_{\RP})$ agrees with the full subcategory of $D^{b}(k_{\RP})$
	consisting of objects $G$ with $H^{q} G \in \Alg^{\RP}_{u} / k$ for all $q$.
\end{Prop}

\begin{proof}
	The latter category is triangulated by \cite[Tag 06UQ]{Sta22}
	and hence contains $D_{0}(k_{\RP})$.
	For the opposite inclusion, it is enough to show that
	$\Alg^{\RP}_{u} / k \subset D_{0}(k_{\RP})$.
	Let $G \in \Alg^{\RP}_{u} / k$ be the relative perfection of
	a quasi-compact smooth unipotent group scheme $G_{0}$.
	By \cite[Propositions 8.13 and 8.14]{BS20},
	we know that $\pi_{0}(G) \cong \pi_{0}(G_{0})$ is a finite \'etale group scheme over $k$.
	The Artin-Schreier sequence shows that $\Lambda \in D_{0}(k_{\RP})$
	and hence $\pi_{0}(G) \in D_{0}(k_{\RP})$.
	Therefore we may assume that $G$ (and hence $G_{0}$) is connected.
	Then $G_{0}$ admits a finite filtration whose graded pieces are
	smooth connected group schemes killed by $p$ by \cite[Corollary B.3.3]{CGP15}.
	A short exact sequence of smooth group schemes remains exact after relative perfection
	by \cite[Proposition 8.8]{BS20}.
	Hence we may assume that $G_{0}$ is killed by $p$.
	Then there exists an exact sequence
	$0 \to G_{0} \to \Ga^{\dim G_{0} + 1} \to \Ga \to 0$
	by \cite[B.1.13]{CGP15}.
	Hence $G \in D_{0}(k_{\RP})$.
\end{proof}

We say that $G \in \Alg^{\RP}_{u} / k$ is \emph{wound}
if any morphism $\Ga^{\RP} \to G$ is zero.
This is equivalent that $G$ be the relative perfection
of a wound smooth unipotent group scheme $G_{0}$ over $k$
in the sense that any morphism $\Ga \to G_{0}$ be zero (\cite[Definition B.2.1]{CGP15}).
Here we allow disconnected groups.
Any $p$-primary finite \'etale group is wound in our sense.
We say that $G \in \Alg^{\RP}_{u} / k$ is \emph{split}
if it admits a finite filtration whose graded pieces are isomorphic to $\Ga^{\RP}$.
\newcommand{\spl}{\mathrm{split}}
By \cite[Theorem B.3.4]{CGP15},
any $G \in \Alg^{\RP}_{u} / k$ admits a unique maximal split subgroup $G_{\spl}$
and $G / G_{\spl}$ is wound.

For $n \ge 1$, let $\nu_{n}(r) \in \Ab(k_{\RP})$ be 
the dlog part of the Hodge-Witt sheaf $W_{n} \Omega_{k}^{r}$ (\cite[Section 4]{Kat86}).
Let $\nu_{\infty}(r)$ be the direct limit of $\nu_{n}(r)$ over $n \ge 1$.
We have
	\[
			R \sheafhom_{k_{\RP}}(\var, \nu_{\infty}(r))
		\cong
			\dirlim_{n}
			R \sheafhom_{k_{\RP}}(\var, \nu_{n}(r))
	\]
on $D_{0}(k_{\RP})$
by the same proof as \cite[Proposition 8.10]{BS20}.
Recall from \cite[Theorem 4.3]{Kat86} that
the endofunctor $R \sheafhom_{k_{\RP}}(\var, \nu_{\infty}(r))$ on $D(k_{\RP})$
restricts to a contravariant autoequivalence on the subcategory $D_{0}(k_{\RP})$
with inverse itself.

\begin{Prop} \label{0002}
	Let $G \in \Alg^{\RP}_{u} / k$.
	Then $\sheafhom_{k_{\RP}}(G, \nu_{\infty}(r)) \in \Alg^{\RP}_{u} / k$ is wound,
	$\sheafext_{k_{\RP}}^{1}(G, \nu_{\infty}(r)) \in \Alg^{\RP}_{u} / k$ is split
	and $\sheafext_{k_{\RP}}^{q}(G, \nu_{\infty}(r)) = 0$ for all $q \ge 2$.
\end{Prop}

\begin{proof}
	We have $R \sheafhom_{k_{\RP}}(G, \nu_{\infty}(r)) \in D_{0}(k_{\RP})$
	by \cite[Theorem 4.3 (ii)]{Kat86}.
	Hence $\sheafext_{k_{\RP}}^{q}(G, \nu_{\infty}(r)) \in \Alg^{\RP}_{u} / k$
	for all $q$ by Proposition \ref{0000}.
	Also $\sheafext_{k_{\RP}}^{q}(G, \nu_{\infty}(r)) = 0$ for all $q \ge 2$
	by \cite[Theorem 3.2]{Kat86}.
	Since $\sheafhom_{k_{\RP}}(\Ga^{\RP}, \nu_{\infty}(r)) = 0$
	by \cite[Theorem 3.2 (ii)]{Kat86},
	any morphism $\Ga^{\RP} \to \sheafhom_{k_{\RP}}(G, \nu_{\infty}(r))$ is zero.
	Therefore $\sheafhom_{k_{\RP}}(G, \nu_{\infty}(r))$ is wound.
	To show the splitness of $\sheafext_{k_{\RP}}^{1}(G, \nu_{\infty}(r)) \in \Alg^{\RP}_{u} / k$,
	note that a quotient of a split group and an extension of split groups are split.
	Hence again by \cite[Corollary B.3.3]{CGP15},
	we may assume that $p G = 0$.
	Take an exact sequence $0 \to G_{0} \to \Ga^{\dim G_{0} + 1} \to \Ga \to 0$ as above.
	We have $\sheafext_{k_{\RP}}^{1}(\Ga^{\RP}, \nu_{\infty}(r)) \cong \Omega_{k}^{r} \cong \Ga^{\RP}$
	by \cite[Theorem 3.2 (ii)]{Kat86}.
	Therefore applying $\sheafext_{k_{\RP}}^{1}(\var, \nu_{\infty}(r))$ to $G_{0} \into \Ga^{\dim G_{0} + 1}$ gives
	a surjection from $(\Ga^{\RP})^{\dim G_{0} + 1}$ to $\sheafext_{k_{\RP}}^{1}(G, \nu_{\infty}(r))$.
	Thus $\sheafext_{k_{\RP}}^{1}(G, \nu_{\infty}(r))$ is split.
\end{proof}

\begin{Prop} \label{0079}
	Let $G \in \Alg^{\RP}_{u} / k$.
	Then $\sheafhom_{k_{\RP}}(G, \nu_{\infty}(r)) = 0$ if $G$ is split
	and $\sheafext_{k_{\RP}}^{1}(G, \nu_{\infty}(r)) = 0$ if $G$ is wound.
\end{Prop}

\begin{proof}
	The first statement is \cite[Theorem 3.2 (ii)]{Kat86}.
	Assume that $G$ is wound.
	By Proposition \ref{0002}, the isomorphism
		\[
				G
			\isomto
				R \sheafhom_{k_{\RP}} \bigl(
					R \sheafhom_{k_{\RP}}(G, \nu_{\infty}(r)),
					\nu_{\infty}(r)
				\bigr)
		\]
	induces an exact sequence
		\[
				0
			\to
				\sheafext_{k_{\RP}}^{1} \bigl(
					\sheafext_{k_{\RP}}^{1}(G, \nu_{\infty}(r)),
					\nu_{\infty}(r)
				\bigr)
			\to
				G
			\to
				\sheafhom_{k_{\RP}} \bigl(
					\sheafhom_{k_{\RP}}(G, \nu_{\infty}(r)),
					\nu_{\infty}(r)
				\bigr)
			\to
				0.
		\]
	The first term in this sequence is split
	by Proposition \ref{0002} and hence zero.
	This implies that the split group
	$\sheafext_{k_{\RP}}^{1}(G, \nu_{\infty}(r))$ is zero.
\end{proof}

\begin{Prop}
	Let $G \in \Alg_{u}^{\RP} / k$.
	If $G$ is wound,
	then $\sheafhom_{k_{\RP}}(G, \nu_{\infty}(r)) \in \Alg_{u}^{\RP} / k$ is wound
	and
		\[
				G
			\cong
				\sheafhom_{k_{\RP}} \bigl(
					\sheafhom_{k_{\RP}}(G, \nu_{\infty}(r)),
					\nu_{\infty}(r)
				\bigr).
		\]
	If $G$ is split,
	then $\sheafext_{k_{\RP}}^{1}(G, \nu_{\infty}(r)) \in \Alg_{u}^{\RP} / k$ is split
	and
		\[
				G
			\cong
				\sheafext_{k_{\RP}}^{1} \bigl(
					\sheafext_{k_{\RP}}^{1}(G, \nu_{\infty}(r)),
					\nu_{\infty}(r)
				\bigr).
		\]
\end{Prop}

\begin{proof}
	This follows from Propositions \ref{0002} and \ref{0079}.
\end{proof}

This gives a duality for would groups and for split groups in $\Alg_{u}^{\RP} / k$.
For example, for any $n \ge 1$ and $0 \le s \le r$,
the groups $\nu_{n}(s)$ and $\nu_{n}(r - s)$ are wound groups dual to each other
by \cite[Theorem 4.3]{Kat86}.
Note that $\nu_{n}(0) = \Z / p^{n} \Z$ is discrete,
while one can check that its dual $\nu_{n}(r)$ is connected if $r > 0$.
On the other hand, for a finite-dimensional $k$-vector group $V$,
its dual in the above sense (that is, $\sheafext_{k_{\RP}}^{1}(V, \nu_{\infty}(r))$)
is given by the vector space dual $V^{\ast}$ tensored with $\Omega_{k}^{r}$
by \cite[Theorem 3.2 (ii)]{Kat86}.


\section{Duality for unipotent groups over local fields of positive characteristic}
\label{0007}

Recall that $F$ is a fixed perfect field of characteristic $p > 0$.
Let $k$ be a complete discrete valuation field of characteristic $p$ with residue field $F$.
In this section, we give a duality for cohomology of $k$
with coefficients in relatively perfect unipotent groups.
Compare it with Rosengarten's duality \cite{Ros18}.

Let $\Order_{k}$ be the ring of integers of $k$
with maximal ideal $\ideal{p}_{k}$.
The functor
	\[
			F'
		\mapsto
			\alg{k}(F')
		=
			(F' \ctensor_{F} \Order_{k}) \tensor_{\Order_{k}} k
	\]
(where $F' \ctensor_{F} \Order_{k}$ means
the inverse limit of $F' \tensor_{F} \Order_{k} / \ideal{p}_{k}^{n}$ over $n$)
defines a premorphism of sites
	\[
			\pi_{k, \RP}
		\colon
			\Spec k_{\RP}
		\to
			\Spec F^{\perar}_{\et}.
	\]

\begin{Prop} \label{0005}
	Let $G \in \Alg^{\RP}_{u} / k$.
	Then $\pi_{k, \RP, \ast} G \in \mathcal{W}_{F}$.
	It is pro-algebraic if $G$ is wound.
\end{Prop}

\begin{proof}
	We have $R^{q} \pi_{k, \RP, \ast} G_{\spl} = 0$ for $q \ge 1$.
	Hence we have an exact sequence
		\[
				0
			\to
				\pi_{k, \RP, \ast} G_{\spl}
			\to
				\pi_{k, \RP, \ast} G
			\to
				\pi_{k, \RP, \ast} (G / G_{\spl})
			\to
				0
		\]
	in $\Ab(F^{\perar}_{\et})$.
	Since $\pi_{k, \RP, \ast} \Ga^{\RP} \cong \alg{k} \cong \Ga^{\N} \oplus \Ga^{\oplus \N} \in \mathcal{W}_{F}$,
	we have $\pi_{k, \RP, \ast} G_{\spl} \in \mathcal{W}_{F}$.
	Let $G' = G / G_{\spl}$.
	Since it is wound, it admits a N\'eron finite-type model $\mathcal{G}'$ over $\Order_{k}$
	by \cite[Section 10.2, Theorem 1]{BLR90}.
	For any $F' \in F^{\perar}$, we have
		\[
				(\pi_{k, \RP, \ast} G')(F')
			\cong
				\mathcal{G}'(F' \ctensor_{F} \Order_{k}).
		\]
	Therefore $\pi_{k, \RP, \ast} G'$ is represented by the perfect Greenberg transform of $\mathcal{G}'$
	\cite[Section 14, Equation (74)]{BGA18Green}.
	By \cite[Proposition 11.1]{BGA18Green}, the natural morphism
	$\pi_{k, \RP, \ast} G' \to \mathcal{G}'_{F}$ to the special fiber of $\mathcal{G}'$
	is surjective and its kernel is pro-algebraic in $\mathcal{W}_{F}$.
	Thus $\pi_{k, \RP, \ast} G' \in \mathcal{W}_{F}$.
	Hence $\pi_{k, \RP, \ast} G \in \mathcal{W}_{F}$.
\end{proof}

\begin{Prop} \label{0001}
	Let $G \in \Ind \Pro \Alg_{u} / F$ be such that $\pi_{0}(G)$ is \'etale.
	Assume that
		\[
				\dirlim_{\lambda} G(F'_{\lambda})
			\isomto
				G(\dirlim_{\lambda} F'_{\lambda})
		\]
	for any filtered direct system $\{F'_{\lambda}\}$ of perfect fields over $F$.
	Then $G \in \Ind \Alg_{u} / F$.
\end{Prop}

\begin{proof}
	The assumption and conclusion are satisfied if $G$ is \'etale.
	Hence we may assume that $G$ is connected.
	Then $G$ can be written as the direct limit of a filtered direct system $\{G_{\mu}\}$
	of connected groups in $\Pro \Alg_{u} / F$.
	Hence we need to show that any morphism $H \to G$
	from a connected group $H \in \Pro \Alg_{u} / F$
	factors through an object of $\Alg_{u} / F$.
	Let $\xi_{H}$ be the generic point of $H$.
	The same argument as the proof of \cite[Lemma (3.4.4)]{Suz20} applied to this situation
	gives the result.
\end{proof}

\begin{Prop} \label{0004}
	Let $G \in \Alg^{\RP}_{u} / k$.
	Then $R^{1} \pi_{k, \RP, \ast} G \in \Ind \Alg_{u} / F$ is connected
	and $R^{q} \pi_{k, \RP, \ast} G = 0$ for $q \ge 2$.
\end{Prop}

\begin{proof}
	As in the proof of Proposition \ref{0000},
	we may assume that $G$ is wound and there exists an exact sequence
	$0 \to G \to \Ga^{\RP n} \to \Ga^{\RP} \to 0$, where $n = \dim(G) + 1$.
	We have $R^{q} \pi_{k, \RP, \ast} \Ga^{\RP} = 0$ for $q \ge 1$.
	Hence $R^{q} \pi_{k, \RP, \ast} G = 0$ for $q \ge 2$,
	and we have an exact sequence
		\[
				0
			\to
				\pi_{k, \RP, \ast} G
			\to
				\alg{k}^{n}
			\to
				\alg{k}
			\to
				R^{1} \pi_{k, \RP, \ast} G
			\to
				0
		\]
	in $\Ab(F^{\perar}_{\et})$.
	The morphism $\pi_{k, \RP, \ast} G \to \alg{k}^{n}$ is injective in $\Ind \Pro \Alg_{u} / F$.
	Let $H$ be its cokernel in $\Ind \Pro \Alg_{u} / F$,
	which is in $\mathcal{W}_{F}$ by Propositions \ref{0005} and \ref{0041}.
	Then, since the pro-\'etale cohomology with coefficients in an object of $\mathcal{W}_{F}$
	is isomorphic to the \'etale cohomology by \cite[Proposition 3.1.2]{Suz24},
	this cokernel $H$ represents the cokernel of $\pi_{k, \RP, \ast} G \to \alg{k}^{n}$
	in $\Ab(F^{\perar}_{\et})$.
	The same argument applied to $H \into \alg{k}$ shows that
	$R^{1} \pi_{k, \RP, \ast} G \in \Ind \Pro \Alg_{u} / F$.
	It is connected since $\alg{k}$ is connected.
	For any filtered direct system $\{F'_{\lambda}\}$ of perfect fields over $F$,
	we have
		\[
				\dirlim_{\lambda}
					H^{1}(\alg{k}(F'_{\lambda}), G)
			\isomto
				H^{1} \bigr(
					\dirlim_{\lambda}
						\alg{k}(F'_{\lambda}),
					G
				\bigr)
			\isomto
				H^{1} \bigl(
					\alg{k}(\dirlim_{\lambda} F'_{\lambda}),
					G
				\bigr)
		\]
	by \cite[Proposition 3.5.3 (2)]{GGMB14}.
	This implies that
	$R^{1} \pi_{k, \RP, \ast} G \in \Ind \Pro \Alg_{u} / F$ satisfies
	the assumption of Proposition \ref{0001}.
	Thus $R^{1} \pi_{k, \RP, \ast} G \in \Ind \Alg_{u} / F$.
\end{proof}

We have
	\[
			R^{q} \pi_{k, \RP, \ast} \nu_{\infty}(1)
		\cong
			\begin{cases}
					\alg{k}^{\times} \tensor_{\Z} \Lambda_{\infty}
				&	\text{if }
					q = 0,
				\\
					0
				&	\text{otherwise}
			\end{cases}
	\]
by \cite[Proposition 6.1.6]{Suz24}.
With the valuation map $\alg{k}^{\times} \onto \Z$,
this gives a canonical morphism
	\[
			R \pi_{k, \RP, \ast} \nu_{\infty}(1)
		\to
			\Lambda_{\infty}
	\]
in $D(F^{\perar}_{\et})$.

\begin{Prop} \label{0099}
	Let $G \in D_{0}(k_{\RP})$ and set $H = R \sheafhom_{k_{\RP}}(G, \nu_{\infty}(1))$.
	Then the composite
		\begin{equation} \label{0003}
					R \pi_{k, \RP, \ast} G
				\tensor^{L}
					R \pi_{k, \RP, \ast} H
			\to
				R \pi_{k, \RP, \ast} \nu_{\infty}(1)
			\to
				\Lambda_{\infty}
		\end{equation}
	is a perfect pairing in $D(F^{\perar}_{\et})$.
\end{Prop}

\begin{proof}
	We may assume that $G = \Ga^{\RP}$ and hence $H = \Omega_{k}^{1}[-1]$.
	This case can be proved in the same way as the proof of \cite[Proposition (5.2.2.4)]{Suz20}
	(or is reduced to it).
	More explicitly, we have $R \pi_{k, \RP, \ast} \Ga^{\RP} \cong \alg{k}$
	and $R \pi_{k, \RP, \ast} \Omega_{k}^{1} \cong \Omega_{\alg{k}}^{1}$,
	where $\Omega_{\alg{k}}^{1} = \Omega_{k}^{1} \tensor_{k} \alg{k}$.
	The multiplication pairing $k \times \Omega_{k}^{1} \to \Omega_{k}^{1}$
	followed by the residue map $\Omega_{k}^{1} \onto F$ is a perfect pairing of Tate vector spaces over $F$.
	Hence \cite[Propositions 3.4.5 and 3.4.4]{Suz24} imply the result.
\end{proof}

\begin{Prop} \label{0006}
	Let $G \in \Alg^{\RP}_{u} / k$ be wound
	and set $H = \sheafhom_{k_{\RP}}(G, \nu_{\infty}(1))$.
	Then the paring \eqref{0003} induces isomorphisms
		\begin{gather*}
					\pi_{0}(\pi_{k, \RP, \ast} G)
				\cong
					\sheafhom_{F^{\perar}_{\et}}(
						\pi_{0}(\pi_{k, \RP, \ast} H),
						\Lambda_{\infty}
					),
			\\
					R^{1} \pi_{k, \RP, \ast} G
				\cong
					\sheafext_{F^{\perar}_{\et}}^{1}(
						(\pi_{k, \RP, \ast} H)^{0},
						\Lambda_{\infty}
					).
		\end{gather*}
\end{Prop}

\begin{proof}
	This follows from the connectedness of $R^{1} \pi_{k, \RP, \ast} H$
	(Proposition \ref{0004}).
\end{proof}

\begin{Prop} \label{0009}
	Let $G \in \Alg^{\RP}_{u} / k$.
	Then $R^{1} \pi_{k, \RP, \ast} G \in \mathcal{W}_{F}$.
\end{Prop}

\begin{proof}
	We may assume that $G$ is wound.
	Let $H = \sheafhom_{k_{\RP}}(G, \nu_{\infty}(1))$.
	Since $\pi_{k, \RP, \ast} H \in \mathcal{W}_{F}$ by
	Proposition \ref{0005},
	we have
		\[
				\sheafext_{F^{\perar}_{\et}}^{1}(
					(\pi_{k, \RP, \ast} H)^{0},
					\Lambda_{\infty}
				)
			\in
				\mathcal{W}_{F}
		\]
	by \cite[Proposition 3.1.7 (2)]{Suz24}.
	With Proposition \ref{0006}, we get the result.
\end{proof}

For $G \in \Alg_{u}^{\RP}$, set
	\[
			R \alg{\Gamma}(\alg{k}, G)
		=
			\algebrize
			R \pi_{k, \RP, \ast} G
		\in
			D(F^{\ind\rat}_{\pro\et})
	\]
and $\alg{H}^{q} = H^{q} R \alg{\Gamma}$ for $q \in \Z$
and $\alg{\Gamma} = \alg{H}^{0}$.

\begin{Prop}
	Let $G \in \Alg^{\RP}_{u} / k$.
	\begin{enumerate}
		\item
			Assume $G$ is split.
			Set $H = \sheafext^{1}_{k_{\RP}}(G, \nu_{\infty}(1))$.
			Then $\alg{H}^{q}(\alg{k}, G) = 0$ unless $q = 0$.
			We have a Serre duality
				\[
						\alg{\Gamma}(\alg{k}, G)
					\leftrightarrow
						\alg{\Gamma}(\alg{k}, H)
				\]
			of connected groups in $\mathcal{W}_{F}$.
		\item
			Assume $G$ is wound.
			Set $H = \sheafhom_{k_{\RP}}(G, \nu_{\infty}(1))$.
			Then $\alg{H}^{q}(\alg{k}, G) = 0$ unless $0 \le q \le 1$.
			We have a Pontryagin duality
				\begin{equation} \label{0097}
						\pi_{0}
						\alg{\Gamma}(\alg{k}, G)
					\leftrightarrow
						\pi_{0}
						\alg{\Gamma}(\alg{k}, H)
				\end{equation}
			of finite \'etale groups over $F$ and a Serre duality
				\begin{equation} \label{0098}
						\alg{H}^{1}(\alg{k}, G)
					\leftrightarrow
						\alg{\Gamma}(\alg{k}, H)^{0}
				\end{equation}
			of connected groups in $\mathcal{W}_{F}$.
	\end{enumerate}
\end{Prop}

\begin{proof}
	This follows from Propositions \ref{0078}, \ref{0005}, \ref{0004}, \ref{0099}, \ref{0006} and \ref{0009}.
\end{proof}

The dualities \eqref{0097} and \eqref{0098} above are unipotent group analogues
of Grothendieck's and Shafarevich's duality conjectures, respectively, for abelian varieties
proved in \cite{Suz20} after Grothendieck, Werner, B\'egueri, Bester, Bertapelle and others.


\section{Two-dimensional local fields}
\label{0070}

Let $k$ be as in Section \ref{0007}.
Let $K$ be a complete discrete valuation field of characteristic zero
with residue field $k$.
Let $\Order_{K}$ be its ring of integers.
In this section, we show that
the derived category level duality statement for $K$ in \cite[Section 6]{Suz24}
has a counterpart for each cohomology objects.
The key point is to show that the nearby cycle functor from $K$ to $k$ takes values in
relatively perfect unipotent groups,
which allows us to apply the results of Section \ref{0007} to these groups.
There are some technicalities to deal with
about the choice of sites.
Below we identify $\Z / p^{n} \Z(r)$ with the $p$-adic \'etale Tate twist $\mathfrak{T}_{n}(r)$
(\cite{Sch94, Gei04Ded, Sat07, Sat13}).

Let $\Spec k_{\RPS}$ be the relatively perfectly smooth site of $k$ (\cite[Section 2]{KS19}).
It is the category of relatively perfectly smooth $k$-schemes
(that is, $k$-schemes Zariski locally isomorphic to relative perfections of smooth $k$-schemes)
with $k$-scheme morphisms endowed with the \'etale topology.
Let $D_{0}(k_{\RPS}) \subset D(k_{\RPS})$ be
the smallest triangulated subcategory containing $\Ga^{\RP}$.
Let $\alpha \colon \Spec k_{\RP} \to \Spec k_{\RPS}$ be the premorphism
defined by the inclusion functor on the underlying categories.
Its (exact) pushforward functor gives an equivalence
$D_{0}(k_{\RP}) \isomto D_{0}(k_{\RPS})$ (\cite[Corollary 2.3]{KS19}).
We identify these equivalent categories.
For $G \in D_{0}(k_{\RP})$ (or $D_{0}(k_{\RPS})$), the natural morphism gives an isomorphism
	\[
			\alpha_{\ast}
			R \sheafhom_{k_{\RP}}(G, \nu_{\infty}(1))
		\isomto
			R \sheafhom_{k_{\RPS}}(G, \nu_{\infty}(1))
	\]
by \cite[Proposition 2.2]{KS19}.

Let $\Spec \Order_{K, \RPS}$ be relatively perfectly smooth site of $\Order_{K}$
(\cite[Section 3]{KS19}),
whose objects are flat $\Order_{K}$-algebras $R$
such that $R \tensor_{\Order_{K}} k$ is relatively perfectly smooth over $k$.
The topology of $\Spec \Order_{K, \RPS}$ is the \'etale topology.
Let $\Spec K_{\Et}$ be the big \'etale site of $K$.
By \cite[Proposition 3.2, Equation (3.1)]{KS19},
the base change functors define morphisms of sites
	\[
			\Spec K_{\Et}
		\stackrel{j_{\RPS}}{\to}
			\Spec \Order_{K, \RPS}
		\stackrel{i_{\RPS}}{\gets}
			\Spec k_{\RPS}.
	\]
Let $R \Psi_{\RPS} = i_{\RPS}^{\ast} R j_{\RPS, \ast}$ and
$R^{q} \Psi_{\RPS} = i_{\RPS}^{\ast} R^{q} j_{\RPS, \ast}$ for $q \in \Z$.
For $n \ge 1$ and $r \ge 0$, we have a canonical morphism
	\[
			R^{r} \Psi_{\RPS} \Lambda_{n}(r)
		\to
			\nu_{n}(r - 1)
	\]
in $\Ab(k_{\RPS})$ by \cite[Proposition 5.1]{KS19}
(where we set $\nu_{n}(-1) = 0$).
Hence we have a canonical morphism
	\begin{equation} \label{0107}
			\tau_{\le r}
			R j_{\RPS, \ast} \Lambda_{n}(r)
		\to
			i_{\RPS, \ast} \nu_{n}(r - 1)
	\end{equation}
in $D(\Order_{K, \RPS})$.
Let $\mathfrak{T}_{n}(r) \in D(\Order_{K, \RPS})$ be its canonical mapping fiber.
For $r < 0$, let $\mathfrak{T}_{n}(r) = j_{\RPS, !} \Lambda_{n}(r)$.
We have a canonical natural transformation
	\[
			i_{\RPS}^{\ast} \to i_{\RPS}^{\ast} R j_{\RPS, \ast} j_{\RPS}^{\ast}
		\colon
			D(\Order_{K, \RPS})
		\to
			D(k_{\RPS}).
	\]
Since $j_{\RPS}^{\ast}$ sends K-injectives to K-injectives,
we have its canonical mapping fiber, which we denote by $R i_{\RPS}^{!}$,
so that we have a canonical distinguished triangle
	\[
			R i_{\RPS}^{!} G
		\to
			i_{\RPS}^{\ast} G
		\to
			R \Psi_{\RPS} j_{\RPS}^{\ast} G
	\]
in $D(k_{\RPS})$ for $G \in D(\Order_{K, \RPS})$ functorially.
In particular, we have a distinguished triangle
	\[
			i_{\RPS}^{\ast} \mathfrak{T}_{n}(r)
		\to
			R \Psi_{\RPS} \Lambda_{n}(r)
		\to
			R i_{\RPS}^{!} \mathfrak{T}_{n}(r)[1].
	\]

\begin{Prop} \label{0008}
	Let $n \ge 1$ and $r \in \Z$.
	Then the objects $R \Psi_{\RPS} \Lambda_{n}(r)$, $i_{\RPS}^{\ast} \mathfrak{T}_{n}(r)$ and
	$R i_{\RPS}^{!} \mathfrak{T}_{n}(r)$ are in $D_{0}(k_{\RPS})$.
\end{Prop}

\begin{proof}
	It is enough to treat $R \Psi_{\RPS} \Lambda_{n}(r)$ and $i_{\RPS}^{\ast} \mathfrak{T}_{n}(r)$.
	We may assume that $n = 1$.
	We may also assume that $K$ contains a fixed primitive $p$-th root of unity $\zeta_{p}$.
	We identify all the twists $\Lambda(r)$ with $\Lambda$ using $\zeta_{p}$.
	Then for any $q \in \Z$, the sheaf $R^{q} \Psi_{\RPS} \Lambda$ has a finite filtration
	whose graded pieces are either $\Ga^{\RP}$, $\Omega_{k}^{1}$, $\Ga^{\RP} / (\Ga^{\RP})^{p}$, $\Lambda$ or $\nu(1)$
	by \cite[Proposition 6.1]{KS19}.
	All these graded pieces are in $\Alg^{\RP}_{u} / k$.
	Hence $R^{q} \Psi_{\RPS} \Lambda \in \Alg^{\RP}_{u} / k$
	and so $R \Psi_{\RPS} \Lambda \in D_{0}(\Alg^{\RP}_{u} / k)$.
	We have
		\[
				H^{q} i_{\RPS}^{\ast} \mathfrak{T}(r)
			\cong
				\begin{cases}
						R^{q} \Psi_{\RPS} \Lambda
					&	\text{if }
						q < r,
					\\
						\Ker \bigl(
							R^{r} \Psi_{\RPS} \Lambda \onto \nu(r - 1)
						\bigr)
					&	\text{if }
						q = r,
					\\
						0
					&	\text{if }
						q > r.
				\end{cases}
		\]
	Hence it is in $\Alg^{\RP}_{u} / k$ for all $q$.
\end{proof}

\begin{Prop} \label{0010}
	Let $n \ge 1$ and $q, s, r \in \Z$.
	Let $G \in \Alg^{\RP}_{u} / k$ be either
	$R^{q} \Psi_{\RPS} \Lambda_{n}(r)$, $H^{q} i_{\RPS}^{\ast} \mathfrak{T}_{n}(r)$
	or $H^{q} R i_{\RPS}^{!} \mathfrak{T}_{n}(r)$.
	Then $R^{s} \pi_{k, \RP, \ast} G \in \mathcal{W}_{F}$.
\end{Prop}

\begin{proof}
	This follows from Propositions \ref{0005}, \ref{0009} and \ref{0008}.
\end{proof}

Recall from \cite[Section 1, Definition 1]{Kat82} and \cite[Section 2.3]{Suz24}
that for a relatively perfect $k$-algebra $R$,
its Kato canonical lifting over $\Order_{K}$ is a unique complete flat $\Order_{K}$-algebra $S$
equipped with an isomorphism $S \tensor_{\Order_{K}} k \cong R$ over $k$.

For $F' \in F^{\perar}$, let $\alg{O}_{K}(F')$ be
the Kato canonical lifting of $\alg{k}(F')$ over $\Order_{K}$.
Define $\alg{K}(F') = \alg{O}_{K}(F') \tensor_{\Order_{K}} K$.
We recall the relative sites $\Spec \alg{K}_{\et}$, $\Spec \alg{O}_{K, \et}$ and $\Spec \alg{k}_{\et}$
from \cite[Sections 4.1 and 5.1]{Suz24}.
The underlying category of $\Spec \alg{k}_{\et}$ consists of pairs $(k', F')$,
where $F' \in F^{\perar}$ and $k'$ an \'etale $\alg{k}(F')$-algebra.
Morphisms are morphisms between morphisms $(\alg{k}(F') \to k') \to (\alg{k}(F'') \to k'')$.
A covering is a finite family $\{(k', F') \to (k'_{i}, F'_{i})\}$
such that $F' \to F'_{i}$ is \'etale for all $i$
and $k' \to \prod_{i} k'_{i}$ is faithfully flat \'etale.
The sites $\Spec \alg{K}_{\et}$ and $\Spec \alg{O}_{K, \et}$ are defined similarly
using the functors $\alg{K}$ and $\alg{O}_{K}$ instead.

We have morphisms of sites
	\[
		\begin{CD}
				\Spec \alg{K}_{\et}
			@> j >>
				\Spec \alg{O}_{K, \et}
			@< i <<
				\Spec \alg{k}_{\et}
			\\
			@V \pi_{\alg{K}} VV
			@V \pi_{\alg{O}_{K}} VV
			@V \pi_{\alg{k}} VV
			\\
				\Spec F^{\perar}_{\et}
			@=
				\Spec F^{\perar}_{\et}
			@=
				\Spec F^{\perar}_{\et},
		\end{CD}
	\]
where vertical morphisms are defined by the functors
$F' \mapsto (\alg{K}(F'), F')$, $(\alg{O}_{K}(F'), F')$ and $(\alg{k}(F'), F')$
and the upper horizontal morphisms are by
$(A', F') \mapsto (A' \tensor_{\Order_{K}} K, F')$ and
$(A', F') \mapsto (A' \tensor_{\Order_{K}} k, F')$
(where $A'$ is an \'etale $\alg{O}_{K}(F')$-algebra).

Let $R \Psi = i^{\ast} R j_{\ast}$.
For $r \in \Z$ and $n \ge 1$, the objects
$\Lambda_{n}(r)$ and $\mathfrak{T}_{n}(r)$ are naturally viewed as
objects of $\Ab(\alg{K}_{\et})$ and $D(\alg{O}_{K, \et})$, respectively.
For $q \in \Z$, the sheaves $R^{q} \pi_{\alg{K}, \ast} \Lambda_{n}(r)$
and $R^{q} \pi_{\alg{O}_{K}, \ast} \mathfrak{T}_{n}(r)$ are the \'etale sheafifications of the presheaves
	\[
			F'
		\mapsto
			H^{q}(\alg{K}(F'), \Lambda_{n}(r)),
		\quad
			H^{q}(\alg{O}_{K}(F'), \mathfrak{T}_{n}(r)).
	\]
The object $R \pi_{\alg{O}_{K}, !} \mathfrak{T}_{n}(r)$ is defined as
$R \pi_{\alg{k}, \ast} R i^{!} \mathfrak{T}_{n}(r)$,
where $i^{!}$ is the right adjoint of $i_{\ast}$
(see \cite[Section 4.3]{Suz24} for more details).

\begin{Prop} \label{0012}
	Let $G \in \Ab(\alg{k}_{\et})$ be torsion.
	Then $R^{q} \pi_{\alg{k}, \ast} G = 0$ for all $q \ge 2$.
\end{Prop}

\begin{proof}
	The sheaf $R^{q} \pi_{\alg{k}, \ast} G \in \Ab(F^{\perar}_{\et})$ is
	the \'etale sheafification of the presheaf
	$F' \mapsto H^{q}(\alg{k}(F'), G)$,
	where $G$ is restricted to the small \'etale site of $\alg{k}(F')$.
	Let $F' \in F^{\perar}$ be a field
	and $\closure{F'}$ its algebraic closure.
	Then the direct limit of $\alg{k}(F'')$ over all finite subextensions $F''$ of $\closure{F'} / F'$
	is a henselian discrete valuation field with residue field $\closure{F'}$.
	Such a field has cohomological dimension $1$.
	This implies the result.
\end{proof}

\begin{Prop}
	Let $F' \in F^{\perar}$.
	Then the base change functor defines a morphism of sites
	$\Spec \alg{k}(F')_{\RPS} \to \Spec k_{\RPS}$.
\end{Prop}

\begin{proof}
	The only point to check is the exactness of the pullback functor
	for sheaves of sets.
	It is enough to show that $\alg{k}(F')$ is a filtered direct limit of
	relatively perfectly smooth $k$-algebras.
	We may assume that $F'$ is a field.
	The ring $F' \ctensor_{F} \Order_{k}$ is regular over $\Order_{k}$
	by \cite[Section 3.6, Lemma 2]{BLR90}.
	Hence $\alg{k}(F')$ is a filtered direct limit of smooth $k$-algebras
	by the paragraph before \cite[Section 3.6, Lemma 2]{BLR90}.
	Applying the relative perfections,
	we know that $\alg{k}(F')$ is a filtered direct limit of relatively perfectly smooth $k$-algebras
	since $\alg{k}(F')$ is relatively perfect over $k$.
\end{proof}

\begin{Prop} \label{0013}
	Let $q, s, r \in \Z$ and $n \ge 1$.
	Let $G \in \Ab(\alg{k}_{\et})$ be either $R^{q} \Psi \Lambda_{n}(r)$,
	$H^{q} i^{\ast} \mathfrak{T}_{n}(r)$ or
	$H^{q} R i^{!} \mathfrak{T}_{n}(r)$.
	Then $R^{s} \pi_{\alg{k}, \ast} G \in \mathcal{W}_{F}$.
\end{Prop}

\begin{proof}
	Let $G' \in \Alg^{\RP}_{u} / k$ be either
	$R^{q} \Psi_{\RPS} \Lambda_{n}(r)$, $H^{q} i_{\RPS}^{\ast} \mathfrak{T}_{n}(r)$
	or $H^{q} R i_{\RPS}^{!} \mathfrak{T}_{n}(r)$, correspondingly.
	By Proposition \ref{0010},
	it is enough to show that
	$R^{s} \pi_{\alg{k}, \ast} G \cong R^{s} \pi_{k, \RP, \ast} G'$.
	They are the sheafifications of the presheaves
		\[
				F'
			\mapsto
					H^{s}(\alg{k}(F'), G),
				\quad
					H^{s}(\alg{k}(F'), G' \times_{k} \alg{k}(F')),
		\]
	where the $G$ in the first presheaf is the restriction of $G$ to the small \'etale site of $\alg{k}(F')$
	and the $G' \times_{k} \alg{k}(F')$ in the second presheaf is
	the fiber product of schemes ($G' \in \Alg^{\RP}_{u} / k$ being representable).
	Hence it is enough to show that
	the group scheme $G' \times_{k} \alg{k}(F')$ as a sheaf on $\alg{k}(F')_{\et}$
	is isomorphic to the restriction of $G$ to $\alg{k}(F')_{\et}$.
	We may assume that $F'$ is a field.
	The base change functors define morphisms of sites
		\[
				\Spec \alg{K}(F')_{\Et}
			\stackrel{j_{\RPS, F'}}{\to}
				\Spec \alg{O}_{K}(F')_{\RPS}
			\stackrel{i_{\RPS, F'}}{\gets}
				\Spec \alg{k}(F')_{\RPS}.
		\]
	Let $R \Psi_{\RPS, F'} = i_{\RPS, F'}^{\ast} R j_{\RPS, F', \ast}$.
	Let
		\[
			\begin{CD}
					\Spec \alg{K}(F')_{\Et}
				@> j_{\RPS, F'} >>
					\Spec \alg{O}_{K}(F')_{\RPS}
				@< i_{\RPS, F'} <<
					\Spec \alg{k}(F')_{\RPS}
				\\
				@V g_{K} VV
				@V g_{\Order_{K}} VV
				@V g_{k} VV
				\\
					\Spec K_{\Et}
				@> j_{\RPS} >>
					\Spec \Order_{K, \RPS}
				@< i_{\RPS} <<
					\Spec k_{\RPS}
			\end{CD}
		\]
	be the premorphisms of sites defined by base change functors.
	Let $G'' \in \Ab(\alg{k}(F')_{\RPS})$ be either
	$R^{q} \Psi_{\RPS, F'} \Lambda_{n}(r)$, $H^{q} i_{\RPS, F'}^{\ast} \mathfrak{T}_{n}(r)$
	or $H^{q} R i_{\RPS, F'}^{!} \mathfrak{T}_{n}(r)$, correspondingly.
	Then the restriction of $G''$ to $\alg{k}(F')_{\et}$
	is the restriction of $G$ to $\alg{k}(F')_{\et}$.
	Hence it is enough to show that $G' \times_{k} \alg{k}(F') \cong g_{k}^{\ast} G'$
	is isomorphic to $G''$.
	We have an isomorphism of functors
		\[
				g_{k}^{\ast} i_{\RPS}^{\ast}
			\cong
				i_{\RPS, F'}^{\ast} L g_{\Order_{K}}^{\ast}
			\colon
				D(\Order_{K, \RPS})
			\to
				D(\alg{k}(F')_{\RPS})
		\]
	and a natural morphism of functors
		\[
				L g_{\Order_{K}}^{\ast} R j_{\RPS, \ast}
			\to
				R j_{\RPS, F', \ast} g_{K}^{\ast}
			\colon
				D(K_{\Et})
			\to
				D(\alg{O}_{K}(F')_{\RPS})
		\]
	by \cite[Proposition 2.6]{Suz21Imp}.
	Combining them, we have a morphism of functors
		\[
				g_{k}^{\ast} R \Psi_{\RPS}
			\to
				R \Psi_{\RPS, F'} g_{K}^{\ast}
			\colon
				D(K_{\Et})
			\to
				D(\alg{k}(F')_{\RPS}).
		\]
	This induces a morphism
		\begin{equation} \label{0011}
				g_{k}^{\ast} R \Psi_{\RPS} \Lambda_{n}(r)
			\to
				R \Psi_{\RPS, F'} \Lambda_{n}(r).
		\end{equation}
	For $n = 1$, the explicit calculations of the graded pieces of
	$R^{q} \Psi_{\RPS} \Lambda(r)$ (\cite[(6.3.1), (6.3.2)]{Suz24}) when $\zeta_{p} \in K$
	shows that the morphism
		\[
				g_{k}^{\ast} R^{q} \Psi_{\RPS} \Lambda(r)
			\to
				R^{q} \Psi_{\RPS, F'} \Lambda(r)
		\]
	is an isomorphism (with $\zeta_{p} \in K$ or not) for all $q$.
	Hence \eqref{0011} is an isomorphism for general $n$.
	This settles the case
	$G' = R^{q} \Psi_{\RPS} \Lambda(r)$ and $G'' = R^{q} \Psi_{\RPS, F'} \Lambda(r)$.
	We have a commutative diagram
		\[
			\begin{CD}
					g_{k}^{\ast} R^{r} \Psi_{\RPS} \Lambda_{n}(r)
				@>>>
					g_{k}^{\ast} \nu_{n}(r - 1)
				\\ @VVV @VVV \\
					R^{r} \Psi_{\RPS, F'} \Lambda_{n}(r)
				@>>>
					\nu_{n}(r - 1)
			\end{CD}
		\]
	in $\Ab(\alg{k}(F')_{\RPS})$ by the construction of the Kato boundary map.
	The vertical morphisms are isomorphisms.
	Hence they induce an isomorphism of triangles
		\[
			\begin{CD}
					g_{k}^{\ast} i_{\RPS}^{\ast} \mathfrak{T}_{n}(r)
				@>>>
					g_{k}^{\ast} \tau_{\le r} R \Psi_{\RPS} \Lambda_{n}(r)
				@>>>
					g_{k}^{\ast} \nu_{n}(r - 1)
				\\ @| @| @| \\
					i_{\RPS, F'}^{\ast} \mathfrak{T}_{n}(r)
				@>>>
					\tau_{\le r} R \Psi_{\RPS, F'} \Lambda_{n}(r)
				@>>>
					\nu_{n}(r - 1)
			\end{CD}
		\]
	This settles the case
	$G' = H^{q} i_{\RPS}^{\ast} \mathfrak{T}_{n}(r)$ and
	$G'' = H^{q} i_{\RPS, F'}^{\ast} \mathfrak{T}_{n}(r)$.
	This in turn induces an isomorphism of triangles
		\[
			\begin{CD}
					g_{k}^{\ast} i_{\RPS}^{\ast} \mathfrak{T}_{n}(r)
				@>>>
					g_{k}^{\ast} R \Psi_{\RPS} \Lambda_{n}(r)
				@>>>
					g_{k}^{\ast} R i_{\RPS}^{!} \mathfrak{T}_{n}(r)[1]
				\\ @| @| @| \\
					i_{\RPS, F'}^{\ast} \mathfrak{T}_{n}(r)
				@>>>
					R \Psi_{\RPS, F'} \Lambda_{n}(r)
				@>>>
					R i_{\RPS, F'}^{!} \mathfrak{T}_{n}(r)[1],
			\end{CD}
		\]
	settling the remaining case.
\end{proof}

\begin{Prop} \label{0014}
	Let $q, r \in \Z$ and $n \ge 1$.
	Then the objects $R^{q} \pi_{\alg{K}, \ast} \Lambda_{n}(r)$,
	$R^{q} \pi_{\alg{O}_{K}, \ast} \mathfrak{T}_{n}(r)$ and
	$R^{q + 1} \pi_{\alg{O}_{K}, !} \mathfrak{T}_{n}(r)$ are in $\mathcal{W}_{F}$.
	They are zero unless $q = 0, 1, 2$.
\end{Prop}

\begin{proof}
	Let $G \in D(\alg{k}_{\et})$ be either $R \Psi \Lambda_{n}(r)$,
	$i^{\ast} \mathfrak{T}_{n}(r)$ or
	$R i^{!} \mathfrak{T}_{n}(r)$.
	By Proposition \ref{0012},
	we have an exact sequence
		\[
				0
			\to
				R^{1} \pi_{\alg{k}, \ast} H^{q - 1} G
			\to
				R^{q} \pi_{\alg{k}, \ast} G
			\to
				\pi_{\alg{k}, \ast} H^{q} G
			\to
				0
		\]
	in $\Ab(F^{\perar}_{\et})$.
	The first and third terms are in $\mathcal{W}_{F}$
	by Proposition \ref{0013}.
	Hence so is the second.
	The vanishing for $q \ne 0, 1, 2$ follows from
	\cite[Propositions 6.2.1 and 6.5.1]{Suz24}.
\end{proof}

\begin{Prop} \label{0016}
	Let $r, q, \in \Z$ and $n \ge 1$.
	\begin{enumerate}
		\item
			The perfect pairing
				\[
							R \pi_{\alg{K}, \ast} \Lambda_{n}(r)
						\tensor^{L}
							R \pi_{\alg{K}, \ast} \Lambda_{n}(2 - r)
					\to
						\Lambda_{\infty}[-2],
				\]
			in \cite[Proposition 6.2.2]{Suz24} induces a Pontryagin duality
				\[
						\pi_{0}(R^{q} \pi_{\alg{K}, \ast} \Lambda_{n}(r))
					\leftrightarrow
						\pi_{0}(R^{2 - q} \pi_{\alg{K}, \ast} \Lambda_{n}(2 - r))
				\]
			of finite \'etale group schemes over $F$ and a Serre duality
				\[
						(R^{q} \pi_{\alg{K}, \ast} \Lambda_{n}(r))^{0}
					\leftrightarrow
						(R^{3 - q} \pi_{\alg{K}, \ast} \Lambda_{n}(2 - r))^{0}
				\]
			of connected groups in $\mathcal{W}_{F}$.
		\item
			The perfect pairing
				\[
							R \pi_{\alg{O}_{K}, \ast} \mathfrak{T}_{n}(r)
						\tensor^{L}
							R \pi_{\alg{O}_{K}, !} \mathfrak{T}_{n}(2 - r)
					\to
						\Lambda_{\infty}[-3],
				\]
			in \cite[Proposition 6.5.1]{Suz24} induces a Pontryagin duality
				\[
						\pi_{0}(R^{q} \pi_{\alg{O}_{K}, \ast} \mathfrak{T}_{n}(r))
					\leftrightarrow
						\pi_{0}(R^{3 - q} \pi_{\alg{O}_{K}, !} \mathfrak{T}_{n}(2 - r))
				\]
			of finite \'etale group schemes over $F$ and a Serre duality
				\[
						(R^{q} \pi_{\alg{O}_{K}, \ast} \mathfrak{T}_{n}(r))^{0}
					\leftrightarrow
						(R^{4 - q} \pi_{\alg{O}_{K}, !} \mathfrak{T}_{n}(r))^{0}
				\]
			of connected groups in $\mathcal{W}_{F}$.
	\end{enumerate}
\end{Prop}

\begin{proof}
	This follows from Propositions \ref{0078} and \ref{0014}.
\end{proof}

\begin{Prop} \label{0030}
	Let $n \ge 1$.
	\begin{enumerate}
		\item \label{0025}
			$\pi_{\alg{O}_{K}, \ast} \Lambda_{n} \cong \Lambda_{n}$ is finite,
			$R^{1} \pi_{\alg{O}_{K}, \ast} \Lambda_{n}$ is connected ind-algebraic and
			$R^{2} \pi_{\alg{O}_{K}, \ast} \Lambda_{n} = 0$.
		\item \label{0026}
			$\pi_{\alg{O}_{K}, \ast} \mathfrak{T}_{n}(1)$ is finite and
			$R^{2} \pi_{\alg{O}_{K}, \ast} \mathfrak{T}_{n}(1) = 0$.
		\item \label{0074}
			$\pi_{\alg{O}_{K}, \ast} \mathfrak{T}_{n}(2)$ is finite.
		\item \label{0027}
			$R^{1} \pi_{\alg{O}_{K}, !} \mathfrak{T}_{n}(r) = 0$ for all $r$.
		\item \label{0029}
			$R^{2} \pi_{\alg{O}_{K}, !} \mathfrak{T}_{n}(1)$ is finite.
		\item \label{0028}
			$R^{2} \pi_{\alg{O}_{K}, !} \mathfrak{T}_{n}(2) = 0$ and
			$R^{3} \pi_{\alg{O}_{K}, \ast} \mathfrak{T}_{n}(2)$ is pro-algebraic.
	\end{enumerate}
\end{Prop}

\begin{proof}
	\eqref{0025}, \eqref{0026} and \eqref{0074} follow by applying \'etale sheafification to
	the descriptions of the Zariski sheaves in the proof of \cite[Proposition 6.5.2]{Suz24}.
	We have $\pi_{\alg{O}_{K}, \ast} \mathfrak{T}_{n}(r) \isomto \pi_{\alg{K}, \ast} \Lambda_{n}(r)$
	for all $r$.
	The morphisms
	$R^{q} \pi_{\alg{O}_{K}, \ast} \mathfrak{T}_{n}(r) \to R^{q} \pi_{\alg{K}, \ast} \Lambda_{n}(r)$
	are injective for all $q$ and $r$
	by the injectivity stated in the last paragraph of the proof of \cite[Proposition 6.5.2]{Suz24}.
	These imply \eqref{0027}.
	The rest follow from Proposition \ref{0016} by duality.
\end{proof}

Recall from \cite[paragraph after Proposition 4.3.2]{Suz24}
that we define
	\[
			R \alg{\Gamma}(\alg{K}, \var)
		=
			\algebrize R \pi_{\alg{K}, \ast}
		\colon
			D(\alg{K}_{\et})
		\to
			D(F^{\ind\rat}_{\pro\et}),
	\]
	\[
			R \alg{\Gamma}(\alg{O}_{K}, \var)
		=
			\algebrize R \pi_{\alg{O}_{K}, \ast}
		\colon
			D(\alg{O}_{K, \et})
		\to
			D(F^{\ind\rat}_{\pro\et}),
	\]
	\[
			R \alg{\Gamma}_{c}(\alg{O}_{K}, \var)
		=
			\algebrize R \pi_{\alg{O}_{K}, !}
		\colon
			D(\alg{O}_{K, \et})
		\to
			D(F^{\ind\rat}_{\pro\et})
	\]
and $\alg{H}^{q} = H^{q} R \alg{\Gamma}$, $\alg{H}^{q}_{c} = H^{q} R \alg{\Gamma}_{c}$.
If $G \in D^{b}(\alg{K}_{\et})$ or $D^{b}(\alg{O}_{K, \et})$ is concentrated in non-negative degrees
and $R^{q} \pi_{\alg{K}, \ast} G$ is in $\mathcal{W}_{F}$ for all $q$ and zero for large enough $q$,
then we occasionally write $\alg{\Gamma}(\alg{K}, G) = \alg{H}^{0}(\alg{K}, G)$
or $\alg{\Gamma}(\alg{O}_{K}, G) = \alg{H}^{0}(\alg{O}_{K}, G)$, respectively.

\begin{Prop} \label{0031}
	Let $q, r \in \Z$ and $n \ge 1$.
	\begin{enumerate}
		\item \label{0032}
			$\alg{H}^{q}(\alg{K}, \Lambda_{n}(r))$,
			$\alg{H}^{q}(\alg{O}_{K}, \mathfrak{T}_{n}(r))$ and
			$\alg{H}^{q}_{c}(\alg{O}_{K}, \mathfrak{T}_{n}(r))$ are
			in $\mathcal{W}_{F}$ for all $q$ and $r$.
		\item \label{0033}
			$\alg{\Gamma}(\alg{K}, \Lambda_{n}(r))$ is finite and
			$\alg{H}^{q}(\alg{K}, \Lambda_{n}(r)) = 0$ for all $q \ne 0, 1, 2$ and all $r$.
		\item \label{0034}
			$\alg{\Gamma}(\alg{O}_{K}, \Lambda_{n}) \cong \Lambda_{n}$ is finite,
			$\alg{H}^{1}(\alg{O}_{K}, \Lambda_{n})$ is connected ind-algebraic and
			$\alg{H}^{q}(\alg{O}_{K}, \Lambda_{n}) = 0$ for all $q \ne 0, 1$.
		\item \label{0035}
			$\alg{H}^{q}_{c}(\alg{O}_{K}, \Lambda_{n}) = 0$ for all $q \ne 2, 3$.
		\item \label{0036}
			$\alg{\Gamma}(\alg{O}_{K}, \mathfrak{T}_{n}(1))$ is finite and
			$\alg{H}^{q}(\alg{O}_{K}, \mathfrak{T}_{n}(1)) = 0$ for all $q \ne 0, 1$.
		\item \label{0037}
			$\alg{H}^{2}_{c}(\alg{O}_{K}, \mathfrak{T}_{n}(1))$ is finite and
			$\alg{H}^{q}_{c}(\alg{O}_{K}, \mathfrak{T}_{n}(1)) = 0$ for all $q \ne 2, 3$.
		\item \label{0038}
			$\alg{\Gamma}(\alg{O}_{K}, \mathfrak{T}_{n}(2))$ is finite and
			$\alg{H}^{q}(\alg{O}_{K}, \mathfrak{T}_{n}(2)) = 0$ for all $q \ne 0, 1, 2$.
		\item \label{0039}
			$\alg{H}^{3}_{c}(\alg{O}_{K}, \mathfrak{T}_{n}(2))$ is pro-algebraic and
			$\alg{H}^{q}_{c}(\alg{O}_{K}, \mathfrak{T}_{n}(2)) = 0$ for all $q \ne 3$.
	\end{enumerate}
\end{Prop}

\begin{proof}
	This follows from Propositions \ref{0015}, \ref{0014} and \ref{0030}.
\end{proof}

\begin{Prop} \label{0092}
	Let $r, q, \in \Z$ and $n \ge 1$.
	\begin{enumerate}
		\item
			The perfect pairing
				\[
							R \alg{\Gamma}(\alg{K}, \Lambda_{n}(r))
						\tensor^{L}
							R \alg{\Gamma}(\alg{K}, \Lambda_{n}(2 - r))
					\to
						\Lambda_{\infty}[-2],
				\]
			in \cite[Proposition 6.2.2]{Suz24} induces a Pontryagin duality
				\[
						\pi_{0}(\alg{H}^{q}(\alg{K}, \Lambda_{n}(r)))
					\leftrightarrow
						\pi_{0}(\alg{H}^{2 - q}(\alg{K}, \Lambda_{n}(2 - r)))
				\]
			of finite \'etale group schemes over $F$ and a Serre duality
				\[
						\alg{H}^{q}(\alg{K}, \Lambda_{n}(r))^{0}
					\leftrightarrow
						\alg{H}^{3 - q}(\alg{K}, \Lambda_{n}(2 - r))^{0}
				\]
			of connected groups in $\mathcal{W}_{F}$.
		\item
			The perfect pairing
				\[
							R \alg{\Gamma}(\alg{O}_{K}, \mathfrak{T}_{n}(r))
						\tensor^{L}
							R \alg{\Gamma}_{c}(\alg{O}_{K}, \mathfrak{T}_{n}(2 - r))
					\to
						\Lambda_{\infty}[-3]
				\]
			in \cite[Proposition 6.5.1]{Suz24} induces a Pontryagin duality
				\[
						\pi_{0}(\alg{H}^{q}(\alg{O}_{K}, \mathfrak{T}_{n}(r)))
					\leftrightarrow
						\pi_{0}(\alg{H}^{3 - q}_{c}(\alg{O}_{K}, \mathfrak{T}_{n}(2 - r)))
				\]
			of finite \'etale group schemes over $F$ and a Serre duality
				\[
						\alg{H}^{q}(\alg{O}_{K}, \mathfrak{T}_{n}(r))^{0}
					\leftrightarrow
						\alg{H}^{4 - q}_{c}(\alg{O}_{K}, \mathfrak{T}_{n}(2 - r))^{0}
				\]
			of connected groups in $\mathcal{W}_{F}$.
	\end{enumerate}
\end{Prop}

\begin{proof}
	This follows from Proposition \ref{0016}.
\end{proof}

Here is a summary of the above finiteness results:

	\[
		\begin{array}{c|c}
				\alg{H}^{q}(\alg{K}, \Lambda_{n}(r))
			&
				\forall r
			\\ \hline
				q = 0
			&
				\text{finite}
			\\
				q = 1
			&
				\text{general}
			\\
				q = 2
			&
				\text{general}
		\end{array}
	\]
	\[
		\begin{array}{c|ccc}
				\alg{H}^{q}(\alg{O}_{K}, \mathfrak{T}_{n}(r))
			&
				r = 0
			&
				r = 1
			&
				r = 2
			\\ \hline
				q = 0
			&
				\Lambda_{n}
			&
				\text{finite}
			&
				\text{finite}
			\\
				q = 1
			&
				\text{ind-alg}
			&
				\text{general}
			&
				\text{general}
			\\
				q = 2
			&
				0
			&
				0
			&
				\text{general}
		\end{array}
	\]
	\[
		\begin{array}{c|ccc}
				\alg{H}^{q}_{c}(\alg{O}_{K}, \mathfrak{T}_{n}(r))
			&
				r = 0
			&
				r = 1
			&
				r = 2
			\\ \hline
				q = 2
			&
				\text{general}
			&
				\text{finite}
			&
				0
			\\
				q = 3
			&
				\text{general}
			&
				\text{general}
			&
				\text{pro-alg}
		\end{array}
	\]
These groups are zero for $(q, r)$ outside the ranges in the tables.
By Proposition \ref{0092} (or by a direct calculation),
we have
	\begin{equation} \label{0093}
			\pi_{0} \alg{H}^{3}_{c}(\alg{O}_{K}, \mathfrak{T}_{n}(2))
		\cong
			\Lambda_{n}.
	\end{equation}


\section{Curves and their tubular neighborhoods}
\label{0059}

In this section, we show that
the ind-pro-algebraic group structures on cohomology of smooth affine curves over $F$
and their $p$-adic tubular neighborhoods
constructed in \cite[Section 9]{Suz24} belong to $\mathcal{W}_{F}$.

Let $V = \Spec B$ be a smooth affine geometrically connected curve over $F$.
Let $Y$ be the smooth compactification of $V$ and set $T = Y \setminus V$.
We briefly recall the constructions in \cite[Section 9.1]{Suz24}.
For $F' \in F^{\perar}$, we set $\alg{B}(F') = F' \tensor_{F} B$.
The functor $\alg{B}$ defines a site $\Spec \alg{B}_{\et}$
and a morphism of sites
	\[
			\pi_{\alg{B}}
		\colon
			\Spec \alg{B}_{\et}
		\to
			\Spec F^{\perar}_{\et}
	\]
by the method in Section \ref{0070}.
That is, $\Spec \alg{B}_{\et}$ is the category of pairs $(B', F')$
(where $F' \in F^{\perar}$ and $B'$ an \'etale $\alg{B}(F')$-algebra)
endowed with the \'etale topology.
For $x \in T$, let $\Hat{\Order}_{k_{x}}$ be the complete local ring of $Y$ at $x$.
Let $\Hat{k}_{x}$ and $F_{x}$ be its fraction field and residue field, respectively.
As in Section \ref{0070},
we have a functor
	\[
			\Hat{\alg{k}}_{x, 0}(F'_{x})
		=
			(F'_{x} \ctensor_{F_{x}} \Hat{\Order}_{k_{x}}) \tensor_{\Hat{\Order}_{k_{x}}} \Hat{k}_{x}
	\]
in $F'_{x} \in F_{x}^{\perar}$.
Let $\Hat{\alg{k}}_{x}$ be the Weil restriction $\Weil_{F_{x} / F} \Hat{\alg{k}}_{x, 0}$,
so
	\[
			\Hat{\alg{k}}_{x}(F')
		=
			\Hat{\alg{k}}_{x, 0}(F' \tensor_{F} F_{x})
		=
			(F' \ctensor_{F} \Hat{\Order}_{k_{x}}) \tensor_{\Hat{\Order}_{k_{x}}} \Hat{k}_{x}
	\]
for $F' \in F^{\perar}$.
By the method in Section \ref{0070},
the functor $\Hat{\alg{k}}_{x}$ defines a site $\Spec \Hat{\alg{k}}_{x}$ and a morphism of sites
	\[
			\pi_{\Hat{\alg{k}}_{x}}
		\colon
			\Spec \Hat{\alg{k}}_{x}
		\to
			\Spec F^{\perar}_{\et}.
	\]
Let
	\[
			\pi_{\alg{k}_{x} / \alg{B}}
		\colon
			\Spec \Hat{\alg{k}}_{x, \et}
		\to
			\Spec \alg{B}_{\et}
	\]
be the morphism defined by the functor $(B', F') \mapsto (B' \tensor_{\alg{B}(F')} \Hat{\alg{k}}_{x}(F'), F')$
(where $F' \in F^{\perar}$ and $B'$ an \'etale $\alg{B}(F')$-algebra).

A $B$-module $M$ can naturally be viewed as a sheaf
$(B', F') \mapsto B' \tensor_{B} M$ on $\Spec \alg{B}_{\et}$.
A similar process exists for viewing a $\Hat{k}_{x}$-module as a sheaf on $\Spec \Hat{\alg{k}}_{x, \et}$.
Hence the sheaf $\nu(1)$ (the dlog part of $\Omega^{1}$) can also be viewed
as a sheaf on $\Spec \alg{B}_{\et}$ or $\Spec \Hat{\alg{k}}_{x, \et}$.

Let $G \in \Ab(\alg{B}_{\et})$ and $G_{x} \in \Ab(\Hat{\alg{k}}_{x})$ for each $x \in T$ be sheaves
and $G \to \pi_{\Hat{\alg{k}}_{x} / \alg{B}, \ast} G_{x}$ a morphism in $\Ab(\alg{B}_{\et})$.
The latter induces a morphism
	\[
			R \pi_{\alg{B}, \ast} G
		\to
			R \pi_{\Hat{\alg{k}}_{x}, \ast} G_{x}
	\]
in $D(F^{\perar}_{\et})$.
Then the object $R \pi_{\alg{B}, \Hat{!}} G \in D(F^{\perar}_{\et})$ in \cite[Section 9.1]{Suz24}
is a canonical object fitting in a distinguished triangle%
\footnote{
	Actually it is denoted as $R \Bar{\pi}_{\alg{B}, \Hat{!}} G$ in \cite[Section 9.1]{Suz24}.
	The bar in $\Bar{\pi}$ indicates it is defined by ``fibered sites''.
	Since we do not get into the details of fibered sites in this paper,
	we will not need or use this notation.
	We have $R \pi_{\alg{B}, \ast} G \cong R \Bar{\pi}_{\alg{B}, \ast} G$
	and $R \pi_{\Hat{\alg{k}}_{x}, \ast} G_{x} \cong R \Bar{\pi}_{\Hat{\alg{k}}_{x}, \ast} G_{x}$
	by construction.
}
	\[
			R \pi_{\alg{B}, \Hat{!}} G
		\to
			R \pi_{\alg{B}, \ast} G
		\to
			\bigoplus_{x \in T}
				R \pi_{\Hat{\alg{k}}_{x}, \ast} G_{x}.
	\]
For $q \in \Z$, denote $R^{q} \pi_{\alg{B}, \Hat{!}} G = H^{q} R \pi_{\alg{B}, \Hat{!}} G$.
For $G = \Lambda$ or $\nu(1)$, we take $G_{x}$ to be its natural counterpart over $\Spec \Hat{\alg{k}}_{x}$.
For $G = M$ a $B$-module, we take $G_{x} = M \tensor_{B} \Hat{k}_{x}$.
The morphism $G \to \pi_{\Hat{\alg{k}}_{x} / \alg{B}, \ast} G_{x}$ is the natural one.

\begin{Prop} \label{0017}
	Let $q \in \Z$.
	Let $G$ be either
	$\Lambda$, $\nu(1)$ or a finite projective $B$-module.
	Then $R^{q} \pi_{\alg{B}, \ast} G \in \mathcal{W}_{F}$ is ind-algebraic
	and $R^{q + 1} \pi_{\alg{B}, \Hat{!}} G \in \mathcal{W}_{F}$ is pro-algebraic.
	They are zero unless $q = 0, 1$.
\end{Prop}

\begin{proof}
	This follows from \cite[Proposition 9.1.5]{Suz24} and its proof.
\end{proof}

Next, let $A$ be a ring.
Assume all of the following:
\begin{enumerate}
	\item
		$A$ is a two-dimensional regular integral domain.
	\item
		$A$ contains a primitive $p$-th root of unity $\zeta_{p}$.
	\item
		The radical $I$ of the ideal $(p)$ of $A$ is principal.
	\item
		$B := A / I$ is a one-dimensional geometrically connected smooth algebra over $F$.
	\item
		The pair $(A, I)$ is complete.
\end{enumerate}
Set $R = A[1 / p]$.
We apply the notation above for $V = \Spec B$.
We briefly recall the constructions in \cite[Section 9.2]{Suz24}.
For $F' \in F^{\perar}$,
let $\Hat{\alg{A}}(F')$ be the Kato canonical lifting
of the relatively perfect $B$-algebra $\alg{B}(F')$ to $A$.
Set $\Hat{\alg{R}}(F') = \Hat{\alg{A}}(F') \tensor_{A} R$.
This functor defines a site $\Spec \Hat{\alg{R}}_{\et}$ 
and a morphism of sites $\pi_{\Hat{\alg{R}}} \colon \Spec \Hat{\alg{R}}_{\et} \to \Spec F^{\perar}_{\et}$
by the method in Section \ref{0070}.
For $x \in T$, Let $\Hat{\Order}_{\eta_{x}}$ be the the Kato canonical lifting
of the relatively perfect $B$-algebra $\Hat{k}_{x}$ to $A$.
Let $\Hat{K}_{\eta_{x}}$ be its fraction field,
which is a complete discrete valuation field with residue field $\Hat{k}_{x}$.
Let $\Hat{\alg{K}}_{\eta_{x}, 0}(F'_{x})$ be the functor in $F'_{x} \in F_{x}^{\perar}$
defined in Section \ref{0070} for $\Hat{K}_{\eta_{x}}$.
Let $\Hat{\alg{K}}_{\eta_{x}} = \Weil_{F_{x} / F} \Hat{\alg{K}}_{\eta_{x}, 0}$.
Then we have a site a site $\Spec \Hat{\alg{K}}_{\eta_{x}, \et}$
and a morphism of sites
$\pi_{\Hat{\alg{K}}_{\eta_{x}}} \colon \Spec \Hat{\alg{K}}_{\eta_{x}, \et} \to \Spec F^{\perar}_{\et}$
by the method in Section \ref{0070}.
The object $R \pi_{\Hat{\alg{K}}_{\eta_{x}}, \ast} \Lambda \in D(F^{\perar}_{\et})$
is the Weil restriction $\Weil_{F_{x} / F}$ of
$R \pi_{\Hat{\alg{K}}_{\eta_{x}, 0}, \ast} \Lambda \in D(F^{\perar}_{x, \et})$.
We have a canonical morphism
	\[
			R \pi_{\Hat{\alg{R}}, \ast} \Lambda
		\to
			R \pi_{\Hat{\alg{K}}_{\eta_{x}}, \ast} \Lambda
	\]
in $D(F^{\perar}_{\et})$.
Then the object $R \pi_{\Hat{\alg{R}}, \Hat{!}} \Lambda \in D(F^{\perar}_{\et})$ in \cite[Section 9.2]{Suz24}
is a canonical object fitting in a distinguished triangle
	\[
			R \pi_{\Hat{\alg{R}}, \Hat{!}} \Lambda
		\to
			R \pi_{\Hat{\alg{R}}, \ast} \Lambda
		\to
			\bigoplus_{x \in T}
				R \pi_{\Hat{\alg{K}}_{\eta_{x}}, \ast} \Lambda.
	\]
For $q \in \Z$, denote
$R^{q} \pi_{\Hat{\alg{R}}, \Hat{!}} \Lambda = H^{q} R \pi_{\Hat{\alg{R}}, \Hat{!}} \Lambda$.

\begin{Prop} \label{0018}
	Let $q \in \Z$.
	Then $R^{q} \pi_{\Hat{\alg{R}}, \ast} \Lambda \in \mathcal{W}_{F}$ is ind-algebraic
	and $R^{q + 1} \pi_{\Hat{\alg{R}}, \Hat{!}} \Lambda \in \mathcal{W}_{F}$ is pro-algebraic.
	They are zero unless $q = 0, 1, 2$.
\end{Prop}

\begin{proof}
	By \cite[Section 9.3, Proposition 9.2.4, Propositions 9.3.1--3]{Suz24},
	there are spectral sequences
		\[
				E_{2}^{i j}
			=
				R^{i} \pi_{\alg{B}, \ast} H^{j} \mathcal{E}
			\Longrightarrow
				R^{i + j} \pi_{\Hat{\alg{R}}, \ast} \Lambda,
		\]
		\[
				E_{2}^{i j}
			=
				R^{i} \pi_{\alg{B}, !} H^{j} \mathcal{E}
			\Longrightarrow
				R^{i + j} \pi_{\Hat{\alg{R}}, !} \Lambda,
		\]
	where $H^{j} \mathcal{E} \in \Ab(\alg{B}_{\et})$
	(and, implicitly, $H^{j} \mathcal{E}_{x} \in \Ab(\Hat{\alg{k}}_{x, \et})$
	and a morphism $H^{j} \mathcal{E} \to \pi_{\Hat{\alg{k}}_{x} / \alg{B}, \ast} H^{j} \mathcal{E}_{x}$)
	is a certain sheaf admitting a finite filtration with graded pieces
	isomorphic to $\Lambda$, $\nu(1)$ or a finite projective $B$-module
	(and this filtration is compatible at all $x \in T$).
	We have $R^{3} \pi_{\Hat{\alg{R}}, \ast} \Lambda = 0$
	by \cite[Proposition 9.5.3]{Suz24}.
	Also $R^{3} \pi_{\Hat{\alg{K}}_{\eta_{x}}, \ast} \Lambda = 0$
	by Proposition \ref{0014}.
	With Proposition \ref{0012},
	the result follows by writing down the above (degenerate) spectral sequences
	and applying Propositions \ref{0017} and \ref{0014}.
\end{proof}

Set
	\[
			R \alg{\Gamma}(\Hat{\alg{R}}, \Lambda)
		=
			\algebrize R \pi_{\Hat{\alg{R}}, \Lambda} \Lambda
		\in
			D(F^{\ind\rat}_{\pro\et}),
	\]
	\[
			R \alg{\Gamma}_{c}(\Hat{\alg{R}}, \Lambda)
		=
			\algebrize R \pi_{\Hat{\alg{R}}, \Hat{!}} \Lambda
		\in
			D(F^{\ind\rat}_{\pro\et}),
	\]
and $\alg{H}^{q} = H^{q} R \alg{\Gamma}$, $\alg{H}^{q}_{c} = H^{q} R \alg{\Gamma}_{c}$.
The object $R \alg{\Gamma}(\Hat{\alg{K}}_{\eta_{x}}, \Lambda) = \algebrize R \pi_{\Hat{\alg{K}}_{\eta_{x}}} \Lambda$
of $D(F^{\ind\rat}_{\pro\et})$ is the Weil restriction $\Weil_{F_{x} / F}$ of
$R \alg{\Gamma}(\Hat{\alg{K}}_{\eta_{x}, 0}, \Lambda) \in D(F^{\ind\rat}_{x, \pro\et})$.

\begin{Prop} \label{0020}
	Let $q \in \Z$.
	Then $\alg{H}^{q}(\Hat{\alg{R}}, \Lambda) \in \mathcal{W}_{F}$ is ind-algebraic
	and $\alg{H}^{q + 1}_{c}(\Hat{\alg{R}}, \Lambda) \in \mathcal{W}_{F}$ is pro-algebraic.
	They are zero unless $q = 0, 1, 2$.
\end{Prop}

\begin{proof}
	This follows from Propositions \ref{0015} and \ref{0018}.
\end{proof}


\section{Two-dimensional local rings: preliminaries}
\label{0086}

For the rest of the paper,
let $A$ be a two-dimensional normal noetherian complete local ring
of mixed characteristic with residue field $F$.

In this section, we first show that
the ind-pro-algebraic group structures on cohomology of punctured spectra of $A$
constructed in \cite{Suz24} belong to $\mathcal{W}_{F}$
if $A$ is ``enough resolved''.
We then show a slightly weaker property for a general $A$
as a preparation for the next section.
This is done by taking a resolution of singularities of $A$
and combining the statements already proved for
tubular neighborhoods of the smooth part of the exceptional divisor (Section \ref{0059})
and the local rings of the resolution at closed points.

Let $K$ be the fraction field of $A$.
Let $X = \Spec A \setminus \{\ideal{m}\}$,
where $\ideal{m}$ is the maximal ideal of $A$.
Let $P$ be the set of height one primes ideals of $A$.
For $\ideal{p} \in P$, let $\Hat{A}_{\ideal{p}}$ be the complete local ring of $A$ at $\ideal{p}$
and $\Hat{K}_{\ideal{p}}$ and $\kappa(\ideal{p})$ its fraction field and residue field, respectively.
Let $F_{\ideal{p}}$ be the residue field of $\kappa(\ideal{p})$.
Let $S \subset P$ be a finite subset.
Set $U_{S} = X \setminus S$.

We briefly recall the constructions in \cite[Section 10.1]{Suz24}.
The ring $A$ has a canonical $W(F)$-algebra structure
(\cite[Chapter V, Section 4, Theorem 2.1]{DG70b}).
For $F' \in F^{\perar}$, let
	\[
			\alg{A}(F')
		=
			W(F') \ctensor_{W(F)} A
		=
			\invlim_{n} (W(F') \tensor_{W(F)} A / \ideal{m}^{n}),
	\]
	\[
			\alg{U}_{S}(F')
		=
			U_{S} \times_{\Spec A} \Spec \alg{A}(F').
	\]
Let $\alg{U}_{S, \et}$ be the category of pairs $(U', F')$
(where $F' \in F^{\perar}$ and $U'$ an \'etale $\alg{U}_{S}(F')$-scheme)
endowed with the \'etale topology.
The functor $F' \mapsto (\alg{U}_{S}(F'), F')$ defines a morphism of sites
	\[
			\pi_{\alg{U}_{S}}
		\colon
			\alg{U}_{S, \et}
		\to
			\Spec F^{\perar}_{\et}.
	\]
Let $\lambda_{S} \colon \alg{U}_{S, \et} \to \alg{X}_{\et}$ be the morphism defined by
$(X', F') \mapsto (X' \times_{\alg{X}(F')} \alg{U}_{S}(F'), F')$.
Let $R \pi_{\alg{U}_{S}, !} = R \pi_{\alg{X}, \ast} \lambda_{S, !}$.
Define
	\[
			R \alg{\Gamma}(\alg{U}_{S}, \var)
		=
			\algebrize R \pi_{\alg{U}_{S}, \ast}
		\colon
			D(\alg{U}_{S, \et})
		\to
			D(F^{\ind\rat}_{\pro\et}),
	\]
	\[
			R \alg{\Gamma}_{c}(\alg{U}_{S}, \var)
		=
			\algebrize R \pi_{\alg{U}_{S}, !}
		\colon
			D(\alg{U}_{S, \et})
		\to
			D(F^{\ind\rat}_{\pro\et})
	\]
and $\alg{H}^{q} = H^{q} R \alg{\Gamma}$, $\alg{H}^{q}_{c} = H^{q} R \alg{\Gamma}_{c}$.
If $G \in D^{b}(\alg{U}_{S, \et})$ is concentrated in non-negative degrees
and $R^{q} \pi_{\alg{U}_{S}, \ast} G$ is in $\mathcal{W}_{F}$ for all $q$ and zero for large enough $q$,
then we occasionally write $\alg{\Gamma}(\alg{U}_{S}, G) = \alg{H}^{0}(\alg{U}_{S}, G)$.
If $U_{S} = \Spec R$ is affine,
then $\pi_{\alg{U}_{S}}$, $R \alg{\Gamma}(\alg{U}_{S}, \var)$ and $R \alg{\Gamma}_{c}(\alg{U}_{S}, \var)$
are also denoted as $\pi_{\alg{R}}$, $R \alg{\Gamma}(\alg{R}, \var)$ and $R \alg{\Gamma}_{c}(\alg{R}, \var)$, respectively.

For $\ideal{p} \in P$ and $F' \in F^{\perar}$,
let $\Hat{\alg{A}}_{\ideal{p}}(F')$ be the ring $\alg{A}(F') \tensor_{A} \Hat{A}_{\ideal{p}}$
completed with respect to the ideal $\alg{A}(F') \tensor_{A} \ideal{p} \Hat{A}_{\ideal{p}}$.
Let
	\[
			\varalg{\kappa}(\ideal{p})(F')
		=
				\Hat{\alg{A}}_{\ideal{p}}(F')
			\tensor_{\Hat{A}_{\ideal{p}}}
				\kappa(\ideal{p}),
	\]
	\[
			\Hat{\alg{K}}_{\ideal{p}}(F')
		=
				\Hat{\alg{A}}_{\ideal{p}}(F')
			\tensor_{\Hat{A}_{\ideal{p}}}
				\Hat{K}_{\ideal{p}}.
	\]
If $\ideal{p}$ divides $p$, then these functors are the Weil restrictions $\Weil_{F_{\ideal{p}} / F}$
of the functors defined in Section \ref{0070} for $\Hat{K}_{\ideal{p}}$.
We have sites and morphisms of sites
	\[
		\begin{CD}
				\Spec \Hat{\alg{K}}_{\ideal{p}, \et}
			@> \lambda_{\ideal{p}} >>
				\Spec \Hat{\alg{A}}_{\ideal{p}, \et}
			@< i_{\ideal{p}} <<
				\Spec \varalg{\kappa}(\ideal{p})_{\et}
			\\
			@V \pi_{\Hat{\alg{K}}_{\ideal{p}}} VV
			@V \pi_{\Hat{\alg{A}}_{\ideal{p}}} VV
			@V \pi_{\varalg{\kappa}(\ideal{p})} VV
			\\
				\Spec F^{\perar}_{\et}
			@=
				\Spec F^{\perar}_{\et}
			@=
				\Spec F^{\perar}_{\et}
		\end{CD}
	\]
as in Section \ref{0070}.
Let $R \alg{\Gamma}(\Hat{\alg{K}}_{\ideal{p}}, \var) = \algebrize R \pi_{\Hat{\alg{K}}_{\ideal{p}}, \ast}$ and so on.

For any $n \ge 1$, $r \ge 0$ and $\ideal{p} \in P$ dividing $p$,
we have the sheaf $\nu_{n}(r)$ on $\Spec \varalg{\kappa}(\ideal{p})_{\et}$.
Setting $S$ to be the set of primes dividing $p$,
we have a morphism
	\[
			\tau_{\le r}
			R \lambda_{S, \ast} \Lambda_{n}(r)
		\to
			i_{\ideal{p}, \ast} \nu_{n}(r - 1)
	\]
in $D(\alg{X}_{\et})$
by the restriction of the morphism \eqref{0107}
(where we set $\nu_{n}(-1) = 0$).
Let $\mathfrak{T}_{n}(r) \in D(\alg{X}_{\et})$ to be
the canonical mapping cone of the resulting morphism
	\[
			\tau_{\le r}
			R \lambda_{S, \ast} \Lambda_{n}(r)
		\to
			\bigoplus_{\ideal{p} \in S}
				i_{\ideal{p}, \ast} \nu_{n}(r - 1).
	\]
For $r < 0$, we set $\mathfrak{T}_{n}(r) = \lambda_{S, !} \Lambda_{n}(r)$.
For a general $S \subset P$ and $r \in \Z$,
we obtain an object $\mathfrak{T}_{n}(r) \in D(\alg{U}_{S, \et})$
by the restriction of $\mathfrak{T}_{n}(r) \in D(\alg{X}_{\et})$.

\begin{Prop} \label{0019}
	Assume all of the following:
	\begin{enumerate}
		\item \label{0053}
			$A$ is regular.
		\item \label{0055}
			$A$ contains a fixed primitive $p$-th root of unity $\zeta_{p}$.
		\item \label{0056}
			Either of the following holds:
			\begin{enumerate}
				\item \label{0057}
					$(p)$ is divisible by exactly one prime ideal $\ideal{p}$
					and $A / \ideal{p}$ is regular, or
				\item \label{0058}
					$(p)$ is divisible by exactly two prime ideals
					$\ideal{p}_{\alpha}$ and $\ideal{p}_{\beta}$,
					and $\ideal{p}_{\alpha} + \ideal{p}_{\beta} = \ideal{m}$,
					and both $A / \ideal{p}_{\alpha}$ and $A / \ideal{p}_{\beta}$ are regular.
			\end{enumerate}
	\end{enumerate}
	Set $R = A[1 / p]$ and write $\Spec R = U_{S}$
	(so $S$ is $\{\ideal{p}\}$ in Case \eqref{0057} and
	$\{\ideal{p}_{\alpha}, \ideal{p}_{\beta}\}$ in Case \eqref{0058}).
	Let $q \in \Z$.
	
	Then $R^{q} \pi_{\alg{R}, \ast} \Lambda \in \mathcal{W}_{F}$ is pro-algebraic
	and $R^{q + 1} \pi_{\alg{R}, !} \Lambda \in \mathcal{W}_{F}$ is ind-algebraic.
	They are zero unless $q = 0, 1, 2$.
	We have $\pi_{\alg{R}, \ast} \Lambda \cong \Lambda$.
\end{Prop}

\begin{proof}
	The statement for $R^{q} \pi_{\alg{R}, \ast} \Lambda$ follows from
	\cite[Proposition 10.3.5]{Suz24}.
	We have a perfect pairing
		\[
					R \pi_{\alg{R}, \ast} \Lambda
				\tensor^{L}
					R \pi_{\alg{R}, !} \Lambda
			\to
				\Lambda_{\infty}[-3]
		\]
	in $D(F^{\perar}_{\et})$ by \cite[Proposition 10.1.5]{Suz24}.
	Hence Proposition \ref{0078} implies the statement for $R^{q + 1} \pi_{\alg{R}, !} \Lambda$.
\end{proof}

\begin{Prop} \label{0021}
	Under the assumptions of Proposition \ref{0019},
	let $q \in \Z$.
	Then $\alg{H}^{q}(\alg{R}, \Lambda) \in \mathcal{W}_{F}$ is pro-algebraic
	and $\alg{H}^{q + 1}_{c}(\alg{R}, \Lambda) \in \mathcal{W}_{F}$ is ind-algebraic.
	They are zero unless $q = 0, 1, 2$.
	We have $\alg{\Gamma}(\alg{R}, \Lambda) \cong \Lambda$.
\end{Prop}

\begin{proof}
	This follows from Propositions \ref{0015} and \ref{0019}.
\end{proof}

Let $\pi_{\mathfrak{X} / A} \colon \mathfrak{X} \to \Spec A$ be a resolution of singularities
such that $Y \cup \closure{S} \subset \mathfrak{X}$ is
supported on a strict normal crossing divisor,
where $Y$ is the reduced part of $\mathfrak{X} \times_{A} F$
and $\closure{S}$ is the (reduced) closure of $S$ in $\mathfrak{X}$.
We briefly recall the constructions in \cite[Sections 10.4 and 10.5]{Suz24}.

Let $Y_{0}$ (resp.\ $Y_{1}$) be the set of closed (resp.\ generic) points of $Y$.
For $\eta \in Y_{1}$, let $Y_{\eta}$ be the closure of $\eta$ in $Y$
and set $\eta_{0} = Y_{\eta} \cap Y_{0}$.
Let $F_{\eta}$ be the constant field of $Y_{\eta}$.
Let $\Hat{\Order}_{K_{\eta}}$ be the complete local ring of $\mathfrak{X}$ at $\eta$
and $\Hat{K}_{\eta}$ its fraction field.
For $x \in Y_{0}$, let $\Hat{A}_{x}$ be the complete local ring of $\mathfrak{X}$ at $x$
with residue field $F_{x}$
and let $Y_{1}^{x}$ be the set of height one primes of $\Hat{A}_{x}$
lying over some element of $Y_{1}$ (via the morphism $\Spec \Hat{A}_{x} \to \mathfrak{X}$).
For $\eta \in Y_{1}$ and $x \in \eta_{0}$,
there is a unique $\eta_{x} \in Y_{1}^{x}$ lying over $\eta$.
Let $\Hat{K}_{\eta_{x}}$ be the complete local field of $\Hat{A}_{x}$ at $\eta_{x}$,
$\Hat{\Order}_{K_{\eta_{x}}}$ its ring of integers
and $\Hat{\kappa}(\eta_{x})$ its residue field.
For each $x \in Y_{0}$,
let $\Hat{B}_{x} = \Order(\Spec \Hat{A}_{x} \times_{\mathfrak{X}} Y)$,
$\Hat{R}_{x} = \Order(\Spec \Hat{A}_{x} \times_{\mathfrak{X}} X)$
and $\Hat{R}_{x, S} = \Order(\Spec \Hat{A}_{x} \times_{\mathfrak{X}} U_{S})$.

We will take completed neighborhoods of small enough affine opens
of irreducible components of $Y$.
For each $\eta \in Y_{1}$,
choose an affine open neighborhood $W_{\eta} \subset \mathfrak{X}$ of $\eta$
small enough so that:
\begin{itemize}
	\item \label{0094}
		$W_{\eta} \cap Y_{\eta'} = \emptyset$
		for any $\eta' \in Y_{1} \setminus \{\eta\}$,
	\item \label{0095}
		$W_{\eta}$ does not contain the specialization of any element of $S$ in $Y$ and
	\item \label{0096}
		$W_{\eta} \cap Y \subset W_{\eta}$ is a principal divisor.
\end{itemize}
Set $T = Y \setminus \bigcup_{\eta \in Y_{1}} (W_{\eta} \cap Y)$
and $B_{\eta, T} = \Order(W_{\eta} \cap Y)$.
For each $\eta \in Y_{1}$,
define $\Hat{A}_{\eta, T}$ to be the the completion of
$\Order(W_{\eta})$ with respect to the ideal $\ideal{m} \Order(W_{\eta})$.
Write $\Spec \Hat{R}_{\eta, T} = \Spec \Hat{A}_{\eta, T} \setminus \Spec B_{\eta, T}$.
For any $\eta \in Y_{1}$ and $x \in \eta_{0}$,
we have a canonical $A$-algebra homomorphism
$\Hat{A}_{\eta, T} \to \Hat{\Order}_{K_{\eta_{x}}}$
inducing the natural inclusion map
$B_{\eta, T} \into \Hat{\kappa}(\eta_{x})$ on the quotients.

We can apply the construction in this section for $\Hat{A}_{x}$
and the construction in Section \ref{0059} for $\Hat{A}_{\eta, T}$.
With Weil restrictions $\Weil_{F_{x} / F}$ and $\Weil_{F_{\eta} / F}$,
we have objects
$R \pi_{\Hat{\alg{R}}_{x, S}, \ast} \Lambda \in D(F^{\perar}_{\et})$,
$R \alg{\Gamma}_{c}(\Hat{\alg{R}}_{\eta, T}, \Lambda) \in D(F^{\ind\rat}_{\pro\et})$
and so on.

\begin{Prop} \label{0022}
	Assume that $S$ is the set of height one primes dividing $p$
	and that $\zeta_{p} \in A$.
	Let $q \in \Z$.
	Then $\alg{H}^{q}(\alg{U}_{S}, \Lambda) \in \Pro \Alg_{u} / F$
	and $\alg{H}^{q + 1}_{c}(\alg{U}_{S}, \Lambda) \in \Ind \Alg_{u} / F$.
	They are zero unless $q = 0, 1, 2, 3$ and $q = 0, 1, 2$, respectively.
	We have $\alg{\Gamma}(\alg{U}_{S}, \Lambda) \cong \Lambda$.
\end{Prop}

\begin{proof}
	We have distinguished triangles
		\begin{gather*}
					\bigoplus_{\eta \in Y_{1}}
						R \alg{\Gamma}_{c}(\Hat{\alg{R}}_{\eta, T}, \Lambda)
				\to
					R \alg{\Gamma}(\alg{U}_{S}, \Lambda)
				\to
					\bigoplus_{x \in T}
						R \alg{\Gamma}(\Hat{\alg{R}}_{x, S}, \Lambda),
			\\
					\bigoplus_{x \in T}
						R \alg{\Gamma}_{c}(\Hat{\alg{R}}_{x, S}, \Lambda)
				\to
					R \alg{\Gamma}_{c}(\alg{U}_{S}, \Lambda)
				\to
					\bigoplus_{\eta \in Y_{1}}
						R \alg{\Gamma}(\Hat{\alg{R}}_{\eta, T}, \Lambda)
		\end{gather*}
	in $D(F^{\ind\rat}_{\pro\et})$
	by the diagram in the proof of \cite[Proposition 10.5.5]{Suz24}.%
	\footnote{
		The objects
		$R \Bar{\pi}_{\Hat{\alg{U}}_{\Hat{T}}, \ast} \Lambda$ and
		$R \Bar{\pi}_{\Hat{\alg{U}}_{\Hat{T}}, \Hat{!}} \Lambda$
		in the proof of \cite[Proposition 10.5.5]{Suz24}
		are isomorphic to
		$R \pi_{\alg{U}_{S}, \ast} \Lambda$ and
		$R \Bar{\pi}_{\alg{U}_{S}, !} \Lambda$
		($= R \pi_{\alg{U}_{S}, !} \Lambda$ in our notation),
		respectively,
		by \cite[Propositions 10.4.1 and 10.5.1]{Suz24}.
		After applying $\algebrize$,
		they become $R \alg{\Gamma}(\alg{U}_{S}, \Lambda)$
		and $R \alg{\Gamma}_{c}(\alg{U}_{S}, \Lambda)$, respectively.
	}
	Hence Propositions \ref{0020} and \ref{0021} imply the result.
\end{proof}

\begin{Prop} \label{0040}
	Assume that $S$ contains all height one primes dividing $p$.
	Let $n \ge 1$ and $r \in \Z$.
	Then $\alg{H}^{q}(\alg{U}_{S}, \Lambda_{n}(r)) \in \Pro \Alg_{u} / F$
	and $\alg{H}^{q + 1}_{c}(\alg{U}_{S}, \Lambda_{n}(r)) \in \Ind \Alg_{u} / F$.
	They are zero unless $q = 0, 1, 2, 3$ and $q = 0, 1, 2$, respectively.
	The group $\alg{\Gamma}(\alg{U}_{S}, \Lambda_{n}(r))$ is finite.
\end{Prop}

\begin{proof}
	The case $n \ge 1$ is reduced to the case $n = 1$.
	Let $S' \subset P$ be a finite subset containing $S$.
	Let $U_{S'} = X \setminus S'$.
	Then the statement for $U_{S}$ is equivalent to the statement for $U_{S'}$.
	Indeed, we have distinguished triangles
		\begin{gather*}
					\bigoplus_{\ideal{p} \in S' \setminus S}
						R \alg{\Gamma}_{c}(\Hat{\alg{A}}_{\ideal{p}}, \Lambda(r))
				\to
					R \alg{\Gamma}(\alg{U}_{S}, \Lambda(r))
				\to
					R \alg{\Gamma}(\alg{U}_{S'}, \Lambda(r)),
			\\
					R \alg{\Gamma}_{c}(\alg{U}_{S'}, \Lambda(r))
				\to
					R \alg{\Gamma}_{c}(\alg{U}_{S}, \Lambda(r))
				\to
					\bigoplus_{\ideal{p} \in S' \setminus S}
						R \alg{\Gamma}(\Hat{\alg{A}}_{\ideal{p}}, \Lambda(r))
		\end{gather*}
	in $D(F^{\ind\rat}_{\pro\et})$;
	see the diagram in \cite[Section 10.6]{Suz24}.
	For any $\ideal{p} \in P$ not dividing $p$, we have
		\begin{gather*}
					R \alg{\Gamma}_{c}(\Hat{\alg{A}}_{\ideal{p}}, \Lambda(r))
				\cong
					R \alg{\Gamma}(\varalg{\kappa}(\ideal{p}), \Lambda(r - 1))[-2],
			\\
					R \alg{\Gamma}(\Hat{\alg{A}}_{\ideal{p}}, \Lambda(r))
				\cong
					R \alg{\Gamma}(\varalg{\kappa}(\ideal{p}), \Lambda(r)),
		\end{gather*}
	which are objects of $D^{b}(\Alg_{u} / F)$ concentrated in
	degrees $2, 3$, degrees $0, 1$, respectively, by \cite[Theorem 7.2.4]{Suz24}.
	This implies the equivalence.
	Now a standard d\'evissage argument (see \cite[Section 10.6]{Suz24})
	reduces the general case to Proposition \ref{0022}.
\end{proof}

We will later show in Proposition \ref{0075}
that the groups in Proposition \ref{0040} are actually in $\mathcal{W}_{F}$.


\section{Finiteness and duality}
\label{0115}

We continue the notation of Section \ref{0086}.
In this section, we prove Theorems \ref{0077} and \ref{0088}.
We first treat the case $S = \emptyset$ (where $U = U_{S}$)
since this case admits stronger finiteness properties than the general case
and hence is easier.
Based on it, we then treat the case where $S \subset P$ contains all prime ideals dividing $p$.
Finally, we treat the general case.

We first recall the perfect pairing
	\begin{equation} \label{0108}
				R \alg{\Gamma}(\alg{U}_{S}, \mathfrak{T}_{n}(r))
			\tensor^{L}
				R \alg{\Gamma}_{c}(\alg{U}_{S}, \mathfrak{T}_{n}(2 - r))
		\to
			\Lambda_{\infty}[-3]
	\end{equation}
in $D(F^{\ind\rat}_{\pro\et})$ from \cite[Theorem 10.6.1]{Suz24}.

\begin{Prop} \label{0043}
	Let $n \ge 1$.
	Then $\alg{\Gamma}(\alg{X}, \Lambda_{n}) \cong \Lambda_{n}$ and
	$\alg{H}^{1}(\alg{X}, \Lambda_{n})$ are finite,
	$\alg{H}^{2}(\alg{X}, \Lambda_{n})$ and $\alg{H}^{3}(\alg{X}, \Lambda_{n})$ are ind-algebraic in $\mathcal{W}_{F}$
	and $\alg{H}^{q}(\alg{X}, \Lambda_{n}) = 0$ for $q \ne 0, 1, 2, 3$.
\end{Prop}

\begin{proof}
	Let $S$ be the set of height one primes dividing $p$.
	By Proposition \ref{0031} \eqref{0035}, we have 
	$\alg{\Gamma}(\alg{X}, \Lambda_{n}) \isomto \alg{\Gamma}(\alg{U}_{S}, \Lambda_{n}) \cong \Lambda_{n}$.
	Consider the distinguished triangle
		\[
				R \alg{\Gamma}_{c}(\alg{U}_{S}, \Lambda_{n})
			\to
				R \alg{\Gamma}(\alg{X}, \Lambda_{n})
			\to
				\bigoplus_{\ideal{p} \in S}
					R \alg{\Gamma}(\Hat{\alg{A}}_{\ideal{p}}, \Lambda_{n})
		\]
	in $D(F^{\ind\rat}_{\pro\et})$.
	The first and third terms are in $D^{b}(\Ind \Alg_{u} / F)$
	by Propositions \ref{0040} and \ref{0031} \eqref{0034}, respectively.
	Hence so is the second.
	The same propositions show $\alg{H}^{q}(\alg{X}, \Lambda_{n}) = 0$ for $q \ne 0, 1, 2, 3$.
	In the distinguished triangle
		\begin{equation} \label{0042}
				R \alg{\Gamma}(\alg{X}, \Lambda_{n})
			\to
				R \alg{\Gamma}(\alg{U}_{S}, \Lambda_{n})
			\to
				\bigoplus_{\ideal{p} \in S}
					R \alg{\Gamma}_{c}(\Hat{\alg{A}}_{\ideal{p}}, \Lambda_{n})[1],
		\end{equation}
	the second term is in $D^{b}(\Pro \Alg_{u} / F)$
	by Proposition \ref{0040}
	and the cohomologies of the third term are in $\mathcal{W}_{F}$
	by Proposition \ref{0031} \eqref{0032}.
	Therefore we may apply Proposition \ref{0041} to
	the long exact sequence associated with \eqref{0042}.
	This implies that $\alg{H}^{q}(\alg{X}, \Lambda_{n}) \in \mathcal{W}_{F}$ for all $q$.
	Since $\alg{H}^{1}_{c}(\Hat{\alg{A}}_{\ideal{p}}, \Lambda_{n}) = 0$,
	we have an exact sequence
		\[
				0
			\to
				\alg{H}^{1}(\alg{X}, \Lambda_{n})
			\to
				\alg{H}^{1}(\alg{U}_{S}, \Lambda_{n})
			\to
				\bigoplus_{\ideal{p} \in S}
					\alg{H}^{2}_{c}(\Hat{\alg{A}}_{\ideal{p}}, \Lambda_{n}).
		\]
	All the terms are in $\mathcal{W}_{F}$,
	with the first ind-algebraic and the second pro-algebraic.
	Therefore $\alg{H}^{1}(\alg{X}, \Lambda_{n}) \in \Alg_{u} / F$.
	By \cite[Theorem 1.2]{Suz23}, we know that
	the group
		$
				\alg{H}^{1}(\alg{X}, \Lambda_{n})(\closure{F})
			\cong
				H^{1}(\alg{X}(\closure{F}), \Lambda_{n})
		$
	is finite.
	This implies that $\alg{H}^{1}(\alg{X}, \Lambda_{n})$ is finite.
\end{proof}

\begin{Prop}
	Let $n \ge 1$.
	Then $\alg{\Gamma}(\alg{X}, \mathfrak{T}_{n}(2))$ and
	$\alg{H}^{3}(\alg{X}, \mathfrak{T}_{n}(2)) \cong \Lambda_{n}$ are finite,
	$\alg{H}^{1}(\alg{X}, \mathfrak{T}_{n}(2))$ and
	$\alg{H}^{2}(\alg{X}, \mathfrak{T}_{n}(2))$ are pro-algebraic in $\mathcal{W}_{F}$
	and $\alg{H}^{q}(\alg{X}, \Lambda_{n}) = 0$ for $q \ne 0, 1, 2, 3$.
\end{Prop}

\begin{proof}
	This follows from Propositions \ref{0078} and \ref{0043} by
	the perfect pairing \eqref{0108} with $S = \emptyset$ and $r = 0$.
\end{proof}

\begin{Prop} \label{0046}
	Let $n \ge 1$.
	Then $\alg{\Gamma}(\alg{X}, \mathfrak{T}_{n}(1))$ is finite,
	$\alg{H}^{1}(\alg{X}, \mathfrak{T}_{n}(1))$ is pro-algebraic in $\mathcal{W}_{F}$,
	$\alg{H}^{2}(\alg{X}, \mathfrak{T}_{n}(1))$ is in $\Alg_{u} / F$,
	$\alg{H}^{3}(\alg{X}, \mathfrak{T}_{n}(1))$ is ind-algebraic in $\mathcal{W}_{F}$
	and $\alg{H}^{q}(\alg{X}, \mathfrak{T}_{n}(1)) = 0$ for $q \ne 0, 1, 2, 3$.
\end{Prop}

\begin{proof}
	The statement for $\alg{\Gamma}(\alg{X}, \mathfrak{T}_{n}(1))$ is obvious.
	Let $S$ be the set of primes dividing $p$.
	Consider the distinguished triangle
		\begin{equation} \label{0071}
				\bigoplus_{\ideal{p} \in S}
					R \alg{\Gamma}_{c}(\Hat{\alg{A}}_{\ideal{p}}, \mathfrak{T}_{n}(1))
			\to
				R \alg{\Gamma}(\alg{X}, \mathfrak{T}_{n}(1))
			\to
				R \alg{\Gamma}(\alg{U}_{S}, \Lambda_{n}(1))
		\end{equation}
	in $D(F^{\ind\rat}_{\pro\et})$.
	By Propositions \ref{0031} and \ref{0040},
	we know that $\alg{H}^{q}(\alg{X}, \mathfrak{T}_{n}(1)) \in \Pro \Alg_{u} / F$
	for $q = 1$ and $2$ and zero for $q \ne 0, 1, 2, 3$.
	Consider the distinguished triangle
		\[
				R \alg{\Gamma}_{c}(\alg{U}_{S}, \Lambda_{n}(1))
			\to
				R \alg{\Gamma}(\alg{X}, \mathfrak{T}_{n}(1))
			\to
				\bigoplus_{\ideal{p} \in S}
					R \alg{\Gamma}(\Hat{\alg{A}}_{\ideal{p}}, \mathfrak{T}_{n}(1))
		\]
	in $D(F^{\ind\rat}_{\pro\et})$.
	As $\alg{H}^{2}(\Hat{\alg{A}}_{\ideal{p}}, \mathfrak{T}_{n}(1)) = 0$
	by Propositions \ref{0031},
	this induces a distinguished triangle
		\[
				\tau_{\le 2}
				R \alg{\Gamma}_{c}(\alg{U}_{S}, \Lambda_{n}(1))
			\to
				\tau_{\le 2}
				R \alg{\Gamma}(\alg{X}, \mathfrak{T}_{n}(1))
			\to
				\tau_{\le 2}
				\bigoplus_{\ideal{p} \in S}
					R \alg{\Gamma}(\Hat{\alg{A}}_{\ideal{p}}, \mathfrak{T}_{n}(1)).
		\]
	The first, second and third terms have cohomologies in
	$\Ind \Alg_{u} / F$, $\Pro \Alg_{u} / F$ and $\mathcal{W}_{F}$, respectively.
	Applying Proposition \ref{0041}, we know that
	$\alg{H}^{1}(\alg{X}, \mathfrak{T}_{n}(1)) \in \mathcal{W}_{F}$
	and $\alg{H}^{2}(\alg{X}, \mathfrak{T}_{n}(1)) \in \Alg_{u} / F$.
	Also, we have $\alg{H}_{c}^{3}(\alg{U}_{S}, \Lambda_{n}(1)) \isomto \alg{H}^{3}(\alg{X}, \mathfrak{T}_{n}(1))$,
	so it is in $\Ind \Alg_{u} / F$.
	The distinguished triangle \eqref{0071} induces an exact sequence
		\begin{align*}
			&
					\alg{H}^{2}(\alg{X}, \mathfrak{T}_{n}(1))
				\to
					\alg{H}^{2}(\alg{U}_{S}, \Lambda_{n}(1))
				\to
					\bigoplus_{\ideal{p} \in S}
						\alg{H}^{3}_{c}(\Hat{\alg{A}}_{\ideal{p}}, \mathfrak{T}_{n}(1))
			\\
			& \quad
				\to
					\alg{H}^{3}(\alg{X}, \mathfrak{T}_{n}(1))
				\to
					\alg{H}^{3}(\alg{U}_{S}, \Lambda_{n}(1))
				\to
					0,
		\end{align*}
	where the kernel of the first morphism is finite.
	These terms, from left to right, are in
	$\Alg_{u} / F$, $\Pro \Alg_{u} / F$, $\mathcal{W}_{F}$,
	$\Ind \Alg_{u} / F$ and $\Pro \Alg_{u} / F$.
	Applying Proposition \ref{0041}, we know that
	$\alg{H}^{3}(\alg{X}, \mathfrak{T}_{n}(1)) \in \mathcal{W}_{F}$.
\end{proof}

\begin{Prop} \label{0075}
	Let $n \ge 1$ and $r \in \Z$.
	Assume that $S$ contains all primes dividing $p$.
	\begin{enumerate}
		\item
			$\alg{\Gamma}(\alg{U}_{S}, \Lambda_{n}(r))$ is finite,
			$\alg{H}^{1}(\alg{U}_{S}, \Lambda_{n}(r))$ and
			$\alg{H}^{2}(\alg{U}_{S}, \Lambda_{n}(r))$ are pro-algebraic in $\mathcal{W}_{F}$
			and $\alg{H}^{q}(\alg{U}_{S}, \Lambda_{n}(r)) = 0$ for $q \ne 0, 1, 2$.
		\item
			$\alg{H}^{1}_{c}(\alg{U}_{S}, \Lambda_{n}(r))$ is finite,
			$\alg{H}^{2}_{c}(\alg{U}_{S}, \Lambda_{n}(r))$ and
			$\alg{H}^{3}_{c}(\alg{U}_{S}, \Lambda_{n}(r))$ are ind-algebraic in $\mathcal{W}_{F}$
			and $\alg{H}^{q}_{c}(\alg{U}_{S}, \Lambda_{n}(r)) = 0$ for $q \ne 1, 2, 3$.
	\end{enumerate}
\end{Prop}

\begin{proof}
	That these groups are in $\mathcal{W}_{F}$ is already proved in the proof of
	Propositions \ref{0043} and \ref{0046}
	if $S$ is the set of all primes dividing $p$.
	The case of larger $S$ can be reduced to this case.
	With Proposition \ref{0040},
	what is remaining to show is $\alg{H}^{3}(\alg{U}_{S}, \Lambda_{n}(r)) = 0$.
	In the exact sequence
		\[
				\bigoplus_{\ideal{p} \in S}
					\alg{\Gamma}(\Hat{\alg{A}}_{\ideal{p}}, \Lambda_{n})
			\to
				\alg{H}^{1}_{c}(\alg{U}_{S}, \Lambda_{n})
			\to
				\alg{H}^{1}(\alg{X}, \Lambda_{n}),
		\]
	the first and third terms are finite \'etale
	trivially and by Proposition \ref{0043}, respectively.
	Therefore $\alg{H}^{1}_{c}(\alg{U}_{S}, \Lambda_{n})$ is finite \'etale.
	This implies that $\alg{H}^{1}_{c}(\alg{U}_{S}, \Lambda(r))$ is finite \'etale for all $r$
	(reduce to the case $\zeta_{p} \in A$)
	and hence $\alg{H}^{1}_{c}(\alg{U}_{S}, \Lambda_{n}(r))$ is finite \'etale for all $r$ and $n$.
	With this and by the perfect pairing \eqref{0108} and Proposition \ref{0078},
	we know that $\alg{H}^{3}(\alg{U}_{S}, \Lambda_{n}(r)) = 0$.
\end{proof}

Now we make no assumption on $S$.

\begin{Prop} \label{0050}
	Let $n \ge 1$ and $r \le 0$.
	Then $\alg{\Gamma}(\alg{U}_{S}, \mathfrak{T}_{n}(r))$ is finite,
	$\alg{H}^{1}(\alg{U}_{S}, \mathfrak{T}_{n}(r))$ is pro-algebraic in $\mathcal{W}_{F}$,
	$\alg{H}^{2}(\alg{U}_{S}, \mathfrak{T}_{n}(r))$ is in $\mathcal{W}_{F}$,
	$\alg{H}^{3}(\alg{U}_{S}, \mathfrak{T}_{n}(r))$ is ind-algebraic in $\mathcal{W}_{F}$
	and $\alg{H}^{q}(\alg{U}_{S}, \mathfrak{T}_{n}(r)) = 0$ for $q \ne 0, 1, 2, 3$.
\end{Prop}

\begin{proof}
	Let $S_{0}$ be the set of all primes dividing $p$.
	We may assume that $S \subset S_{0}$.
	Let $U_{0} = U_{S_{0}}$.
	The finiteness of $\alg{\Gamma}(\alg{U}_{S}, \mathfrak{T}_{n}(r))$ is obvious.
	Denote the distinguished triangles
		\begin{gather*}
					\bigoplus_{\ideal{p} \in S}
						R \alg{\Gamma}_{c}(\Hat{\alg{A}}_{\ideal{p}}, \mathfrak{T}_{n}(r))
				\to
					R \alg{\Gamma}(\alg{X}, \mathfrak{T}_{n}(r))
				\to
					R \alg{\Gamma}(\alg{U}_{S}, \mathfrak{T}_{n}(r)),
			\\
					\bigoplus_{\ideal{p} \in S_{0} \setminus S}
						R \alg{\Gamma}_{c}(\Hat{\alg{A}}_{\ideal{p}}, \mathfrak{T}_{n}(r))
				\to
					R \alg{\Gamma}(\alg{U}_{S}, \mathfrak{T}_{n}(r))
				\to
					R \alg{\Gamma}(\alg{U}_{0}, \mathfrak{T}_{n}(r)),
		\end{gather*}
	by
		\begin{gather*}
					D
				\to
					C_{X}
				\to
					C_{U_{S}},
			\\
					E
				\to
					C_{U_{S}}
				\to
					C_{U_{0}},
		\end{gather*}
	respectively.
	
	The morphism $H^{1} C_{U_{0}} \to H^{2} E$ is a morphism
	from a pro-algebraic group in $\mathcal{W}_{F}$
	to an object of $\mathcal{W}_{F}$.
	Hence its kernel $H^{1} C_{U_{S}}$ is pro-algebraic.
	Hence the morphism $H^{1} C_{U_{S}} \to H^{2} D$ is a morphism
	from a pro-algebraic group to an object of $\mathcal{W}_{F}$.
	Hence its image is pro-algebraic.
	The morphism $H^{2} D \to H^{2} C_{X}$ is a morphism
	from an object of $\mathcal{W}_{F}$ to an ind-algebraic group in $\mathcal{W}_{F}$.
	Hence its image is ind-algebraic.
	Therefore $\Im(H^{1} C_{U_{S}} \to H^{2} D), \Im(H^{2} D \to H^{2} C_{X}) \in \mathcal{W}_{F}$
	by Proposition \ref{0041}.
	As $H^{1} C_{X}$ is finite, we have $H^{1} C_{U_{S}} \in \mathcal{W}_{F}$.
	
	In the exact sequence
		\[
				0
			\to
				H^{1} C_{U_{S}}
			\to
				H^{1} C_{U_{0}}
			\to
				H^{2} E,
		\]
	the terms are in $\mathcal{W}_{F}$, with the first and second terms pro-algebraic.
	Hence $\Im(H^{1} C_{U_{0}} \to H^{2} E)$ is pro-algebraic in $\mathcal{W}_{F}$
	by Proposition \ref{0047}.
	Hence $\Im(H^{2} E \to H^{2} C_{U_{S}}) \in \mathcal{W}_{F}$
	again by Proposition \ref{0047}.
	As $H^{2} C_{X}$ is ind-algebraic in $\mathcal{W}_{F}$,
	the object $\Im(H^{2} C_{X} \to H^{2} C_{U_{S}})$ is ind-algebraic
	by Proposition \ref{0024}.
	
	The morphism $H^{2} C_{U_{0}} \to H^{3} E$ is a morphism
	from a pro-algebraic group in $\mathcal{W}_{F}$ to an object of $\mathcal{W}_{F}$.
	Hence $\Im(H^{2} C_{U_{S}} \to H^{2} C_{U_{0}}) = \Ker(H^{2} C_{U_{0}} \to H^{3} E)$
	and $\Im(H^{2} C_{U_{0}} \to H^{3} E)$ are pro-algebraic.
	Since the object $H^{3} C_{X}$ is ind-algebraic in $\mathcal{W}_{F}$
	and the morphism $H^{3} C_{X} \to H^{3} C_{U_{S}}$ is surjective,
	it follows that $H^{3} C_{U_{S}}$ is ind-algebraic by Proposition \ref{0024}.
	We have $H^{3} E \in \mathcal{W}_{F}$.
	Applying Proposition \ref{0041} to the exact sequence
		\[
				0
			\to
				\Im(H^{2} C_{U_{0}} \to H^{3} E)
			\to
				H^{3} E
			\to
				H^{3} C_{U_{S}}
			\to
				0,
		\]
	we know that the first and third terms in this sequence are in $\mathcal{W}_{F}$.
	Thus $H^{3} C_{U_{S}}$ is ind-algebraic in $\mathcal{W}_{F}$.
	
	The morphism $H^{3} D \to H^{3} C_{X}$ is a morphism
	from an object of $\mathcal{W}_{F}$ to an ind-algebraic group in $\mathcal{W}_{F}$.
	Hence we have $\Im(H^{3} D \to H^{3} C_{X}) \in \mathcal{W}_{F}$
	by Proposition \ref{0048}.
	Hence the morphism $H^{3} D \onto \Im(H^{3} D \to H^{3} C_{X})$
	is a morphism from an object of $\mathcal{W}_{F}$
	to an ind-algebraic group in $\mathcal{W}_{F}$.
	Therefore Proposition \ref{0049} implies that
	its kernel $\Im(H^{2} C_{U_{S}} \to H^{3} D)$ has identity component in $\mathcal{W}_{F}$
	and component group \'etale.
	
	Consider the exact sequences
		\begin{gather*}
					0
				\to
					\Im(H^{2} C_{X} \to H^{2} C_{U_{S}})
				\to
					H^{2} C_{U_{S}}
				\to
					\Im(H^{2} C_{U_{S}} \to H^{3} D)
				\to
					0,
			\\
					0
				\to
					\Im(H^{2} E \to H^{2} C_{U_{S}})
				\to
					H^{2} C_{U_{S}}
				\to
					\Im(H^{2} C_{U_{S}} \to H^{2} C_{U_{0}})
				\to
					0.
		\end{gather*}
	With what have been proved so far,
	the first sequence shows that $\pi_{0}(H^{2} C_{U_{S}})$ is \'etale
	and the second sequence shows that $\pi_{0}(H^{2} C_{U_{S}})$ is profinite.
	Thus $\pi_{0}(H^{2} C_{U_{S}})$ is finite,
	and $\Im(H^{2} C_{U_{S}} \to H^{3} D) \in \mathcal{W}_{F}$.
	The image of the composite
		\[
				\Im(H^{2} C_{X} \to H^{2} C_{U_{S}})
			\into
				H^{2} C_{U_{S}}
			\onto
				\Im(H^{2} C_{U_{S}} \to H^{2} C_{U_{0}})
		\]
	is $\Im(H^{2} C_{X} \to H^{2} C_{U_{0}})$,
	which is in $\Alg_{u} / F$
	by what was shown about \eqref{0042} in the proof of Proposition \ref{0043}.
	Hence, with what have been proved so far, we know that the induced morphism
		\[
				\Im(H^{2} C_{U_{S}} \to H^{3} D)
			\onto
				\frac{
					\Im(H^{2} C_{U_{S}} \to H^{2} C_{U_{0}})
				}{
					\Im(H^{2} C_{X} \to H^{2} C_{U_{0}})
				}
		\]
	is a surjection from an object of $\mathcal{W}_{F}$
	onto a pro-algebraic group.
	Therefore Proposition \ref{0047} implies that
	$\Im(H^{2} C_{U_{S}} \to H^{2} C_{U_{0}}) / \Im(H^{2} C_{X} \to H^{2} C_{U_{0}})$
	is in $\mathcal{W}_{F}$.
	Hence $\Im(H^{2} C_{U_{S}} \to H^{2} C_{U_{0}}) \in \mathcal{W}_{F}$.
	With $\Im(H^{2} E \to H^{2} C_{U_{S}}) \in \mathcal{W}_{F}$,
	we know that $H^{2} C_{U_{S}} \in \mathcal{W}_{F}$.
\end{proof}

\begin{Prop}
	Let $n \ge 1$ and $r \ge 2$.
	Then $\alg{H}_{c}^{0}(\alg{U}_{S}, \mathfrak{T}_{n}(r))$ is finite,
	$\alg{H}^{1}_{c}(\alg{U}_{S}, \mathfrak{T}_{n}(r))$ is pro-algebraic in $\mathcal{W}_{F}$,
	$\alg{H}^{2}_{c}(\alg{U}_{S}, \mathfrak{T}_{n}(r))$ is in $\mathcal{W}_{F}$,
	$\alg{H}^{3}_{c}(\alg{U}_{S}, \mathfrak{T}_{n}(r))$ is ind-algebraic in $\mathcal{W}_{F}$
	and $\alg{H}^{q}_{c}(\alg{U}_{S}, \mathfrak{T}_{n}(r)) = 0$ for $q \ne 0, 1, 2, 3$.
\end{Prop}

\begin{proof}
	This follows from Proposition \ref{0050} by duality.
\end{proof}

\begin{Prop} \label{0051}
	Let $n \ge 1$ and $r \le 0$.
	Then $\alg{H}^{0}_{c}(\alg{U}_{S}, \mathfrak{T}_{n}(r))$ and
	$\alg{H}^{1}_{c}(\alg{U}_{S}, \mathfrak{T}_{n}(r))$ are finite,
	$\alg{H}^{2}_{c}(\alg{U}_{S}, \mathfrak{T}_{n}(r))$ and
	$\alg{H}^{3}_{c}(\alg{U}_{S}, \mathfrak{T}_{n}(r))$ are ind-algebraic in $\mathcal{W}_{F}$
	and $\alg{H}^{q}(\alg{U}_{S}, \mathfrak{T}_{n}(r)) = 0$ for $q \ne 0, 1, 2, 3$.
\end{Prop}

\begin{proof}
	The case $r < 0$ follows from Proposition \ref{0075}.
	For $r = 0$, the statement follows from the distinguished triangle
		\[
				R \alg{\Gamma}_{c}(\alg{U}_{S}, \Lambda_{n})
			\to
				R \alg{\Gamma}(\alg{X}, \Lambda_{n})
			\to
				\bigoplus_{\ideal{p} \in S}
					R \alg{\Gamma}(\Hat{\alg{A}}_{\ideal{p}}, \Lambda_{n})
		\]
	and Propositions \ref{0031} and \ref{0043}.
\end{proof}

\begin{Prop}
	Let $n \ge 1$ and $r \ge 2$.
	Then $\alg{\Gamma}(\alg{U}_{S}, \mathfrak{T}_{n}(r))$ and
	$\alg{H}^{3}(\alg{U}_{S}, \mathfrak{T}_{n}(r))$ are finite,
	$\alg{H}^{1}(\alg{U}_{S}, \mathfrak{T}_{n}(r))$ and
	$\alg{H}^{2}(\alg{U}_{S}, \mathfrak{T}_{n}(r))$ are pro-algebraic in $\mathcal{W}_{F}$ and
	$\alg{H}^{q}(\alg{U}_{S}, \mathfrak{T}_{n}(r)) = 0$ for $q \ne 0, 1, 2, 3$.
\end{Prop}

\begin{proof}
	This follows from the distinguished triangle
		\[
				\bigoplus_{\ideal{p} \in S}
					R \alg{\Gamma}_{c}(\Hat{\alg{A}}_{\ideal{p}}, \mathfrak{T}_{n}(r))
			\to
				R \alg{\Gamma}(\alg{X}, \mathfrak{T}_{n}(r))
			\to
				R \alg{\Gamma}(\alg{U}_{S}, \mathfrak{T}_{n}(r))
		\]
	and Propositions \ref{0031} and \ref{0043}
	(or from Proposition \ref{0051} by duality).
\end{proof}

\begin{Prop} \label{0052}
	Let $n \ge 1$.
	Then $\alg{\Gamma}(\alg{U}_{S}, \mathfrak{T}_{n}(1))$ is finite,
	$\alg{H}^{1}(\alg{U}_{S}, \mathfrak{T}_{n}(1))$ and
	$\alg{H}^{2}(\alg{U}_{S}, \mathfrak{T}_{n}(1))$ are pro-algebraic in $\mathcal{W}_{F}$,
	$\alg{H}^{3}(\alg{U}_{S}, \mathfrak{T}_{n}(1))$ is ind-algebraic in $\mathcal{W}_{F}$ and
	$\alg{H}^{q}(\alg{U}_{S}, \mathfrak{T}_{n}(1)) = 0$ for $q \ne 0, 1, 2, 3$.
\end{Prop}

\begin{proof}
	Let $S_{0}$ be the set of all primes dividing $p$.
	We may assume that $S \subset S_{0}$.
	Let $U_{0} = U_{S_{0}}$.
	The finiteness of $\alg{\Gamma}(\alg{U}_{S}, \mathfrak{T}_{n}(1))$ is obvious.
	Denote the distinguished triangles
		\begin{gather*}
					\bigoplus_{\ideal{p} \in S}
						R \alg{\Gamma}_{c}(\Hat{\alg{A}}_{\ideal{p}}, \mathfrak{T}_{n}(1))
				\to
					R \alg{\Gamma}(\alg{X}, \mathfrak{T}_{n}(1))
				\to
					R \alg{\Gamma}(\alg{U}_{S}, \mathfrak{T}_{n}(1)),
			\\
					\bigoplus_{\ideal{p} \in S_{0} \setminus S}
						R \alg{\Gamma}_{c}(\Hat{\alg{A}}_{\ideal{p}}, \mathfrak{T}_{n}(1))
				\to
					R \alg{\Gamma}(\alg{U}_{S}, \mathfrak{T}_{n}(1))
				\to
					R \alg{\Gamma}(\alg{U}_{0}, \mathfrak{T}_{n}(1)),
		\end{gather*}
	by
		\begin{gather*}
					D
				\to
					C_{X}
				\to
					C_{U_{S}},
			\\
					E
				\to
					C_{U_{S}}
				\to
					C_{U_{0}},
		\end{gather*}
	respectively.
	
	In the exact sequence
		\[
				0
			\to
				H^{1} C_{X}
			\to
				H^{1} C_{U_{S}}
			\to
				H^{2} D,
		\]
	the first term is pro-algebraic in $\mathcal{W}_{F}$
	and the third term is finite.
	Hence $H^{1} C_{U_{S}}$ is pro-algebraic in $\mathcal{W}_{F}$.
	Since $H^{3} C_{X}$ is ind-algebraic in $\mathcal{W}_{F}$,
	the surjection $H^{3} C_{X} \onto H^{3} C_{U_{S}}$ and Proposition \ref{0024}
	imply that $H^{3} C_{U_{S}} \in \Ind \Alg_{u} / F$.
	The morphism $H^{2} C_{U_{0}} \to H^{3} E$ is a morphism
	from a pro-algebraic group in $\mathcal{W}_{F}$
	to an object of $\mathcal{W}_{F}$.
	Hence its image and kernel are pro-algebraic.
	Applying Proposition \ref{0041}
	to the exact sequence
		\[
				0
			\to
				\Im(H^{2} C_{U_{0}} \to H^{3} E)
			\to
				H^{3} E
			\to
				H^{3} C_{U_{S}}
			\to
				0,
		\]
	we know that $H^{3} C_{U_{S}} \in \mathcal{W}_{F}$.
	In the exact sequence
		\[
				H^{2} E
			\to
				H^{2} C_{U_{S}}
			\to
				\Ker(H^{2} C_{U_{0}} \to H^{3} E)
			\to
				0,
		\]
	the first term is finite and the third term is pro-algebraic.
	Hence $H^{2} C_{U_{S}}$ is pro-algebraic.
	In the exact sequence
		\[
				H^{2} C_{X}
			\to
				H^{2} C_{U_{S}}
			\to
				H^{3} D
			\to
				H^{3} C_{X},
		\]
	the first term is in $\Alg_{u} / F$,
	the second in $\Pro \Alg_{u} / F$,
	the third in $\mathcal{W}_{F}$
	and the fourth ind-algebraic in $\mathcal{W}_{F}$.
	Applying Proposition \ref{0041},
	we know that $H^{2} C_{U_{S}} \in \mathcal{W}_{F}$.
\end{proof}

\begin{Prop}
	Let $n \ge 1$.
	Then $\alg{\Gamma}_{c}(\alg{U}_{S}, \mathfrak{T}_{n}(1))$ is finite,
	$\alg{H}^{1}_{c}(\alg{U}_{S}, \mathfrak{T}_{n}(1))$ is pro-algebraic in $\mathcal{W}_{F}$,
	$\alg{H}^{2}_{c}(\alg{U}_{S}, \mathfrak{T}_{n}(1))$ and
	$\alg{H}^{3}_{c}(\alg{U}_{S}, \mathfrak{T}_{n}(1))$ are ind-algebraic in $\mathcal{W}_{F}$ and
	$\alg{H}^{q}_{c}(\alg{U}_{S}, \mathfrak{T}_{n}(1)) = 0$ for $q \ne 0, 1, 2, 3$.
\end{Prop}

\begin{proof}
	This follows from Proposition \ref{0052} by duality.
\end{proof}

Theorem \ref{0077} follows from the above results.
With Proposition \ref{0078},
we obtain Theorem \ref{0088}.


\section{Hasse principles}
\label{0116}

In this section, we prove Theorem \ref{0109}.

\begin{Def}
	Assume $F = \closure{F}$.
	\begin{enumerate}
		\item
			Let $U \subset X$ be a dense open subscheme.
			Let $\ideal{p} \in P \cap U$.
			A finite \'etale covering $U' / U$ is said to be \emph{completely split at $\ideal{p}$}
			if the $\kappa(\ideal{p})$-scheme $U' \times_{U} \kappa(\ideal{p})$ is isomorphic
			to a disjoint union of copies of $\Spec \kappa(\ideal{p})$.
		\item
			A finite \'etale covering $X' / X$ is said to be \emph{completely split}
			if it is completely split at all elements of $P$.
		\item
			For $n \ge 1$,
			define $H^{1}_{\cs}(X, \Lambda_{n})$ to be the subgroup of $H^{1}(X, \Lambda_{n})$
			consisting of completely split coverings.
	\end{enumerate}
\end{Def}

\begin{Prop} \label{0060}
	Assume that $F = \closure{F}$ and
	$A$ satisfies the assumptions of Proposition \ref{0019}
	and Condition \eqref{0057}.
	Write $\ideal{p} = (\pi)$ and $\ideal{m} = (\pi, t)$.
	Let $U = \Spec A[1 / p]$.
	Let $n \ge 1$ and $U' \in H^{1}(U, \Lambda_{n})$.
	Then, if $U'$ is completely split at the primes $(t^{m} - \pi)$ for all $m \ge 1$,
	then it is trivial.
\end{Prop}

\begin{proof}
	We may assume that $n = 1$.
	Identifying $\Lambda$ with $\Lambda(1)$ over $U$,
	we want to show that the kernel of the natural map
		\[
				\varphi
			\colon
				A[1 / p]^{\times} / A[1 / p]^{\times p}
			\to
				\prod_{m \ge 1}
					\kappa(t^{m} - \pi)^{\times} / \kappa(t^{m} - \pi)^{\times p}
		\]
	is zero.
	This kernel is contained in $G := (1 + \ideal{m}) /  (1 + \ideal{m})^{p}$.
	Let $e_{A}$ be the absolute ramification index of $\Hat{A}_{\ideal{p}}$
	and set $f_{A} = p e_{A} / (p - 1)$.
	For non-negative integers $c, d$ not both zero,
	let $G^{c, d}$ be the image of $1 + \pi^{c} t^{d} A$ in $G$.
	Then $G^{c, d} = 0$ if $c \ge f_{A}$.
	Also $G^{c, d} \subset G^{c + 1, d} + G^{c, d + 1}$
	if $c < f_{A}$ and $p \mid \gcd(c, d)$.
	Hence it is enough to show that
	$G^{c, d} \cap \Ker(\varphi) \subset G^{c + 1, d} + G^{c, d + 1}$
	if $c < f_{A}$ and $p \nmid \gcd(c, d)$.
	Let $m > d / (f_{A} - c)$ be an integer such that $p \nmid m c + d$.
	For a positive integer $l$,
	let $G_{m}^{l}$ be the image of $1 + t^{l} A$ in
	$\kappa(t^{m} - \pi)^{\times} / \kappa(t^{m} - \pi)^{\times p}$.
	Since $m e_{A}$ is the absolute ramification index of $A / (t^{m} - \pi)$
	and $p \nmid m c + d < m f_{A}$,
	we have $G_{m}^{m c + d} / G_{m}^{m c + d + 1} \cong F$
	via $1 + t^{m c + d} x \mapsto x \mod t$.
	Hence the kernel of the composite
		$
				G^{c, d}
			\to
				G_{m}^{m c + d}
			\onto
				F
		$
	is contained in $G^{c + 1, d} + G^{c, d + 1}$.
\end{proof}

As in Section \ref{0086},
we take a resolution $\mathfrak{X} \to \Spec A$
such that $\mathfrak{X} \times_{A} A / p A \subset \mathfrak{X}$ is supported on a strict normal crossing divisor.
For any $n \ge 1$,
the natural map $H^{1}(\mathfrak{X}, \Lambda_{n}) \to H^{1}(X, \Lambda_{n})$ is injective
and the natural map $H^{1}(\mathfrak{X}, \Lambda_{n}) \to H^{1}(Y, \Lambda_{n})$ is an isomorphism.
We view $H^{1}(Y, \Lambda_{n})$ as a subgroup of $H^{1}(X, \Lambda_{n})$.

\begin{Prop} \label{0091}
	Assume $F = \closure{F}$.
	Let $n \ge 1$ and $X' \in H^{1}(X, \Lambda_{n})$.
	Then $X'$ is completely split if and only if
	$X' \in H^{1}(Y, \Lambda_{n})$.
	In particular, the subgroup $H^{1}(Y, \Lambda_{n})$ of $H^{1}(X, \Lambda_{n})$
	is independent of the resolution $\mathfrak{X}$.
\end{Prop}

\begin{proof}
	We may assume that $\zeta_{p} \in K$ by the natural exact sequences
		\begin{gather*}
					0
				\to
					H^{1}_{\cs}(X, \Lambda_{n})
				\to
					H^{1}(K, \Lambda_{n})
				\to
					\prod_{\ideal{p} \in P}
						H^{1}(K_{\ideal{p}}, \Lambda_{n}),
			\\
					0
				\to
					H^{1}(Y, \Lambda_{n})
				\to
					H^{1}(K, \Lambda_{n})
				\to
						\prod_{\ideal{p} \in P}
							\frac{
								H^{1}(K_{\ideal{p}}, \Lambda_{n})
							}{
								H^{1}(\kappa(\ideal{p}), \Lambda_{n})
							}
					\times
						\prod_{\eta \in Y_{1}}
							\frac{
								H^{1}(K_{\eta}, \Lambda_{n})
							}{
								H^{1}(\kappa(\eta), \Lambda_{n})
							}.
		\end{gather*}
	The ``if'' part is easy.
	Assume $X'$ is completely split.
	For any non-singular point $x \in Y_{0}$ that is not
	the specialization of any element of $P$ dividing $p$,
	the covering $X' \times_{X} \Hat{R}_{x}$ of $\Spec \Hat{R}_{x}$ is trivial
	by Propositions \ref{0060}.
	Hence $X'$ extends to a finite \'etale covering of $\mathfrak{X}$
	minus finitely many closed points and hence to a finite \'etale covering of $\mathfrak{X}$
	by the purity of branch locus (\cite[Tag 0BMA]{Sta22}).
\end{proof}

Define
	\[
			\alg{H}^{1}_{\cs}(\alg{X}, \Lambda_{n})
		:=
			\alg{H}^{1}(\alg{Y}, \Lambda_{n}),
	\]
the finite \'etale group scheme over $F$
with group of $\closure{F}$-points given by
$H^{1}(\alg{Y}(\closure{F}), \Lambda_{n})$,
where $\alg{Y}(\closure{F}) = Y \times_{F} \closure{F}$.
It is an $F$-subgroup scheme of $\alg{H}^{1}(\alg{X}, \Lambda_{n})$
independent of $\mathfrak{X}$
by Proposition \ref{0091}.
Also define
	\[
			\pi_{0}(\alg{H}^{2}(\alg{K}, \Lambda_{n}(2)))
		:=
			\dirlim_{S}
				\pi_{0}(\alg{H}^{2}(\alg{U}_{S}, \Lambda_{n}(2))),
	\]
where $S$ runs through all finite subsets of $P$.
It is an \'etale group over $F$.
Applying $\pi_{0}$ to the exact sequences
	\[
			\alg{H}^{2}(\alg{U}_{S}, \mathfrak{T}_{n}(2)))
		\to
			\bigoplus_{\ideal{p} \in S}
				\alg{H}^{3}_{c}(\Hat{\alg{A}}_{\ideal{p}}, \mathfrak{T}_{n}(2))
		\to
			\alg{H}^{3}(\alg{X}, \mathfrak{T}_{n}(2))
		\to
			0
	\]
for $S \subset P$
and using \eqref{0093}
(for $\ideal{p}$ dividing $p$ and \cite[Theorem 7.2.4]{Suz24} otherwise)
induce an exact sequence
	\begin{equation} \label{0062}
			\pi_{0}(\alg{H}^{2}(\alg{K}, \Lambda_{n}(2)))
		\to
			\bigoplus_{\ideal{p} \in S}
				\Weil_{F_{\ideal{p}} / F}
					\Lambda_{n}
		\to
			\Lambda_{n}
		\to
			0.
	\end{equation}

\begin{Prop}
	The kernel of the first morphism in \eqref{0062} is canonically isomorphic to
	the Pontryagin dual of $\alg{H}^{1}_{\cs}(\alg{X}, \Lambda_{n})$.
\end{Prop}

\begin{proof}
	Let $S \subset P$ be finite and non-empty.
	The distinguished triangle
		\[
				R \alg{\Gamma}_{c}(\alg{U}_{S}, \Lambda_{n})
			\to
				R \alg{\Gamma}(\alg{X}, \Lambda_{n})
			\to
				\bigoplus_{\ideal{p} \in S}
					R \alg{\Gamma}(\Hat{\alg{A}}_{\ideal{p}}, \Lambda_{n})
		\]
	induces an exact sequence
		\[
				0
			\to
				\Lambda_{n}
			\to
				\bigoplus_{\ideal{p} \in S}
					\Weil_{F_{\ideal{p}} / F}
						\Lambda_{n}
			\to
				\alg{H}^{1}_{c}(\alg{U}_{S}, \Lambda_{n})
			\to
				\alg{H}^{1}(\alg{X}, \Lambda_{n})
			\to
				\bigoplus_{\ideal{p} \in S}
					\alg{H}^{1}(\varalg{\kappa}(\ideal{p}), \Lambda_{n})
		\]
	Hence we have an exact sequence
		\[
				0
			\to
				\Lambda_{n}
			\to
				\prod_{\ideal{p} \in P}
					\Weil_{F_{\ideal{p}} / F}
						\Lambda_{n}
			\to
				\invlim_{S}
					\alg{H}^{1}_{c}(\alg{U}_{S}, \Lambda_{n})
			\to
				\alg{H}^{1}_{\cs}(\alg{X}, \Lambda_{n})
			\to
				0
		\]
	of pro-finite-\'etale groups over $F$.
	Taking the Pontryagin dual and applying Theorem \ref{0088}, we get the result.
\end{proof}

Let $\alg{\Gamma}_{Y}$ be the dual graph of $\alg{Y}(\closure{F})$.
It has a continuous action of $\Gal(\closure{F} / F)$.
Let $\alg{H}^{1}(\alg{\Gamma}_{Y}, \Lambda_{n})$ be the first cohomology
of $\alg{\Gamma}_{Y}$ with coefficients in $\Lambda_{n}$.
It is a finite \'etale group scheme over $F$.

\begin{Prop}
	We have an exact sequence
		\[
				0
			\to
				\alg{H}^{1}(\alg{\Gamma}_{Y}, \Lambda_{n})
			\to
				\alg{H}^{1}(\alg{Y}, \Lambda_{n})
			\to
				\bigoplus_{\eta \in Y_{1}}
					\alg{H}^{1}(\alg{Y}_{\eta}, \Lambda_{n})
			\to
				0.
		\]
\end{Prop}

\begin{proof}
	This follows from the distinguished triangle
		\[
				R \alg{\Gamma}(\alg{Y}, \Lambda_{n})
			\to
				\bigoplus_{\eta \in Y_{1}}
					R \alg{\Gamma}(\alg{Y}_{\eta}, \Lambda_{n})
			\to
				\bigoplus_{x \in Y_{\mathrm{sin}}}
					\Weil_{F_{x} / F}
						\Lambda_{n},
		\]
	where $Y_{\mathrm{sin}}$ denotes the set of singular points of $Y$.
\end{proof}

This finishes the proof of Theorem \ref{0109}.


\section{Lemmas on localization over resolutions}
\label{0117}

In the next section,
we prove some finiteness statement about $\Br(X)$
(Proposition \ref{0068}).
This uses some arguments local over a resolution of singularities of $A$,
which we deal with in this section.

We give a Hasse principle for Brauer groups of tubular neighborhoods:

\begin{Prop} \label{0067}
	Let $\Hat{K}_{\eta}$ be a complete discrete valuation field
	with ring of integer $\Hat{\Order}_{K_{\eta}}$ and residue field $\kappa(\eta)$.
	Assume that $\Hat{K}_{\eta}$ has characteristic zero, $\zeta_{p} \in \Hat{K}_{\eta}$ and
	$\kappa(\eta)$ is the function field
	of a proper smooth geometrically connected curve $Y_{\eta}$ over $F = \closure{F}$.
	Let $\eta_{0}$ be the set of closed points of $Y_{\eta}$.
	For each $x \in \eta_{0}$,
	let $\Hat{\kappa}(\eta_{x})$ be the complete local field of $Y_{\eta}$ at $x$.
	Let $\Hat{\Order}_{K_{\eta_{x}}}$ be the Kato canonical lifting of
	the relatively perfect $\kappa(\eta)$-algebra $\Hat{\kappa}(\eta_{x})$ to $\Hat{\Order}_{K_{\eta}}$.
	Let $\Hat{K}_{\eta_{x}}$ be the fraction field of $\Hat{\Order}_{K_{\eta_{x}}}$.
	
	Then the kernel of the natural map
		\begin{equation} \label{0064}
				\Br(\Hat{K}_{\eta})[p]
			\to
				\prod_{x \in \eta_{0}}
					\Br(\Hat{K}_{\eta_{x}})[p]
		\end{equation}
	is canonically isomorphic to $H^{1}(Y_{\eta}, \Lambda)$.
\end{Prop}

\begin{proof}
	Let $e_{\eta}$ be the absolute ramification index of $\Hat{K}_{\eta}$
	and set $f_{\eta} = p e_{\eta} / (p - 1)$.
	We have the usual filtrations by symbols on
	$\Br(\Hat{K}_{\eta})[p] \cong K_{2}(\Hat{K}_{\eta}) / p K_{2}(\Hat{K}_{\eta})$.
	Let $\gr^{m} \Br(\Hat{K}_{\eta})[p]$ be the $m$-th graded piece.
	Let $\gr^{m} \Br(\Hat{K}_{\eta_{x}})[p]$ similarly for each $x \in \eta_{0}$.
	We have
		\[
				\gr^{m} \Br(\Hat{K}_{\eta})[p]
			\cong
				\begin{cases}
						\kappa(\eta)^{\times} / \kappa(\eta)^{\times p}
					&	\text{if }
						m = 0,
					\\
						\kappa(\eta) / \kappa(\eta)^{p}
					&	\text{if }
						0 < m < f_{\eta},\, p \mid m,
					\\
						\Omega_{\kappa(\eta)}^{1}
					&	\text{if }
						0 < m < f_{\eta},\, p \nmid m,
					\\
						\kappa(\eta) / \wp(\kappa(\eta))
					&	\text{if }
						m = f_{\eta},
					\\
						0
					&	\text{else.}
				\end{cases}
		\]
	Similarly,
		\[
				\gr^{m} \Br(\Hat{K}_{\eta_{x}})[p]
			\cong
				\begin{cases}
						\Hat{\kappa}(\eta_{x})^{\times} / \Hat{\kappa}(\eta_{x})^{\times p}
					&	\text{if }
						m = 0,
					\\
						\Hat{\kappa}(\eta_{x}) / \Hat{\kappa}(\eta_{x})^{p}
					&	\text{if }
						0 < m < f_{\eta},\, p \mid m,
					\\
						\Omega_{\Hat{\kappa}(\eta_{x})}^{1}
					&	\text{if }
						0 < m < f_{\eta},\, p \nmid m,
					\\
						\Hat{\kappa}(\eta_{x}) / \wp(\Hat{\kappa}(\eta_{x}))
					&	\text{if }
						m = f_{\eta},
					\\
						0
					&	\text{else.}
				\end{cases}
		\]
	
	The map $\Br(\Hat{K}_{\eta})[p] \to \Br(\Hat{K}_{\eta_{x}})[p]$ preserves the filtration.
	The induced map on the $m$-th graded piece
	is injective except for the case $m = f_{\eta}$.
	In the case $m = f_{\eta}$, the map can be identified with
	the natural map $H^{1}(\kappa(\eta), \Lambda) \to H^{1}(\Hat{\kappa}(\eta_{x}), \Lambda)$.
	Hence the intersection of the kernels over all $x$ is isomorphic to $H^{1}(Y_{\eta}, \Lambda)$.
\end{proof}

Next, we study trace (or transfer) maps.
Assume that $\zeta_{p} \in A$
and take $S \subset P$ to be the set of primes dividing $p$.
Let $\mathfrak{X} \to \Spec A$, $W_{\eta} \subset \mathfrak{X}$ and $T \subset Y$
be as in Section \ref{0086}.
By \cite[(10.4.7), (10.5.1)]{Suz24}, we have a morphism of distinguished triangles
	\begin{equation} \label{0087}
		\begin{CD}
				R \alg{\Gamma}_{c}(\alg{U}_{S}, \Lambda)
			@>>>
				R \alg{\Gamma}(\alg{U}_{S}, \Lambda)
			@>>>
				\bigoplus_{\ideal{p} \in S}
					R \alg{\Gamma}(\Hat{\alg{K}}_{\ideal{p}}, \Lambda)
			\\ @| @VVV @VVV \\
				R \alg{\Gamma}_{c}(\alg{U}_{S}, \Lambda)
			@>>>
				\begin{array}{c}
						\bigoplus_{x \in T}
							R \alg{\Gamma}(\Hat{\alg{R}}_{x, S}, \Lambda)
					\\
					\oplus
						\bigoplus_{\eta \in Y_{1}}
							R \alg{\Gamma}(\Hat{\alg{R}}_{\eta, T}, \Lambda)
				\end{array}
			@>>>
				\begin{array}{c}
						\bigoplus_{
							\substack{
								x \in T \\
								\eta_{x} \in Y_{1}^{x}
							}
						}
							R \alg{\Gamma}(\Hat{\alg{K}}_{\eta_{x}}, \Lambda)
					\\
					\oplus
						\bigoplus_{\ideal{p} \in S}
							R \alg{\Gamma}(\Hat{\alg{K}}_{\ideal{p}}, \Lambda)
				\end{array}
		\end{CD}
	\end{equation}
in $D(F^{\ind\rat}_{\pro\et})$.
In particular, we have canonical morphisms
	\begin{equation} \label{0090}
			\alg{H}^{1}(\Hat{\alg{K}}_{\eta_{x}}, \Lambda)
		\to
			\alg{H}^{2}_{c}(\alg{U}_{S}, \Lambda)
		\to
			\alg{H}^{2}(\alg{X}, \mathfrak{T}(1))
	\end{equation}
for each pair $x \in T$ and $\eta_{x} \in Y_{1}^{x}$
(where the second morphism comes from $\zeta_{p} \colon \Lambda \to \mathfrak{T}(1)$).
Hence for any algebraically closed field $F' \in F^{\perar}$,
we obtain a homomorphism
	\begin{equation} \label{0122}
			H^{1}(\Hat{\alg{K}}_{\eta_{x}}(F'), \Lambda)
		\to
			H^{2}(\alg{X}(F'), \mathfrak{T}(1)).
	\end{equation}

\begin{Prop} \label{0066}
	Under the above setting,
	let $x \in T$ and $\eta_{x} \in Y_{1}^{x}$.
	Then the image of \eqref{0122} is contained in the subgroup
	$H^{1}(\alg{X}(F'), \Gm) \tensor \Lambda$
	of the target.
\end{Prop}

\begin{proof}
	We may assume that $F = F' = \closure{F}$.
	We want to show that
	the image of $H^{1}(\Hat{K}_{\eta_{x}}, \Lambda) \to H^{2}(X, \mathfrak{T}(1))$
	is contained in the subgroup
	$H^{1}(X, \Gm) \tensor \Lambda$ of the target.
	
	Let $R_{S} = A[1 / p]$.
	We first briefly recall the construction of the diagram \eqref{0087} on $F$-points.
	Let $R_{S, \Et}$ be the big \'etale site of $R_{S}$.
	View $\Gamma(R_{S}, \var)$, $\Gamma(\Hat{K}_{\ideal{p}}, \var)$, $\Gamma(\Hat{K}_{\eta_{x}}, \var)$
	$\Gamma(\Hat{R}_{x, S}, \var)$ and $\Gamma(\Hat{R}_{\eta, T}, \var)$ as
	the section functors $\Ab(R_{S, \Et}) \to \Ab$ at the objects
	$R_{S}$, $\Hat{K}_{\ideal{p}}$, $\Hat{K}_{\eta_{x}}$, $\Hat{R}_{x, S}$ and $\Hat{R}_{\eta, T}$,
	respectively.
	We have a commutative diagram
		\[
			\begin{CD}
					\Gamma(R_{S}, \var)
				@>>>
					\bigoplus_{\ideal{p} \in S}
						\Gamma(\Hat{K}_{\ideal{p}}, \var)
				\\ @VVV @VVV \\
					\begin{array}{c}
							\bigoplus_{x \in T}
								\Gamma(\Hat{R}_{x, S}, \var)
						\\
						\oplus
							\bigoplus_{\eta \in Y_{1}}
								\Gamma(\Hat{R}_{\eta, T}, \var)
					\end{array}
				@>>>
					\begin{array}{c}
							\bigoplus_{
								\substack{
									x \in T \\
									\eta_{x} \in Y_{1}^{x}
								}
							}
								\Gamma(\Hat{K}_{\eta_{x}}, \var)
						\\
						\oplus
							\bigoplus_{\ideal{p} \in S}
								\Gamma(\Hat{K}_{\ideal{p}}, \var)
					\end{array}
			\end{CD}
		\]
	of functors $\Ab(R_{S, \Et}) \to \Ab$,
	where the left vertical map is the diagonal
		\[
				a
			\mapsto
				\bigl(
					(a|_{\Hat{R}_{x, S}})_{x \in T},
					(a|_{\Hat{R}_{\eta, T}})_{\eta \in Y_{1}}
				\bigr),
		\]
	the lower horizontal map is
		\[
				((a_{x})_{x \in T}, (a_{\eta})_{\eta \in Y_{1}})
			\mapsto
				\bigl(
					(a_{x}|_{\Hat{K}_{\eta_{x}}} - a_{\eta}|_{\Hat{K}_{\eta_{x}}})_{x \in T, \eta_{x} \in Y_{1}^{x}},
					(a_{x}|_{\Hat{K}_{\ideal{p}}})_{\ideal{p} \in S}
				\bigr),
		\]
	the upper horizontal map is the diagonal
	and the right vertical map is the inclusion into the second summand.
	This diagram induces a commutative diagram of functors
	$\Ch(R_{S, \Et}) \to \Ch$ of additive categories with translation
	in the sense of \cite[Definition 10.1.1]{KS06}
	(where $\Ch$ is the category of complexes in $\Ab$).
	By taking the right derived functors (\cite[Sections 13.3, 14.3]{KS06}), it induces a morphism
		\begin{align*}
			&
					R \left[
							\Gamma(R_{S}, \var)
						\to
							\bigoplus_{\ideal{p} \in S}
								\Gamma(\Hat{K}_{\ideal{p}}, \var)
					\right][-1]
			\\
			&	\to
					R \left[
							\begin{array}{c}
									\bigoplus_{x \in T}
										\Gamma(\Hat{R}_{x, S}, \var)
								\\
								\oplus
									\bigoplus_{\eta \in Y_{1}}
										\Gamma(\Hat{R}_{\eta, T}, \var)
							\end{array}
						\to
							\begin{array}{c}
									\bigoplus_{
										\substack{
											x \in T \\
											\eta_{x} \in Y_{1}^{x}
										}
									}
										\Gamma(\Hat{K}_{\eta_{x}}, \var)
								\\
								\oplus
									\bigoplus_{\ideal{p} \in S}
										\Gamma(\Hat{K}_{\ideal{p}}, \var)
							\end{array}
					\right][-1]
		\end{align*}
	of triangulated functors $D(R_{S, \Et}) \to D(\Ab)$.
	It becomes an isomorphism
	when precomposed with the natural functor $D_{\tor}^{+}(R_{S, \et}) \to D(R_{S, \Et})$
	from the bounded below derived category of complexes with torsion cohomology,
	and the resulting isomorphic functors $D_{\tor}^{+}(R_{S, \et}) \rightrightarrows D(\Ab)$
	are further isomorphic to $R \Gamma_{c}(R_{S}, \var)$.
	In particular, we have a morphism of distinguished triangles
		\[
			\begin{CD}
					R \Gamma_{c}(R_{S}, \Lambda)
				@>>>
					R \Gamma(R_{S}, \Lambda)
				@>>>
					\bigoplus_{\ideal{p} \in S}
						R \Gamma(\Hat{K}_{\ideal{p}}, \Lambda)
				\\ @| @VVV @VVV \\
					R \Gamma_{c}(R_{S}, \Lambda)
				@>>>
					\begin{array}{c}
							\bigoplus_{x \in T}
								R \Gamma(\Hat{R}_{x, S}, \Lambda)
						\\
						\oplus
							\bigoplus_{\eta \in Y_{1}}
								R \Gamma(\Hat{R}_{\eta, T}, \Lambda)
					\end{array}
				@>>>
					\begin{array}{c}
							\bigoplus_{
								\substack{
									x \in T \\
									\eta_{x} \in Y_{1}^{x}
								}
							}
								R \Gamma(\Hat{K}_{\eta_{x}}, \Lambda)
						\\
						\oplus
							\bigoplus_{\ideal{p} \in S}
								R \Gamma(\Hat{K}_{\ideal{p}}, \Lambda),
					\end{array}
			\end{CD}
		\]
	which is the $F$-points of the diagram \eqref{0087}.%
	\footnote{
		The original construction in \cite[Section 4.6]{Suz24} uses fibered sites.
		For the purpose of proving Proposition \ref{0066} here,
		it is enough to use big \'etale sites,
		which gives the same diagram on $F$-points.
	}
	
	Let $P_{T} \subset P$ be the set of prime ideals $\ideal{p} \in P$
	that specializes to an element of $T$ in $\mathfrak{X}$.
	Let $S' \subset P_{T}$ be an arbitrary finite subset containing $S$.
	Let $R_{S'} = \Order(U_{S} \setminus S')$.
	We have a commutative diagram
		\begin{equation} \label{0102}
			\begin{CD}
					\Gamma(R_{S}, \var)
				@>>>
					\bigoplus_{\ideal{p} \in S}
						\Gamma(\Hat{K}_{\ideal{p}}, \var)
				\\ @VVV @VVV \\
					\Gamma(R_{S'}, \var)
				@>>>
					\begin{array}{c}
							\bigoplus_{\ideal{p} \in S}
								\Gamma(\Hat{K}_{\ideal{p}}, \var)
						\\
						\oplus
							\bigoplus_{\ideal{p} \in S' \setminus S}
								\Gamma_{c}(\Hat{A}_{\ideal{p}}, \var)[1],
					\end{array}
			\end{CD}
		\end{equation}
	of functors $\Ch(U_{S, \Et}) \to \Ch$ of additive categories with translation,
	where $\Gamma_{c}(\Hat{A}_{\ideal{p}}, \var)[1]$ is the mapping cone of
	$\Gamma(\Hat{A}_{\ideal{p}}, \var) \to \Gamma(\Hat{K}_{\ideal{p}}, \var)$.
	The induced morphism
		\begin{align*}
			&
					R \left[
							\Gamma(R_{S}, \var)
						\to
							\bigoplus_{\ideal{p} \in S}
								\Gamma(\Hat{K}_{\ideal{p}}, \var)
					\right][-1]
			\\
			&	\to
					R \left[
							\Gamma(R_{S'}, \var)
						\to
							\begin{array}{c}
									\bigoplus_{\ideal{p} \in S}
										\Gamma(\Hat{K}_{\ideal{p}}, \var)
								\\
								\oplus
									\bigoplus_{\ideal{p} \in S' \setminus S}
										\Gamma_{c}(\Hat{A}_{\ideal{p}}, \var)[1],
							\end{array}
					\right][-1]
		\end{align*}
	of triangulated functors $D(R_{S, \Et}) \to D(\Ab)$
	precomposed with $D_{\tor}^{+}(R_{S, \et}) \to D(R_{S, \Et})$
	is an isomorphism by excision.
	
	Similarly, we have a commutative diagram
		\begin{equation} \label{0103}
			\begin{CD}
					\begin{array}{c}
							\bigoplus_{x \in T}
								\Gamma(\Hat{R}_{x, S}, \var)
						\\
						\oplus
							\bigoplus_{\eta \in Y_{1}}
								\Gamma(\Hat{R}_{\eta, T}, \var)
					\end{array}
				@>>>
					\begin{array}{c}
							\bigoplus_{
								\substack{
									x \in T \\
									\eta_{x} \in Y_{1}^{x}
								}
							}
								\Gamma(\Hat{K}_{\eta_{x}}, \var)
						\\
						\oplus
							\bigoplus_{\ideal{p} \in S}
								\Gamma(\Hat{K}_{\ideal{p}}, \var)
					\end{array}
				\\ @VVV @VVV \\
					\begin{array}{c}
							\bigoplus_{x \in T}
								\Gamma(\Hat{R}_{x, S'}, \var)
						\\
						\oplus
							\bigoplus_{\eta \in Y_{1}}
								\Gamma(\Hat{R}_{\eta, T}, \var)
					\end{array}
				@>>>
					\begin{array}{c}
							\bigoplus_{
								\substack{
									x \in T \\
									\eta_{x} \in Y_{1}^{x}
								}
							}
								\Gamma(\Hat{K}_{\eta_{x}}, \var)
						\\
						\oplus
							\bigoplus_{\ideal{p} \in S}
								\Gamma(\Hat{K}_{\ideal{p}}, \var)
						\\
						\oplus
							\bigoplus_{\ideal{p} \in S' \setminus S}
								\Gamma_{c}(\Hat{A}_{\ideal{p}}, \var)[1].
					\end{array}
			\end{CD}
		\end{equation}
	The induced morphism
		\begin{align*}
			&
					R \left[
							\begin{array}{c}
									\bigoplus_{x \in T}
										\Gamma(\Hat{R}_{x, S}, \var)
								\\
								\oplus
									\bigoplus_{\eta \in Y_{1}}
										\Gamma(\Hat{R}_{\eta, T}, \var)
							\end{array}
						\to
							\begin{array}{c}
									\bigoplus_{
										\substack{
											x \in T \\
											\eta_{x} \in Y_{1}^{x}
										}
									}
										\Gamma(\Hat{K}_{\eta_{x}}, \var)
								\\
								\oplus
									\bigoplus_{\ideal{p} \in S}
										\Gamma(\Hat{K}_{\ideal{p}}, \var)
							\end{array}
					\right][-1]
			\\
			&	\to
					R \left[
							\begin{array}{c}
									\bigoplus_{x \in T}
										\Gamma(\Hat{R}_{x, S'}, \var)
								\\
								\oplus
									\bigoplus_{\eta \in Y_{1}}
										\Gamma(\Hat{R}_{\eta, T}, \var)
							\end{array}
						\to
							\begin{array}{c}
									\bigoplus_{
										\substack{
											x \in T \\
											\eta_{x} \in Y_{1}^{x}
										}
									}
										\Gamma(\Hat{K}_{\eta_{x}}, \var)
								\\
								\oplus
									\bigoplus_{\ideal{p} \in S}
										\Gamma(\Hat{K}_{\ideal{p}}, \var)
								\\
								\oplus
									\bigoplus_{\ideal{p} \in S' \setminus S}
										\Gamma_{c}(\Hat{A}_{\ideal{p}}, \var)[1].
							\end{array}
					\right][-1]
		\end{align*}
	of triangulated functors $D(R_{S, \Et}) \to D(\Ab)$
	precomposed with $D_{\tor}^{+}(R_{S, \et}) \to D(R_{S, \Et})$
	is an isomorphism by excision.
	
	We have a morphism of diagrams from \eqref{0102} to \eqref{0103}.
	It induces a commutative diagram
		\[
			\begin{CD}
					\frac{
						\bigoplus_{\ideal{p} \in S}
							H^{1}(\Hat{K}_{\ideal{p}}, \Lambda)
					}{
						H^{1}(R_{S}, \Lambda)
					}
				@>>>
					\frac{
							\bigoplus_{\ideal{p} \in S}
								H^{1}(\Hat{K}_{\ideal{p}}, \Lambda)
						\oplus
							\bigoplus_{\ideal{p} \in S' \setminus S}
								H^{2}_{c}(\Hat{A}_{\ideal{p}}, \Lambda),
					}{
						H^{1}(R_{S'}, \Lambda)
					}
				\\ @VVV @VVV \\
					\frac{
							\bigoplus_{
								\substack{
									\scriptscriptstyle
									x \in T \\
									\scriptscriptstyle
									\eta_{x} \in Y_{1}^{x}
								}
							}
								H^{1}(\Hat{K}_{\eta_{x}}, \Lambda)
						\oplus
							\bigoplus_{\ideal{p} \in S}
								H^{1}(\Hat{K}_{\ideal{p}}, \Lambda)
					}{
							\bigoplus_{x \in T}
								H^{1}(\Hat{R}_{x, S}, \Lambda)
						\oplus
							\bigoplus_{\eta \in Y_{1}}
								H^{1}(\Hat{R}_{\eta, T}, \Lambda)
					}
				@>>>
					\frac{
							\bigoplus_{
								\substack{
									\scriptscriptstyle
									x \in T \\
									\scriptscriptstyle
									\eta_{x} \in Y_{1}^{x}
								}
							}
								H^{1}(\Hat{K}_{\eta_{x}}, \Lambda)
						\oplus
							\bigoplus_{\ideal{p} \in S}
								H^{1}(\Hat{K}_{\ideal{p}}, \Lambda)
						\oplus
							\bigoplus_{\ideal{p} \in S' \setminus S}
								H^{2}_{c}(\Hat{A}_{\ideal{p}}, \Lambda).
					}{
							\bigoplus_{x \in T}
								H^{1}(\Hat{R}_{x, S'}, \Lambda)
						\oplus
							\bigoplus_{\eta \in Y_{1}}
								H^{1}(\Hat{R}_{\eta, T}, \Lambda)
					}
			\end{CD}
		\]
	of injective maps, with all the terms mapping injectively to
	$H^{2}_{c}(R_{S}, \Lambda)$.
	For a ring $Q$, denote $Q^{\Bar{\times}} = Q^{\times} / Q^{\times p}$.
	Let $R_{T}$ be the direct limit of $R_{S'}$ over increasing $S'$.
	Taking the direct limit of the above diagram in increasing $S'$,
	we obtain a commutative diagram
		\[
			\begin{CD}
					\frac{
						\bigoplus_{\ideal{p} \in S}
							\Hat{K}_{\ideal{p}}^{\Bar{\times}}
					}{
						H^{1}(R_{S}, \Lambda)
					}
				@>>>
					\frac{
							\bigoplus_{\ideal{p} \in S}
								\Hat{K}_{\ideal{p}}^{\Bar{\times}}
						\oplus
							\bigoplus_{\ideal{p} \in P_{T} \setminus S}
								\Lambda,
					}{
						H^{1}(R_{T}, \Lambda)
					}
				\\ @VVV @VVV \\
					\frac{
							\bigoplus_{
								\substack{
									\scriptscriptstyle
									x \in T \\
									\scriptscriptstyle
									\eta_{x} \in Y_{1}^{x}
								}
							}
								\Hat{K}_{\eta_{x}}^{\Bar{\times}}
						\oplus
							\bigoplus_{\ideal{p} \in S}
								\Hat{K}_{\ideal{p}}^{\Bar{\times}}
					}{
							\bigoplus_{x \in T}
								\Hat{R}_{x, S}^{\Bar{\times}}
						\oplus
							\bigoplus_{\eta \in Y_{1}}
								H^{1}(\Hat{R}_{\eta, T}, \Lambda)
					}
				@>>>
					\frac{
							\bigoplus_{
								\substack{
									\scriptscriptstyle
									x \in T \\
									\scriptscriptstyle
									\eta_{x} \in Y_{1}^{x}
								}
							}
								\Hat{K}_{\eta_{x}}^{\Bar{\times}}
						\oplus
							\bigoplus_{\ideal{p} \in S}
								\Hat{K}_{\ideal{p}}^{\Bar{\times}}
						\oplus
							\bigoplus_{\ideal{p} \in P_{T} \setminus S}
								\Lambda.
					}{
							\bigoplus_{x \in T}
								\Hat{K}_{x}^{\Bar{\times}}
						\oplus
							\bigoplus_{\eta \in Y_{1}}
								H^{1}(\Hat{R}_{\eta, T}, \Lambda)
					}
			\end{CD}
		\]
	of injective maps, with all the terms mapping injectively to
	$H^{2}_{c}(R_{S}, \Lambda)$.
	For any $x \in T$, the map
		$
				\Hat{K}_{x}^{\Bar{\times}}
			\to
				\bigoplus_{\eta_{x} \in Y_{1}^{x}}
					\Hat{K}_{\eta_{x}}^{\Bar{\times}}
		$
	is surjective by the approximation lemma.
	Therefore the right vertical map is an isomorphism.
	Hence for any $x \in T$ and $\eta_{x} \in Y_{1}^{x}$,
	the image of the map
	$\Hat{K}_{\eta_{x}}^{\Bar{\times}} \to H^{2}_{c}(R_{S}, \Lambda)$
	is contained in the image of the map
		$
					\bigoplus_{\ideal{p} \in S}
						\Hat{K}_{\ideal{p}}^{\Bar{\times}}
				\oplus
					\bigoplus_{\ideal{p} \in P_{T} \setminus S}
						\Lambda
			\to
				H^{2}_{c}(R_{S}, \Lambda)
		$.
	Hence the image of the map
	$\Hat{K}_{\eta_{x}}^{\Bar{\times}} \to H^{2}(X, \mathfrak{T}(1))$
	is contained in the image of the map
		$
				\bigoplus_{\ideal{p} \in P_{T}}
					\Lambda
			\to
				H^{2}(X, \mathfrak{T}(1))
		$.
	The latter map factors through $H^{1}(X, \Gm) \tensor \Lambda$
	via the divisor map.
	This proves the proposition.
\end{proof}


\section{Picard-Brauer duality}
\label{0118}

In this section, we prove Theorem \ref{0083}.
We deduce it basically from Theorems \ref{0077} and \ref{0088}.
We also give a duality result for $H^{3}(X, \Gm)$
(Proposition \ref{0110})
and for $\pi_{0}$ of $\Pic(X)$
(Proposition \ref{0111}).

\begin{Prop} \label{0072}
	$\pi_{\alg{X}, \ast} \Gm = \alg{A}^{\times}$ is
	$\Gm$ times a connected pro-algebraic group in $\mathcal{W}_{F}$.
	More specifically,
	the natural reduction morphism $\alg{A}^{\times} \to \Gm$
	and the Teichm\"uller section gives a direct summand $\Gm$ of $\alg{A}^{\times}$.
	The quotient $\alg{A}^{\times} / \Gm$ is in $\mathcal{W}_{F}$.
\end{Prop}

\begin{proof}
	For $F' \in F^{\perar}$ and $n \ge 1$,
	let $\alg{A}_{n}(F') = \alg{A}(F') \tensor_{A} A / \ideal{m}^{n}$.
	Then $\alg{A}^{\times} = \invlim_{n} \alg{A}_{n}^{\times}$.
	The sheaf $\alg{A}_{n}^{\times}$ is represented by
	the perfection of the Greenberg transform (or realization) of $\Gm$ over $A / \ideal{m}^{n}$
	(\cite[Proposition-Definition 6.2]{BGA18Green}).
	The same proof as \cite[Proposition 11.1]{BGA18Green}
	(see also \cite[Proposition 5.6]{BS20}) shows that
	the kernel of the reduction morphism $\alg{A}_{n + 1}^{\times} \onto \alg{A}_{n}^{\times}$
	is the perfection of a vector group over $F$.
	Hence the kernel of $\alg{A}^{\times} \onto \Gm$, or $\alg{A}^{\times} / \Gm$,
	is in $\mathcal{W}_{F}$.
\end{proof}

Given a resolution of singularities $\mathfrak{X} \to \Spec A$
such that the reduced part $Y$ of $\mathfrak{X} \times_{A} F$ is supported on a strict normal crossing divisor,
we have the intersection pairing
	\begin{equation} \label{0104}
					\bigoplus_{\eta \in Y_{1}}
						\Z Y_{\eta}
				\times
					\bigoplus_{\eta \in Y_{1}}
						\Z Y_{\eta}
			\to
				\Z,
		\quad
				(Y_{\eta}, Y_{\eta'})
			\mapsto
				Y_{\eta} \cdot Y_{\eta'},
	\end{equation}
where $Y_{1}$ is the set of generic points of $Y$
and $Y_{\eta}$ the closure of $\eta$.
See \cite[Section 13]{Lip69} for the details.
This pairing is negative-definite by \cite[Lemma (14.1)]{Lip69}.
Let
	\[
			\delta_{Y}
		=
			\frac{
				\Hom(
					\bigoplus_{\eta \in Y_{1}}
						\Z Y_{\eta},
					\Z
				)
			}{
				\bigoplus_{\eta \in Y_{1}}
					\Z Y_{\eta}
			}
	\]
be its discriminant group.
Applying this construction to the resolution
$\mathfrak{X} \times_{A} \alg{A}(\closure{F}) \to \Spec \alg{A}(\closure{F})$,
we obtain a finite Galois module $\delta_{Y \times_{F} \closure{F}}$ over $F$.
Denote the resulting finite \'etale group scheme over $F$ by $\varalg{\delta}_{Y}$.

\begin{Prop} \label{0044}
	We have $R^{1} \pi_{\alg{X}, \ast} \Gm \in \Alg / F$.
	For a resolution of singularities $\mathfrak{X} \to \Spec A$
	such that the reduced part $Y$ of $\mathfrak{X} \times_{A} F$ is supported on a strict normal crossing divisor,
	we have $\pi_{0} R^{1} \pi_{\alg{X}, \ast} \Gm \cong \varalg{\delta}_{Y}$.
\end{Prop}

\begin{proof}
	Let $\mathfrak{X} \to \Spec A$ be as in the statement.
	First, assume that $F = \closure{F}$.
	We recall the constructions made in
	\cite[Section 14]{Lip69} and \cite[Proof of Lemma (7.2)]{Sai86}.
	Let $\Pic(\mathfrak{X}) \to \Hom(\bigoplus_{\eta \in Y_{1}} \Z Y_{\eta}, \Z)$
	be the map given by the intersection pairing.
	Let $\Pic^{0}(\mathfrak{X})$ be its kernel.
	Then the resulting sequence
		\[
				0
			\to
				\Pic^{0}(\mathfrak{X})
			\to
				\Pic(X)
			\to
				\delta_{Y}
			\to
				0
		\]
	is exact.
	For $n \ge 1$, let $\mathfrak{X}_{n} = \mathfrak{X} \times_{A} A / \ideal{m}^{n}$.
	Let $\Pic^{0}(\mathfrak{X}_{n})$ be the kernel of a similar map
	$\Pic(\mathfrak{X}_{n}) \to \Hom(\bigoplus_{\eta \in Y_{1}} \Z Y_{\eta}, \Z)$.
	The group $\Pic^{0}(Y) = \Pic^{0}(\mathfrak{X}_{1})$ is
	the group of $F$-valued points of the semi-abelian variety $\Pic^{0}_{Y / F}$.
	The natural map $\Pic^{0}(\mathfrak{X}_{n + 1}) \to \Pic^{0}(\mathfrak{X}_{n})$ is surjective
	and its kernel has a natural structure as a finite-dimensional $F$-vector space.
	There exists a positive integer $n_{0}$ such that
	the natural map $\Pic^{0}(\mathfrak{X}) \to \Pic^{0}(\mathfrak{X}_{n})$
	is an isomorphism for all $n \ge n_{0}$.
	
	Next, let $F$ be general.
	The above constructions are functorial in residue field extensions.
	That is, for any algebraically closed field $F' \in F^{\perar}$,
	we can apply the above constructions for
	$\varalg{\mathfrak{X}}(F') = \mathfrak{X} \times_{A} \alg{A}(F') \to \Spec \alg{A}(F')$.
	This gives an exact sequence
		\[
				0
			\to
				\Pic^{0}(\varalg{\mathfrak{X}}(F'))
			\to
				\Pic(\alg{X}(F'))
			\to
				\delta_{\alg{Y}(F')}
			\to
				0
		\]
	and quotients $\Pic^{0}(\varalg{\mathfrak{X}}_{n}(F'))$ of $\Pic^{0}(\varalg{\mathfrak{X}}(F'))$
	functorial in $F'$
	(where $\alg{Y}(F') = Y' \times_{F} F'$
	and $\varalg{\mathfrak{X}}_{n}(F') = \mathfrak{X}_{n} \times_{A} \alg{A}(F')$).
	
	The sheaf $R^{1} \pi_{\alg{X}, \ast} \Gm$ sends a field $F' \in F^{\perar}$
	to the $\Gal(\closure{F'} / F')$-invariant part of $\Pic(\alg{X}(\closure{F'}))$.
	The above shows that this sheaf has a finite filtration whose graded pieces are
	sheaves that send a field $F' \in F^{\perar}$ to the $\Gal(\closure{F'} / F')$-invariant parts of
		\begin{enumerate}
			\item
				$\delta_{\alg{Y}(\closure{F'})}$,
			\item
				$\Pic^{0}(\alg{Y}(\closure{F'}))$ or
			\item
					$
						\Ker \bigl(
								\Pic^{0}(\varalg{\mathfrak{X}}_{n + 1}(\closure{F'}))
							\onto
							\Pic^{0}(\varalg{\mathfrak{X}}_{n}(\closure{F'}))
						\bigr)
					$
				for $n < n_{0}$.
		\end{enumerate}
	In this list, the first sheaf is a finite \'etale group (which is $\varalg{\delta}_{Y}$),
	the second a semi-abelian variety
	and the third a vector group.
	Being an extension of such groups,
	the sheaf $R^{1} \pi_{\alg{X}, \ast} \Gm$ is in $\Alg / F$
	with component group $\varalg{\delta}_{Y}$.
\end{proof}

\begin{Prop} \label{0045}
	Let $n \ge 1$.
	Then $R^{0} \pi_{\alg{X}, \ast} \mathfrak{T}_{n}(1)$ is a finite \'etale group and
	$R^{1} \pi_{\alg{X}, \ast} \mathfrak{T}_{n}(1)$ is pro-algebraic in $\mathcal{W}_{F}$.
\end{Prop}

\begin{proof}
	The sheaf $R^{0} \pi_{\alg{X}, \ast} \mathfrak{T}_{n}(1)$ is the finite \'etale group of
	$p^{n}$-th roots of unity in $\alg{A}(\closure{F})$.
	We have an exact sequence
		\[
				0
			\to
				\alg{A}^{\times} / \alg{A}^{\times p^{n}}
			\to
				R^{1} \pi_{\alg{X}, \ast} \mathfrak{T}_{n}(1)
			\to
				(R^{1} \pi_{\alg{X}, \ast} \Gm)[p^{n}]
			\to
				0
		\]
	in $\Ab(F^{\perar}_{\et})$.
	The fourth term $(R^{1} \pi_{\alg{X}, \ast} \Gm)[p^{n}]$ is in $\Alg_{u} / F$
	by Proposition \ref{0044}.
	We have an exact sequence
		\[
				0
			\to
				R^{0} \pi_{\alg{X}, \ast} \mathfrak{T}_{n}(1)
			\to
				\alg{A}^{\times} / \Gm
			\stackrel{p}{\to}
				\alg{A}^{\times} / \Gm
			\to
				\alg{A}^{\times} / \alg{A}^{\times p^{n}}
			\to
				0
		\]
	in $\Ab(F^{\perar}_{\et})$.
	The sheaf $\alg{A}^{\times} / \Gm$ is pro-algebraic in $\mathcal{W}_{F}$
	by Proposition \ref{0072}.
	Therefore $\alg{A}^{\times} / \alg{A}^{\times p^{n}}$ is pro-algebraic in $\mathcal{W}_{F}$.
\end{proof}

\begin{Prop} \label{0073} \mbox{}
	\begin{enumerate}
		\item \label{0100}
			The object $R \pi_{\alg{X}, \ast} \Gm$ is $h$-acyclic.
		\item \label{0101}
			Let $q \ge 2$.
			Then the sheaf $R^{q} \pi_{\alg{X}, \ast} \Gm$ is torsion.
			We have
				\[
						R^{q} \pi_{\alg{X}, \ast} \mathfrak{T}_{\infty}(1)
					\isomto
						(R^{q} \pi_{\alg{X}, \ast} \Gm)[p^{\infty}],
				\]
			where $\mathfrak{T}_{\infty}(1) = \dirlim_{n} \mathfrak{T}_{n}(1)$.
			For any prime $l \ne p$, the sheaf
			$(R^{q} \pi_{\alg{X}, \ast} \Gm)[l^{\infty}]$ is the \'etale group over $F$
			with group of $\closure{F}$-points given by
			$H^{q}(\alg{X}(\closure{F}), \Q_{l} / \Z_{l}(1))$.
	\end{enumerate}
\end{Prop}

\begin{proof}
	The truncation $\tau_{\le 1} R \pi_{\alg{X}, \ast} \Gm$ is $h$-acyclic
	by Propositions \ref{0072} and \ref{0044} and \cite[Proposition 7.2]{Suz21Imp}.
	The sheaf $R^{q} \pi_{\alg{X}, \ast} \Gm$ is torsion for $q \ge 2$
	by \cite[Proposition 1.4]{Gro95Br}.
	Since $(R^{1} \pi_{\alg{X}, \ast} \Gm) \tensor \Q / \Z = 0$
	by Proposition \ref{0044},
	we have
		\[
					R^{q} \pi_{\alg{X}, \ast} \mathfrak{T}_{\infty}(1)
				\oplus
					\bigoplus_{l \ne p}
						R^{q} \pi_{\alg{X}, \ast} \Q_{l} / \Z_{l}(1)
				\isomto
					R^{q} \pi_{\alg{X}, \ast} \Gm
		\]
	for $q \ge 2$.
	For any field $F' \in F^{\perar}$, the map $A \to \alg{A}(F')$ is regular
	and hence ind-smooth by Popescu's theorem.
	With \cite[Introduction, Theorem 1]{ILO14},
	we know that $R^{q} \pi_{\alg{X}, \ast} \Z / l^{n} \Z(1)$ is finite \'etale
	for all $q$ and all primes $l \ne p$.
	This implies that $R^{q} \pi_{\alg{X}, \ast} \Q_{l} / \Z_{l}(1)$ is $h$-acyclic
	by \cite[Proposition 3.11]{Suz21Imp}.
	For any $n \ge 1$, the object $\tau_{\ge 2} R \pi_{\alg{X}, \ast} \mathfrak{T}_{n}(1)$ is
	in $\genby{\mathcal{W}_{F}}_{F^{\perar}_{\et}}$ and $h$-acyclic
	by \cite[Propositions 10.1.5 and Proposition 3.1.4]{Suz24}
	and Proposition \ref{0045}.
	Since $R \alg{\Gamma}(\alg{X}, \mathfrak{T}_{n}(1)) \in D^{b}(\Ind \Pro \Alg_{u} / F)$ is
	concentrated in degrees $\ge 0$ for all $n$ by Proposition \ref{0046},
	we know that $L h^{\ast} R \pi_{\alg{X}, \ast} \mathfrak{T}_{n}(1) \in D^{b}(\Ind \Pro \Alg_{u} / F)$
	is concentrated in degrees $\ge 0$
	by \cite[Proposition 7.6]{Suz21Imp}.
	In particular, $\{L h^{\ast} R \pi_{\alg{X}, \ast} \mathfrak{T}_{n}(1)\}_{n \ge 1}$
	as an object of the derived category of
	the ind-category of $\Ab(F^{\perf}_{\pro\fppf})$ is bounded below.
	On the other hand, $\{L h^{\ast} \tau_{\le 1} R \pi_{\alg{X}, \ast} \mathfrak{T}_{n}(1)\}_{n \ge 1}$
	is concentrated in degrees $0$ and $1$
	by Proposition \ref{0045}.
	Hence $\{L h^{\ast} \tau_{\ge 2} R \pi_{\alg{X}, \ast} \mathfrak{T}_{n}(1)\}_{n \ge 1}$
	is bounded below.
	The functor $R h_{\ast}$ commutes with $\dirlim_{n}$ on the bounded below derived categories
	by \cite[Proposition (2.2.4)]{Suz20}.
	The functor $L h^{\ast}$ commutes with $\dirlim_{n}$
	by \cite[Lemma 3.7.2]{Suz22} and \cite[Corollary 14.4.6 (ii)]{KS06}.
	Therefore
		\[
				R h_{\ast}
				L h^{\ast}
				\tau_{\ge 2}
				R \pi_{\alg{X}, \ast} \mathfrak{T}_{\infty}(1)
			\cong
				\dirlim_{n}
				R h_{\ast}
				L h^{\ast}
				\tau_{\ge 2}
				R \pi_{\alg{X}, \ast} \mathfrak{T}_{n}(1)
			\cong
				\tau_{\ge 2}
				R \pi_{\alg{X}, \ast} \mathfrak{T}_{\infty}(1),
		\]
	showing the $h$-acyclicity of $\tau_{\ge 2} R \pi_{\alg{X}, \ast} \mathfrak{T}_{\infty}(1)$.
	Thus $\tau_{\ge 2} R \pi_{\alg{X}, \ast} \Gm$ is $h$-acyclic.
\end{proof}

By Proposition \ref{0073} \eqref{0100},
for any field $F' \in F^{\perar}$ with algebraic closure $\closure{F'}$ and any $q \in \Z$,
we have
	\[
			\alg{H}^{q}(\alg{X}, \Gm)(F')
		\cong
			H^{q}(\alg{X}(\closure{F'}), \Gm)^{\Gal(\closure{F'} / F')}
	\]
functorially in $F'$.

\begin{Prop} \label{0076}
	$\alg{\Gamma}(\alg{X}, \Gm) \cong \alg{A}^{\times}$ is
	$\Gm$ times a connected pro-algebraic group in $\mathcal{W}_{F}$,
	$\alg{H}^{1}(\alg{X}, \Gm)$ is in $\Alg / F$,
	$\alg{H}^{2}(\alg{X}, \Gm)$ and $\alg{H}^{3}(\alg{X}, \Gm)$ are torsion
	with $p$-primary part in $\Ind \Alg_{u} / F$
	and \'etale $l$-primary part for $l \ne p$
	and $\alg{H}^{q}(\alg{X}, \Gm) = 0$ for all $q \ne 0, 1, 2, 3$.
	We have
		\[
				\alg{H}^{q}(\alg{X}, \mathfrak{T}_{\infty}(1))
			\isomto
				\alg{H}^{q}(\alg{X}, \Gm)[p^{\infty}]
		\]
	for all $q \ge 2$.
\end{Prop}

\begin{proof}
	This follows from Propositions \ref{0015}, \ref{0046}, \ref{0072}, \ref{0044} and \ref{0073}.
\end{proof}

Recall from \cite[Section 2.1]{SuzCurve} that
an object $G \in \Ind \Alg_{u} / F$ is said to be of cofinite type
if $G[p^{n}] \in \Alg_{u} / F$ for all $n \ge 1$.
For a cofinite type $G$,
let $G_{\divis}^{0}$ be the maximal divisible and connected part of $G$
(\cite[Proposition 2.2.2]{SuzCurve}; the notation there was $G_{0 \divis}$).

\begin{Prop}
	The object $\alg{H}^{2}(\alg{X}, \Gm)[p^{\infty}] \in \Ind \Alg_{u} / F$ is of cofinite type.
\end{Prop}

\begin{proof}
	For any $n \ge 1$, we have an exact sequence
		\[
				0
			\to
				\alg{H}^{1}(\alg{X}, \Gm) \tensor \Lambda_{n}
			\to
				\alg{H}^{2}(\alg{X}, \mathfrak{T}_{n}(1))
			\to
				\alg{H}^{2}(\alg{X}, \Gm)[p^{n}]
			\to
				0.
		\]
	The direct limit of the first term in $n$ is zero
	by Proposition \ref{0076}.
	As $\alg{H}^{2}(\alg{X}, \mathfrak{T}_{n}(1)) \in \Alg_{u} / F$ by Proposition \ref{0046},
	we get the result.
\end{proof}

\begin{Prop} \label{0068}
	We have $\alg{H}^{2}(\alg{X}, \Gm)^{0} \in \Alg_{u} / F$.
\end{Prop}

\begin{proof}
	It is enough to show that $\alg{H}^{2}(\alg{X}, \Gm)_{\divis}^{0} = 0$
	by \cite[Proposition 2.2.3]{SuzCurve}.
	We may assume $\zeta_{p} \in A$ and $F = \closure{F}$.
	We have
		\[
				\alg{H}^{2}(\alg{X}, \Gm)^{0}[p^{\infty}]
			\cong
				\dirlim_{n}
				\sheafext^{1}_{F^{\ind\rat}_{\pro\et}} \bigl(
					\alg{H}^{2}(\alg{X}, \mathfrak{T}_{n}(1)),
					\Lambda_{\infty}
				\bigr)
		\]
	by duality.
	We have a natural morphism from the right-hand side to
		\[
				\dirlim_{n}
				\sheafext^{1}_{F^{\ind\rat}_{\pro\et}} \bigl(
					\alg{H}^{1}(\alg{X}, \Gm) \tensor \Lambda_{n},
					\Lambda_{\infty}
				\bigr)
			\cong
				\sheafext^{1}_{F^{\ind\rat}_{\pro\et}} \bigl(
					\alg{H}^{1}(\alg{X}, \Gm),
					\Lambda_{\infty}
				\bigr)^{0}.
		\]
	This final group is in $\Alg_{u} / F$ by Proposition \ref{0044}
	and hence killed by a power of $p$.
	Hence the composite of the above two morphisms maps
	$\alg{H}^{2}(\alg{X}, \Gm)_{\divis}^{0}$ to zero.
	Let $G$ be the inverse image of the subgroup
	$\alg{H}^{2}(\alg{X}, \Gm)_{\divis}^{0}[p]^{0}$
	under the natural morphism
		\[
				\alg{H}^{2}(\alg{X}, \mathfrak{T}(1))^{0}
			\to
				\alg{H}^{2}(\alg{X}, \Gm)^{0}.
		\]
	The image of $G^{0}$ in $\alg{H}^{2}(\alg{X}, \Gm)$ is
	$\alg{H}^{2}(\alg{X}, \Gm)_{\divis}^{0}[p]^{0}$.
	We have natural isomorphism and morphism
		\[
				\alg{H}^{2}(\alg{X}, \mathfrak{T}(1))^{0}
			\cong
				\sheafext^{1}_{F^{\ind\rat}_{\pro\et}} \bigl(
					\alg{H}^{2}(\alg{X}, \mathfrak{T}(1)),
					\Lambda_{\infty}
				\bigr)
			\to
				\sheafext^{1}_{F^{\ind\rat}_{\pro\et}} \bigl(
					\alg{H}^{1}(\alg{X}, \Gm) \tensor \Lambda,
					\Lambda_{\infty}
				\bigr).
		\]
	We have a commutative diagram
		\[
			\begin{CD}
					\alg{H}^{2}(\alg{X}, \mathfrak{T}(1))^{0}
				@>>>
					\sheafext^{1}_{F^{\ind\rat}_{\pro\et}} \bigl(
						\alg{H}^{1}(\alg{X}, \Gm) \tensor \Lambda,
						\Lambda_{\infty}
					\bigr)
				\\ @VVV @VVV \\
					\alg{H}^{2}(\alg{X}, \Gm)^{0}[p^{\infty}]
				@>>>
					\sheafext^{1}_{F^{\ind\rat}_{\pro\et}} \bigl(
						\alg{H}^{1}(\alg{X}, \Gm),
						\Lambda_{\infty}
					\bigr).
			\end{CD}
		\]
	The image of $G$ in the right lower term is zero.
	The right vertical morphism has finite kernel.
	Hence the image of $G^{0}$ in the right upper term is zero.
	
	Let $S \subset P$ be the set of primes dividing $p$.
	Let $\mathfrak{X} \to \Spec A$, $W_{\eta} \subset \mathfrak{X}$ and $T \subset Y$
	be as in Section \ref{0086}.
	Let $x \in Y_{0}$ and $\eta_{x} \in Y_{1}^{x}$.
	Take $T$ large enough so that $x \in T$.
	We have $\alg{H}^{1}(\Hat{\alg{K}}_{\eta_{x}}, \Lambda) \in \mathcal{W}_{F}$
	and $\alg{H}^{2}(\alg{X}, \mathfrak{T}(1)), \alg{H}^{1}(\alg{X}, \Gm) \tensor \Lambda \in \Alg_{u} / F$.
	Hence the image of
	$\alg{H}^{1}(\Hat{\alg{K}}_{\eta_{x}}, \Lambda) \to \alg{H}^{2}(\alg{X}, \mathfrak{T}(1))$
	is in $\Alg_{u} / F$
	by Propositions \ref{0089} and \ref{0024}.
	Therefore by Proposition \ref{0066}, the image of the composite
		\[
				\alg{H}^{1}(\Hat{\alg{K}}_{\eta_{x}}, \Lambda(1))
			\to
				\alg{H}^{2}(\alg{X}, \mathfrak{T}(1))
		\]
	of \eqref{0090} is contained in the subgroup
	$\alg{H}^{1}(\alg{X}, \Gm) \tensor \Lambda$ of the target.
	Hence the dual morphism
		\[
				\sheafext^{1}_{F^{\ind\rat}_{\pro\et}} \bigl(
					\alg{H}^{2}(\alg{X}, \mathfrak{T}(1)),
					\Lambda_{\infty}
				\bigr)
			\to
				\sheafext^{1}_{F^{\ind\rat}_{\pro\et}} \bigl(
					\alg{H}^{1}(\Hat{\alg{K}}_{\eta_{x}}, \Lambda(1)),
					\Lambda_{\infty}
				\bigr)
		\]
	factors through
		$
			\sheafext^{1}_{F^{\ind\rat}_{\pro\et}} \bigl(
				\alg{H}^{1}(\alg{X}, \Gm) \tensor \Lambda,
				\Lambda_{\infty}
			\bigr)
		$.
	Therefore the image of $G^{0}$ in 
		$
			\sheafext^{1}_{F^{\ind\rat}_{\pro\et}} \bigl(
				\alg{H}^{1}(\Hat{\alg{K}}_{\eta_{x}}, \Lambda(1)),
				\Lambda_{\infty}
			\bigr)
		$
	is zero.
	By the diagram in \cite[Section 10.6, Proof of Proposition 10.1.5]{Suz24},
	we have a commutative diagram
		\[
			\begin{CD}
					\alg{H}^{2}(\alg{X}, \mathfrak{T}(1))^{0}
				@> \sim >>
					\sheafext^{1}_{F^{\ind\rat}_{\pro\et}} \bigl(
						\alg{H}^{2}(\alg{X}, \mathfrak{T}(1)),
						\Lambda_{\infty}
					\bigr)
				\\ @VVV @VVV \\
					\alg{H}^{2}(\alg{U}_{S}, \Lambda(1))^{0}
				@> \sim >>
					\sheafext^{1}_{F^{\ind\rat}_{\pro\et}} \bigl(
						\alg{H}^{2}_{c}(\alg{U}_{S}, \Lambda(1)),
						\Lambda_{\infty}
					\bigr).
			\end{CD}
		\]
	By the diagram in the proof of \cite[Proposition 10.5.5]{Suz24},
	we have a commutative diagram
		\[
			\begin{CD}
					\alg{H}^{2}(\alg{U}_{S}, \Lambda(1))^{0}
				@> \sim >>
					\sheafext^{1}_{F^{\ind\rat}_{\pro\et}} \bigl(
						\alg{H}^{2}_{c}(\alg{U}_{S}, \Lambda(1)),
						\Lambda_{\infty}
					\bigr)
				\\ @VVV @VVV \\
					\alg{H}^{2}(\Hat{\alg{K}}_{\eta_{x}}, \Lambda(1))^{0}
				@> \sim >>
					\sheafext^{1}_{F^{\ind\rat}_{\pro\et}} \bigl(
						\alg{H}^{1}(\Hat{\alg{K}}_{\eta_{x}}, \Lambda(1)),
						\Lambda_{\infty}
					\bigr).
			\end{CD}
		\]
	Combining the above two diagrams, we obtain a commutative diagram
		\[
			\begin{CD}
					\alg{H}^{2}(\alg{X}, \mathfrak{T}(1))^{0}
				@> \sim >>
					\sheafext^{1}_{F^{\ind\rat}_{\pro\et}} \bigl(
						\alg{H}^{2}(\alg{X}, \mathfrak{T}(1)),
						\Lambda_{\infty}
					\bigr)
				\\ @VVV @VVV \\
					\alg{H}^{2}(\Hat{\alg{K}}_{\eta_{x}}, \Lambda(1))^{0}
				@> \sim >>
					\sheafext^{1}_{F^{\ind\rat}_{\pro\et}} \bigl(
						\alg{H}^{1}(\Hat{\alg{K}}_{\eta_{x}}, \Lambda(1)),
						\Lambda_{\infty}
					\bigr).
			\end{CD}
		\]
	Hence the image of $G^{0}$ in $\alg{H}^{2}(\Hat{\alg{K}}_{\eta_{x}}, \Lambda(1))^{0}$ is zero.
	The left vertical morphism on $F$-valued points factors as
		\[
				H^{2}(X, \mathfrak{T}(1))
			\to
				\Br(K)[p]
			\to
				\Br(\Hat{K}_{\eta})[p]
			\to
				\Br(\Hat{K}_{\eta_{x}})[p].
		\]
	By Proposition \ref{0067}, we know that the image of $G^{0}(F)$ in $\Br(\Hat{K}_{\eta})[p]$
	is contained in $H^{1}(Y_{\eta}, \Lambda)$.
	Thus we have a homomorphism
	$G^{0}(F) \to H^{1}(Y_{\eta}, \Lambda)$.
	This is compatible with base field extensions.
	Hence we have a homomorphism
	$G^{0}(F') \to H^{1}(\alg{Y}_{\eta}(F'), \Lambda)$
	for any algebraically closed field $F' \in F^{\perar}$
	functorial in $F'$.
	Hence we have a morphism
	$G^{0} \to \alg{H}^{1}(Y_{\eta}, \Lambda)$.
	It is a morphism from a connected group to a finite group,
	hence zero.
	Therefore the image of $G^{0}(F)$ in $\Br(\Hat{K}_{\eta})[p]$ is zero.
	By the purity for Brauer groups,
	we know that the image of $G^{0}(F)$ in $\Br(K)[p]$ is contained in $\Br(\mathfrak{X})[p]$.
	As in the final paragraph of \cite[Section 3]{Sai86},
	we have $\Br(\mathfrak{X}) \cong \Br(Y)$ by the same argument as the proof of
	\cite[Theorem (3.1)]{Gro68}.
	As $F = \closure{F}$, we have $\Br(Y) = 0$.
	Thus the image of $G^{0}(F)$ in $\Br(X)$ is zero.
	The result then follows.
\end{proof}

\begin{Prop}
	There exists a canonical isomorphism
		\[
				\alg{H}^{2}(\alg{X}, \Gm)[p^{\infty}]
			\cong
				\sheafext^{1}_{F^{\ind\rat}_{\pro\et}}(\alg{H}^{1}(\alg{X}, \Gm), \Lambda_{\infty}).
		\]
\end{Prop}

\begin{proof}
	Let $n \ge 1$.
	We have
		\[
				\alg{H}^{2}(\alg{X}, \mathfrak{T}_{n}(1))
			\isomto
				\sheafext^{-1}_{F^{\ind\rat}_{\pro\et}} \bigl(
					R \alg{\Gamma}(\alg{X}, \mathfrak{T}_{n}(1)),
					\Lambda_{\infty}
				\bigr)
		\]
	by duality.
	The right-hand side is isomorphic to
		\[
			\sheafext^{-1}_{F^{\ind\rat}_{\pro\et}} \bigl(
				\tau_{\ge 1} \tau_{\le 2}
				R \alg{\Gamma}(\alg{X}, \mathfrak{T}_{n}(1)),
				\Lambda_{\infty}
			\bigr).
		\]
	Let $G_{n}$ be the canonical mapping cone of the natural morphism
		\[
				\alg{A}^{\times} / \alg{A}^{\times p^{n}}[-1]
			\to
				\tau_{\ge 1} \tau_{\le 2}
				R \alg{\Gamma}(\alg{X}, \mathfrak{T}_{n}(1)).
		\]
	Since $\alg{A}^{\times} / \alg{A}^{\times p^{n}}$ is connected, we have
		\[
				\sheafext^{-1}_{F^{\ind\rat}_{\pro\et}} \bigl(
					\tau_{\ge 1} \tau_{\le 2}
					R \alg{\Gamma}(\alg{X}, \mathfrak{T}_{n}(1)),
					\Lambda_{\infty}
				\bigr)
			\cong
				\sheafext^{-1}_{F^{\ind\rat}_{\pro\et}}(G_{n}, \Lambda_{\infty}).
		\]
	We have a canonical distinguished triangle
		\[
				\alg{H}^{1}(\alg{X}, \Gm) \tensor^{L} \Lambda_{n}[-2]
			\to
				G_{n}
			\to
				\alg{H}^{2}(\alg{X}, \Gm)[p^{n}][-2].
		\]
	Applying $\sheafext^{\var}_{F^{\ind\rat}_{\pro\et}}(\var, \Lambda_{\infty})$,
	we obtain an exact sequence
		\begin{equation} \label{0069}
			\begin{split}
				&
						\sheafext^{1}_{F^{\ind\rat}_{\pro\et}} \bigl(
							\alg{H}^{2}(\alg{X}, \Gm)[p^{n}],
							\Lambda_{\infty}
						\bigr)
					\to
						\alg{H}^{2}(\alg{X}, \mathfrak{T}_{n}(1))
				\\
				&	\quad
					\to
						\sheafext^{1}_{F^{\ind\rat}_{\pro\et}} \bigl(
							\alg{H}^{1}(\alg{X}, \Gm) \tensor^{L} \Lambda_{n},
							\Lambda_{\infty}
						\bigr)
					\to
						0.
			\end{split}
		\end{equation}
	All the terms are in $\Alg_{u} / F$.
	By Proposition \ref{0068}, we have
		\[
				\dirlim_{n}
				\sheafext^{1}_{F^{\ind\rat}_{\pro\et}} \bigl(
					\alg{H}^{2}(\alg{X}, \Gm)[p^{n}],
					\Lambda_{\infty}
				\bigr)
			=
				0.
		\]
	Also
		\[
				\dirlim_{n}
				\sheafext^{1}_{F^{\ind\rat}_{\pro\et}} \bigl(
					\alg{H}^{1}(\alg{X}, \Gm) \tensor^{L} \Lambda_{n},
					\Lambda_{\infty}
				\bigr)
			\cong
				\sheafext^{1}_{F^{\ind\rat}_{\pro\et}} \bigl(
					\alg{H}^{1}(\alg{X}, \Gm),
					\Lambda_{\infty}
				\bigr)
		\]
	by \cite[Proposition (2.4.1) (d)]{Suz20}.
	Thus we get the result by taking the direct limit in $n$.
\end{proof}

Recall from \cite[Proposition (2.4.1)]{Suz20}
that for a group $G \in \Alg / F$,
the group $\pi_{0} \sheafext_{F^{\ind\rat}_{\pro\et}}^{1}(G, \Lambda_{\infty})$
is the Pontryagin dual of the $p$-adic Tate module of the semi-abelian part $G_{\sAb} \subset G$
and the group $\sheafext_{F^{\ind\rat}_{\pro\et}}^{1}(G, \Lambda_{\infty})^{0}$
is the Serre dual of the unipotent group $G^{0} / G_{\sAb}$.

Now we define
	\[
			\alg{Pic}_{X}
		=
			\alg{H}^{1}(\alg{X}, \Gm),
		\quad
			\alg{Br}_{X}
		=
			\alg{H}^{2}(\alg{X}, \Gm).
	\]
Theorem \ref{0083} follows from the above results.

\begin{Prop} \label{0110}
	The group $\alg{H}^{3}(\alg{X}, \Gm)[p^{\infty}]$ is connected.
	It is isomorphic to the Serre dual of $\alg{A}^{\times} / \Gm$.
	In particular, we have $\alg{H}^{3}(\alg{X}, \Gm)[p^{\infty}] \in \mathcal{W}_{F}$.
\end{Prop}

\begin{proof}
	The component group of $\alg{H}^{3}(\alg{X}, \Gm)[p^{\infty}]$ is Pontryagin dual to
	the inverse limit of $\alg{H}^{0}(\alg{X}, \mathfrak{T}_{n}(1)) \cong \alg{A}^{\times}[p^{n}]$.
	This inverse limit is zero
	since $A$ contains only finitely many $p$-power-th roots of unity.
	For any $n \ge 1$, we have an exact sequence
		\[
				0
			\to
				\alg{A}^{\times} / \alg{A}^{\times p^{n}}
			\to
				\alg{H}^{1}(\alg{X}, \mathfrak{T}_{n}(1))
			\to
				\alg{H}^{1}(\alg{X}, \Gm)[p^{n}]
			\to
				0.
		\]
	As $\alg{H}^{1}(\alg{X}, \Gm) \in \Alg / F$,
	the inverse limit of $\alg{H}^{1}(\alg{X}, \Gm)[p^{n}]$ in $n$ is profinite.
	Hence the Serre dual of $\alg{H}^{3}(\alg{X}, \Gm)[p^{\infty}]$ is
	the inverse limit of $\alg{A}^{\times} / \alg{A}^{\times p^{n}}$,
	which is $\alg{A}^{\times} / \Gm$.
\end{proof}

Applying $\invlim_{n}$, $\pi_{0}$ and $(\var)_{\tor}$ to \eqref{0069},
we obtain an isomorphism
	\[
			\pi_{0} \alg{H}^{1}(\alg{X}, \Gm)[p^{\infty}]
		\isomto
			\sheafhom(\pi_{0} \alg{H}^{1}(\alg{X}, \Gm)[p^{\infty}], \Lambda_{\infty}),
	\]
or a Pontryagin duality
	\begin{equation} \label{0120}
			\pi_{0} \alg{H}^{1}(\alg{X}, \Gm)[p^{\infty}]
		\leftrightarrow
			\pi_{0} \alg{H}^{1}(\alg{X}, \Gm)[p^{\infty}].
	\end{equation}
On the other hand,
if $L \times M \to \Z$ is a non-degenerate pairing of finitely generated free abelian groups,
then it induces a Pontryagin duality $\Hom(M, \Z) / L \leftrightarrow \Hom(L, \Z) / M$
on the discriminant groups.
Hence the non-degeneracy of the intersection pairing \eqref{0104} on $\mathfrak{X}$
induces a Pontryagin duality
	\begin{equation} \label{0121}
			\varalg{\delta}_{Y}[p^{\infty}]
		\leftrightarrow
			\varalg{\delta}_{Y}[p^{\infty}].
	\end{equation}
The following states that our duality is compatible with the intersection pairing:

\begin{Prop} \label{0111}
	The pairings \eqref{0120} and \eqref{0121} are compatible under the isomorphism
	$\pi_{0} \alg{H}^{1}(\alg{X}, \Gm) \cong \varalg{\delta}_{Y}$
	in Proposition \ref{0044}.
\end{Prop}

\begin{proof}
	We may assume that $F = \closure{F}$
	and it is enough to compare the two pairings on $F$-valued points.
	By construction, \eqref{0120} is given as follows.
	Let $n \ge 1$.
	The cup product and the trace map gives a pairing
		\[
					H^{1}(X, \mathfrak{T}_{n}(1))
				\times
					H^{2}(X, \mathfrak{T}_{n}(1))
			\to
					H^{3}(X, \mathfrak{T}_{n}(2))
			\to
				\Lambda_{n}.
		\]
	With the exact sequences
		\[
				0
			\to
				A^{\times} \tensor \Lambda_{n}
			\to
				H^{1}(X, \mathfrak{T}_{n}(1))
			\to
				\Pic(X)[p^{n}]
			\to
				0,
		\]
		\[
				0
			\to
				\Pic(X) \tensor \Lambda_{n}
			\to
				H^{2}(X, \mathfrak{T}_{n}(1))
			\to
				\Br(X)[p^{n}]
			\to
				0,
		\]
	since $\alg{A}^{\times} \tensor \Lambda_{n}$ is connected,
	this pairing induces a pairing
		\[
					\Pic(X)[p^{n}]
				\times
					\Pic(X) \tensor \Lambda_{n}
			\to
				\Lambda_{n}.
		\]
	Taking suitable direct and inverse limits,
	this factors as a pairing between component groups,
	which gives \eqref{0120}.
	We calculate it explicitly.
	Let $\ideal{p} \in P$ be arbitrary.
	Consider the natural maps
		\[
				H^{1}(X, \mathfrak{T}_{n}(1))
			\to
				H^{1}(K, \Lambda_{n}(1)),
		\]
		\[
				H^{1}(\Hat{K}_{\ideal{p}}, \Lambda_{n}(1))
			\to
				H^{2}_{\ideal{p}}(\Hat{A}_{\ideal{p}}, \mathfrak{T}_{n}(1))
			\to
				H^{2}(X, \mathfrak{T}_{n}(1)).
		\]
	We have a commutative diagram of pairings
		\[
			\begin{array}{ccccccc}
					H^{1}(X, \mathfrak{T}_{n}(1))
				&
					\times
				&
					H^{2}(X, \mathfrak{T}_{n}(1))
				&
					\longrightarrow
				&
					H^{3}(X, \mathfrak{T}_{n}(2))
				&
					\longrightarrow
				&
					\Lambda_{n}
				\\
					\bigg \downarrow
				&
				&
					\bigg \uparrow
				&
				&
					\bigg \uparrow
				&
				&
					\bigg \|
				\\
					H^{1}(K, \Lambda_{n}(1))
				&
					\times
				&
					H^{1}(\Hat{K}_{\ideal{p}}, \Lambda_{n}(1))
				&
					\longrightarrow
				&
					H^{2}(\Hat{K}_{\ideal{p}}, \Lambda_{n}(2))
				&
					\longrightarrow
				&
					\Lambda_{n},
			\end{array}
		\]
	where the lower pairing is the cup product pairing and the trace map.
	We also have a commutative diagram
		\[
			\begin{CD}
					H^{1}(X, \mathfrak{T}_{n}(1))
				@>>>
					H^{1}(K, \Lambda_{n}(1))
				\\
				@VVV @VVV
				\\
					\Pic(X)[p^{n}]
				@>>>
					(K^{\times} / A^{\times}) \tensor \Lambda_{n},
			\end{CD}
		\]
	where the lower horizontal map comes from the exact sequence
		\[
				0
			\to
				A^{\times}
			\to
				K^{\times}
			\to
				\bigoplus_{\ideal{q} \in P}
					\Z
			\to
				\Pic(X)
			\to
				0,
		\]
	and a commutative diagram
		\[
			\begin{CD}
					\Hat{K}_{\ideal{p}}^{\times} \tensor \Lambda_{n}
				@>>>
					\Pic(X) \tensor \Lambda_{n}
				\\
				@| @VVV
				\\
					H^{1}(\Hat{K}_{\ideal{p}}, \Lambda_{n}(1))
				@>>>
					H^{2}(X, \mathfrak{T}_{n}(1)),
			\end{CD}
		\]
	where the upper horizontal map factors as the valuation map
	$\Hat{K}_{\ideal{p}}^{\times} \tensor \Lambda_{n} \onto \Lambda_{n}$
	followed by the map $1 \mapsto [\ideal{p}]$.
	Hence we have a commutative diagram of pairings
		\[
			\begin{array}{ccccc}
					\Pic(X)[p^{n}]
				&
					\times
				&
					\Pic(X) \tensor \Lambda_{n}
				&
					\longrightarrow
				&
					\Lambda_{n}
				\\
					\bigg \downarrow
				&
				&
					\bigg \uparrow
				&
				&
					\bigg \|
				\\
					(K^{\times} / A^{\times}) \tensor \Lambda_{n}
				&
					\times
				&
					\Hat{K}_{\ideal{p}}^{\times} \tensor \Lambda_{n}
				&
					\longrightarrow
				&
					\Lambda_{n},
			\end{array}
		\]
	where the lower pairing is given by the tame symbol map
	$\Hat{K}_{\ideal{p}}^{\times} \times \Hat{K}_{\ideal{p}}^{\times} \to \kappa(\ideal{p})^{\times}$
	followed by the valuation map on $\kappa(\ideal{p})^{\times}$.
	Hence the value of the upper pairing at an element $(D, [\ideal{p}])$ is given as follows.
	Represent $p^{n} D$ by a rational function $f \in K^{\times}$.
	Take a prime element $\pi$ of the local ring of $A$ at $\ideal{p}$.
	Consider the tame symbol $\{f, \pi\}_{\ideal{p}} \in \kappa(\ideal{p})^{\times}$.
	Then $v_{\kappa(\ideal{p})}(\{f, \pi\}_{\ideal{p}}) / p^{n} \in \Lambda_{\infty}$
	is the value of the pairing \eqref{0120} at $(D, [\ideal{p}])$.
	
	The second pairing \eqref{0121} is given as follows.
	Let $D, \ideal{p}$ and $f$ as above.
	The divisor on $\mathfrak{X}$ defined by $f$ is the sum of $p^{n} D$
	and another divisor $E$ supported on $Y$.
	Consider the intersection number $E \cdot [\ideal{p}]$.
	Then $(E \cdot [\ideal{p}]) / p^{n} \in \Lambda_{\infty}$ is
	the value of the second pairing at $(D, [\ideal{p}])$.
	
	Therefore it is enough to show that
		\[
				v_{\kappa(\ideal{p})}(\{f, \pi\}_{\ideal{p}})
			=
				E \cdot [\ideal{p}].
		\]
	Suitably replacing the resolution $\mathfrak{X}$,
	we may assume that the divisor on $\mathfrak{X}$ defined by $f$, $\ideal{p}$ and $Y$
	is supported on a strict normal crossing divisor.
	Let $\eta$ be the generic point of the irreducible component of $Y$
	containing the specialization of $\ideal{p}$ in $Y$.
	Then we can see that these two numbers are both equal to $v_{\Hat{K}_{\eta}}(f)$.
\end{proof}


\section{The case of finite residue field}
\label{0119}

In this section, we prove Theorem \ref{0112}.
We deduce it from Theorems \ref{0077} and \ref{0088}
by applying $R \Gamma(F, \var)$.

Assume that $F = \F_{q}$.
Let $\ast_{\pro\et}$ be the pro-\'etale site of a point (\cite[Example 4.1.10]{BS15}).
Identify a profinite set with the affine $F$-scheme
having the same set of points.
This defines a morphism of sites
	\[
			\Spec F^{\ind\rat}_{\pro\et}
		\to
			\ast_{\pro\et}.
	\]
Its pushforward functor is simply denoted by $\Gamma(F, \var)$ by abuse of notation.
Let $\Fin_{p}$ be the category of finite abelian $p$-groups.
Let $\Ind \Pro \Fin_{p}$ be the ind-category of the pro-category of $\Fin_{p}$.

\begin{Def}
	Define $\mathcal{W}_{0} \subset \Ind \Pro \Fin_{p}$ to be the full subcategory consisting of
	objects $G$ that fit in an exact sequence
		\[
				0
			\to
				G'
			\to
				G
			\to
				G''
			\to
				0,
		\]
	where $G' = \invlim_{n \ge 1} G'_{n}$ is profinite (with $G'_{n} \in \Fin_{p}$)
	and $G'' = \dirlim_{n \ge 1} G''_{n}$ is indfinite (with $G''_{n} \in \Fin_{p}$)
	both indexed by $\N$.
\end{Def}

All objects of $\mathcal{W}_{0}$ are locally compact and, in particular, Hausdorff.

\begin{Prop} \label{0105}
	Let $G \in \mathcal{W}_{0}$.
	Then $R \sheafhom_{\ast_{\pro\et}}(G, \Lambda_{\infty})$ is concentrated in degree zero
	with cohomology given by the usual Pontryagin dual of $G$.
	In particular, $\sheafhom_{\ast_{\pro\et}}(G, \Lambda_{\infty}) \in \mathcal{W}_{0}$.
\end{Prop}

\begin{proof}
	This is obvious if $G$ is finite.
	Let $G = \invlim_{n \ge 1} G_{n}$ be profinite (with $G_{n} \in \Fin_{p}$).
	Then
		\[
				R \sheafhom_{\ast_{\pro\et}}(G, \Lambda_{\infty})
			\cong
				\dirlim_{n}
					R \sheafhom_{\ast_{\pro\et}}(G_{n}, \Lambda_{\infty})
		\]
	by the same argument as the proof of \cite[Theorem (2.3.1) (c)]{Suz20}
	(see also the proof of \cite[Proposition 8.10]{BS20}).
	Let $G = \dirlim_{n \ge 1} G_{n}$ be indfinite (with $G_{n} \in \Fin_{p}$).
	Then
		\[
				R \sheafhom_{\ast_{\pro\et}}(G, \Lambda_{\infty})
			\cong
				R \invlim_{n}
					R \sheafhom_{\ast_{\pro\et}}(G_{n}, \Lambda_{\infty})
		\]
	by \cite[Proposition (2.2.3)]{Suz20}.
	Combining all these, we get the result.
\end{proof}

Note that non-Hausdorff groups may have non-vanishing $\sheafext^{1}$:
we have
	\[
			R \sheafhom_{\ast_{\pro\et}}(\Lambda^{\N} / \Lambda^{\bigoplus \N}, \Lambda_{\infty})
		\cong
			\Lambda^{\N} / \Lambda^{\bigoplus \N}[-1].
	\]

\begin{Prop} \label{0106}
	Let $G \in D^{b}(F^{\ind\rat}_{\pro\et})$ be such that
	$H^{q} G \in \mathcal{W}_{F}$ for all $q$.
	Then $R \Gamma(F, G) \in D^{b}(\ast_{\pro\et})$
	and $H^{q}(F, G) \in \mathcal{W}_{0}$ for all $q$.
	The group $H^{q}(F, G)$ is profinite if $G$ is pro-algebraic
	and indfinite if $G$ is ind-algebraic.
	We have an exact sequence
		\[
				0
			\to
				H^{1}(F, \pi_{0}(H^{q - 1} G))
			\to
				H^{q}(F, G)
			\to
				\Gamma(F, H^{q} G)
			\to
				0
		\]
	for all $q$.
\end{Prop}

\begin{proof}
	Let $G \in \Alg_{u} / F$ be finite \'etale.
	Then $H^{q}(F, G)$ is finite for $q = 0, 1$ and zero otherwise.
	Let $G \in \Alg_{u} / F$ be connected.
	Then $H^{q}(F, G)$ is the finite group $G(F)$ for $q = 0$ and zero otherwise.
	Let $G = \invlim_{n} G_{n} \in \mathcal{W}_{F}$ be connected pro-algebraic
	(with $G_{n} \in \Alg_{u} / F$ connected
	such that $G_{n + 1} \onto G_{n}$ is surjective with connected kernel).
	Then
		\[
				R \Gamma(F, G)
			\cong
				R \invlim_{n}
					R \Gamma(F, G_{n})
			\cong
				R \invlim_{n}
					(G_{n}(F))
			\cong
				G(F)
		\]
	by \cite[Proposition (2.2.4) (b)]{Suz20},
	and $G(F)$ is profinite.
	Let $G = \dirlim_{n} G_{n} \in \mathcal{W}_{F}$ be connected ind-algebraic
	(with $G_{n} \in \Alg_{u} / F$ connected unipotent
	such that $G_{n} \into G_{n + 1}$ is injective).
	Then
		\[
				R \Gamma(F, G)
			\cong
				\dirlim_{n}
					R \Gamma(F, G_{n})
			\cong
				\dirlim_{n}
					(G_{n}(F))
			\cong
				G(F)
		\]
	by \cite[Proposition (2.2.4) (b)]{Suz20},
	and $G(F)$ is indfinite.
	Combining all these, we get the result.
\end{proof}

The composite functor
	\[
			D(\alg{U}_{S, \et})
		\stackrel{R \alg{\Gamma}(\alg{U}_{S}, \var)}{\longrightarrow}
			D(F^{\ind\rat}_{\pro\et})
		\stackrel{R \Gamma(F, \var)}{\longrightarrow}
			D(\ast_{\pro\et})
	\]
is denoted by $R \Gamma(U_{S}, \var)$ by abuse of notation,
with cohomologies $H^{q}(U_{S}, \var)$.
The composite functor
	\[
			D(\alg{U}_{S, \et})
		\stackrel{R \alg{\Gamma}_{c}(\alg{U}_{S}, \var)}{\longrightarrow}
			D(F^{\ind\rat}_{\pro\et})
		\stackrel{R \Gamma(F, \var)}{\longrightarrow}
			D(\ast_{\pro\et})
	\]
is denoted by $R \Gamma_{c}(U_{S}, \var)$ by abuse of notation,
with cohomologies $H^{q}_{c}(U_{S}, \var)$.

\begin{Prop}
	Let $n \ge 1$ and $q, r \in \Z$.
	Then the objects
	$H^{q}(U_{S}, \mathfrak{T}_{n}(r))$ and $H^{q}_{c}(U_{S}, \mathfrak{T}_{n}(r))$
	are in $\mathcal{W}_{0}$ for all $q$ and zero for all but finitely many $q$.
	We have a perfect pairing
		\[
					R \Gamma(U_{S}, \mathfrak{T}_{n}(r))
				\tensor^{L}
					R \Gamma_{c}(U_{S}, \mathfrak{T}_{n}(2 - r))
			\to
				\Lambda_{\infty}[-4]
		\]
	in $D(\ast_{\pro\et})$ and a Pontryagin duality
		\[
				H^{q}(U_{S}, \mathfrak{T}_{n}(r))
			\leftrightarrow
				H^{4 - q}_{c}(U_{S}, \mathfrak{T}_{n}(2 - r)).
		\]
\end{Prop}

\begin{proof}
	This follows from the perfect pairing
		\[
					R \alg{\Gamma}(\alg{U}_{S}, \mathfrak{T}_{n}(r))
				\tensor^{L}
					R \alg{\Gamma}_{c}(\alg{U}_{S}, \mathfrak{T}_{n}(2 - r))
			\to
				\Lambda_{\infty}[-3]
		\]
	and \cite[Proposition 4.2.1]{SuzCurve},
	Propositions \ref{0105} and \ref{0106}.
\end{proof}

This proves Theorem \ref{0112}.

Finally, we give a local version of this theorem.
Let $K$ and $k$ be as in Section \ref{0070}.
The composite functor
	\[
			D(\alg{K}_{\et})
		\stackrel{R \alg{\Gamma}(\alg{K}, \var)}{\longrightarrow}
			D(F^{\ind\rat}_{\pro\et})
		\stackrel{R \Gamma(F, \var)}{\longrightarrow}
			D(\ast_{\pro\et})
	\]
is denoted by $R \Gamma(K, \var)$ by abuse of notation,
with cohomologies $H^{q}(K, \var)$.
The composite functor
	\[
			D(\alg{O}_{K, \et})
		\stackrel{R \alg{\Gamma}(\alg{O}_{K}, \var)}{\longrightarrow}
			D(F^{\ind\rat}_{\pro\et})
		\stackrel{R \Gamma(F, \var)}{\longrightarrow}
			D(\ast_{\pro\et})
	\]
is denoted by $R \Gamma(\Order_{K}, \var)$ by abuse of notation,
with cohomologies $H^{q}(\Order_{K}, \var)$.
The composite functor
	\[
			D(\alg{O}_{K, \et})
		\stackrel{R \alg{\Gamma}_{c}(\alg{O}_{K}, \var)}{\longrightarrow}
			D(F^{\ind\rat}_{\pro\et})
		\stackrel{R \Gamma(F, \var)}{\longrightarrow}
			D(\ast_{\pro\et})
	\]
is denoted by $R \Gamma_{c}(\Order_{K}, \var)$ by abuse of notation,
with cohomologies $H^{q}_{c}(\Order_{K}, \var)$.

\begin{Prop}
	Let $n \ge 1$ and $q, r \in \Z$.
	Then we have
		\[
				H^{q}(K, \Lambda_{n}(r)),
				H^{q}(\Order_{K}, \mathfrak{T}_{n}(r)),
				H^{q}_{c}(\Order_{K}, \mathfrak{T}_{n}(r))
			\in
				\mathcal{W}_{0}.
		\]
	We have perfect pairings
		\[
					R \Gamma(K, \Lambda_{n}(r))
				\tensor^{L}
					R \Gamma(K, \Lambda_{n}(2 - r))
			\to
				\Lambda_{\infty}[-3],
		\]
		\[
					R \Gamma(\Order_{K}, \mathfrak{T}_{n}(r))
				\tensor^{L}
					R \Gamma_{c}(\Order_{K}, \mathfrak{T}_{n}(2 - r))
			\to
				\Lambda_{\infty}[-4]
		\]
	in $D(\ast_{\pro\et})$ and Pontryagin dualities
		\[
				H^{q}(K, \Lambda_{n}(r))
			\leftrightarrow
				H^{3 - q}(K, \Lambda_{n}(2 - r)),
		\]
		\[
				H^{q}(\Order_{K}, \mathfrak{T}_{n}(r))
			\leftrightarrow
				H^{4 - q}_{c}(\Order_{K}, \mathfrak{T}_{n}(2 - r)).
		\]
\end{Prop}

\begin{proof}
	This follows from the perfect pairings
		\[
					R \alg{\Gamma}(\alg{K}, \Lambda_{n}(r))
				\tensor^{L}
					R \alg{\Gamma}(\alg{K}, \Lambda_{n}(2 - r))
			\to
				\Lambda_{\infty}[-2],
		\]
		\[
					R \alg{\Gamma}(\alg{O}_{K}, \mathfrak{T}_{n}(r))
				\tensor^{L}
					R \alg{\Gamma}_{c}(\alg{O}_{K}, \mathfrak{T}_{n}(2 - r))
			\to
				\Lambda_{\infty}[-3]
		\]
	and \cite[Proposition 4.2.1]{SuzCurve},
	Propositions \ref{0105} and \ref{0106}.
\end{proof}


\end{document}